\documentclass[a4paper,11pt]{amsart}
\usepackage{amsmath,amsthm,amssymb,amsfonts,enumerate,color,hyperref,esint}

\newcommand{\R}{\mathbb{R}}
\newcommand{\Rn}{\mathbb{R}^n}

\newcommand{\F}{{F}}
\newcommand{\M}{\mathcal{M}}
\newcommand{\A}{\mathcal{A}}
\newcommand{\C}{\mathcal{C}}
\newcommand{\tr}{\operatorname{tr}}
\newcommand{\osc}{\operatorname{osc\,}}
\newcommand{\vep}{\varepsilon}
\newcommand{\supp}{\operatorname{supp}}

\newcommand{\USC}{\operatorname{USC}}
\newcommand{\LSC}{\operatorname{LSC}}
\newcommand{\limsupstar}{\operatorname{\textstyle\limsup^*}}

\newtheorem{thm}{Theorem}[section]
\newtheorem{prop}[thm]{Proposition}
\newtheorem{cor}[thm]{Corollary}
\newtheorem{lem}[thm]{Lemma}
\theoremstyle{definition}
\newtheorem{defn}[thm]{Definition}
\newtheorem{rem}[thm]{Remark}

\numberwithin{equation}{section}

\allowdisplaybreaks
 
\author[M.~Soria-Carro]{Mar\'ia Soria-Carro}
    \address{Department of Mathematics\\
    Rutgers University\\
 Hill Center - Busch Campus\\ 
 110 Frelinghuysen Road, Piscataway, NJ 08854, USA}
    \email{maria.soriacarro@rutgers.edu}
    
\author[P.~R.~Stinga]{Pablo Ra\'ul Stinga}
    \address{Department of Mathematics\\
    Iowa State University\\
    396 Carver Hall, Ames, IA 50011, USA}
    \email{stinga@iastate.edu}

\keywords{Transmission problems, fully nonlinear elliptic equations, viscosity solutions, regularity estimates}

\subjclass[2010]{Primary: 35B65, 35Q74. Secondary: 35J60, 74A50}

\thanks{Research partially supported by NSF grant 1500871 and Simons Foundation grant 580911.}
  
\begin{document}

%%%%%%%%%%%%%%%%%%%%%%%%%%%%%%%%%%%%%%%%%%%%%%
\title[Regularity of fully nonlinear transmission problems]{Regularity of viscosity solutions to \\ fully nonlinear elliptic transmission problems}
%%%%%%%%%%%%%%%%%%%%%%%%%%%%%%%%%%%%%%%%%%%%%%

%%%%%%%%%%%%%%%%%%%%%%%%%%%%%%%%%%%%%%%%%%%%%%
\begin{abstract}
We develop the regularity theory of viscosity solutions
to transmission problems for fully nonlinear second order uniformly elliptic equations.
Our results give a complete theory of existence, uniqueness, comparison principle, and regularity
of solutions to flat interface transmission problems; and the $C^{0,\alpha}$, $C^{1,\alpha}$ and $C^{2,\alpha}$
regularity of viscosity solutions up to the transmission surface for the case of curved interfaces.
\end{abstract}
%%%%%%%%%%%%%%%%%%%%%%%%%%%%%%%%%%%%%%%%%%%%%%

\maketitle

%%%%%%%%%%%%%%%%%%%%%%%%%%%%%%%%%%%%%%%%%%%%%%
\section{Introduction}
%%%%%%%%%%%%%%%%%%%%%%%%%%%%%%%%%%%%%%%%%%%%%%

Transmission problems describe phenomena in which a physical quantity changes behavior across some fixed interface. 
Their analysis started in the 1950s with the pioneering work of Picone in elasticity \cite{Picone} and further
subsequent contributions \cite{Campanato,Lions,Stampacchia}. 
In 1960, Schechter generalized the problem of transmission for elliptic equations with smooth coefficients and interfaces \cite{Schechter}.
Other variations include diffraction problems in the theory of discontinuous coefficients
\cite{Ladyzhenskaya-Uraltseva, Oleinik}. 
See \cite{Borsuk} for a detailed exposition on classical transmission problems.
These problems have a wide range of applications in different areas such as
electromagnetic processes, composite materials, vibrating folded membranes and climatology.
For other recent developments, see \cite{CSCS,Citti-Ferrari,Kriventsov,Li-Nirenberg,Li-Vogelius} and references therein.

We consider viscosity solutions to the following transmission problem for fully nonlinear elliptic second order equations in the unit ball $B_1$
of $\R^n$, $n\geq2$:
\begin{align} \label{eq:TP} 
\begin{cases}
\F^+(D^2 u) = f^+& \hbox{in}~\Omega^+=B_1\cap \{x_n> \psi(x')\}\\
\F^-(D^2 u) = f^-& \hbox{in}~\Omega^-=B_1\cap \{x_n< \psi(x')\}\\
 u_\nu^+ -   u_\nu^- = g & \hbox{on}~\Gamma=B_1 \cap \{x_n = \psi(x')\}.
\end{cases}
\end{align}
Here $u\in C(B_1)$ and $u^\pm=u|_{\overline{\Omega^\pm}}$.
The interface $\Gamma$ is given by the graph of a function $\psi: \R^{n-1} \to \R$
with unit normal vector $\nu$ pointing towards $\Omega^+$, and $u_\nu^\pm$ denote the derivatives of $u^\pm$
in the direction $\nu$. Let $\mathcal{S}^n$ be the set of symmetric matrices of size $n\times n$.
We assume that $F^\pm: \mathcal{S}^n \to \R$ are fully nonlinear uniformly elliptic operators with ellipticity constants $0<\lambda\leq\Lambda$,
that is, for every $M,N\in \mathcal{S}^n$ with $N\geq0$, we have
$$
\lambda \|N\| \leq F^\pm(M+N)-F^\pm(M)\leq \Lambda \|N\|.
$$ 
For simplicity, we assume $F^\pm(0)=0$. As usual, $\mathcal{E}(\lambda,\Lambda)$ denotes the class of all
such operators.

The third equation in \eqref{eq:TP} is the transmission condition. It describes
the interaction between $u^+$ and $u^-$ across the common interface $\Gamma$.
If $g$ is nontrivial, the transmission condition prescribes a jump discontinuity between the normal derivatives
of the solution on each side of $\Gamma$.
Hence, $u$ will not be differentiable across $\Gamma$. However, we show that $u^+$ and $u^-$ have better
regularity properties up to the interface.
Even though \eqref{eq:TP} looks like two Neumann problems attached at the interface through the transmission condition,
it is not possible to decouple the equations for $u^+$ and $u^-$. Thus,
the analysis of \eqref{eq:TP} presents challenging difficulties as we will need to deal with both sides of the interface simultaneously. 

We develop the complete theory of existence, uniqueness and regularity 
in the case when the interface is flat. Then we prove the following regularity of viscosity solutions for curved interfaces.
See the end of this section for notation.

\begin{thm}[$C^{0,\alpha}$ regularity] \label{thm:holder}
Let $\Gamma=B_1 \cap \{x_n = \psi(x')\}$. Let $u$ satisfy
$$
\begin{cases}
u \in {S}^*_{\lambda, \Lambda}(f^\pm) & \hbox{in}~\Omega^\pm\\
 u_{\nu}^+ -  u_{\nu}^- = g & \hbox{on}~\Gamma,
\end{cases}
$$
with $f^\pm \in C(\Omega^\pm)\cap L^\infty (B_1^\pm)$, $g \in L^\infty(\Gamma)$ and $\psi\in C^{1,\alpha}(\overline{B_1'})$,
for some $0<\alpha<1$.
Then $u\in C^{0,\alpha_1}(\overline{B_{1/2}})$ and 
$$
\|u\|_{C^{0,\alpha_1}(\overline{B_{1/2}})} \leq C\big ( \|u\|_{L^\infty(B_1)}+  \|g\|_{L^\infty(\Gamma)} + \|f^-\|_{L^n(\Omega^-)}+\|f^+\|_{L^n(\Omega^+)}\big),
$$
where $0<\alpha_1<1$ and $C>0$ depend only on $n$, $\lambda$, $\Lambda$, $\alpha$ and $\|\psi\|_{C^{1,\alpha}(\overline{B_1'})}$.
\end{thm}

\begin{thm}[$C^{1,\alpha}$ regularity] \label{thm:main1}
Let $0<\bar\alpha<1$, depending only on $n$, $\lambda$ and $\Lambda$ be as given below.
Fix any $0<\alpha<\bar\alpha$.
Let $\Gamma=B_1 \cap \{x_n = \psi(x')\}$, where $\psi \in C^{1,\alpha}(\overline{B_1'})$.
Assume that  $g\in C^{0,\alpha}(\Gamma)$ and $f^\pm$ satisfy 
$$
\bigg( \fint_{B_r(x_0) \cap \Omega^\pm} |f^\pm|^n\, dx\bigg)^{1/n} \leq C_{f^\pm} r^{\alpha-1}
$$ 
for all $r>0$ small and $x_0\in \Omega^\pm \cup \Gamma$, for some constants $C_{f^\pm}>0$.
There exists $0<\theta<1$, depending only on $n$, $\lambda$, $\Lambda$ and $\alpha$,
such that if 
\begin{equation}\label{eq:closeness2}
\sup_{M\in \mathcal{S}^n \setminus \{0\}} \frac{\|F^+(M)-F^-(M)\|}{\|M\|} \leq \theta
\end{equation}
then any bounded viscosity solution $u$ of \eqref{eq:TP} in $B_1$ satisfies
$u^\pm\in C^{1,\alpha}(\overline{\Omega^\pm_{1/2}})$, where $\Omega_{1/2}^\pm = \Omega^\pm \cap B_{1/2}$, and 
$$
\|u^\pm\|_{C^{1,\alpha}(\overline{\Omega^\pm_{1/2}})} \leq C\|\psi\|_{C^{1,\alpha}(\overline{B_1'})} \big( \|u\|_{L^\infty(B_1)} + \|g\|_{C^{0,\alpha}(\Gamma)}+ C_{f^-} + C_{f^+}\big)
$$
where $C>0$ depends only on $n$, $\lambda$, $\Lambda$ and $\alpha$. 
\end{thm}

Global $C^{2,\alpha}$ estimates can also be obtained under some constraints on the function $g$ and the interface $\Gamma$.
The key result for such estimates is the following pointwise boundary regularity result, where we require weaker assumptions.
Hence, we omit the statement of global estimates to provide a more complete description in the case of $C^{2,\alpha}$ interfaces.

\begin{thm}[$C^{2,\alpha}$ regularity]  \label{thm:reg2}
Let $\F$ be a uniformly elliptic concave operator and let $0<\bar{\bar{\alpha}}<1$, depending only on $n$, $\lambda$ and $\Lambda$ be
as given below. Fix any $0<\alpha<\bar{\bar\alpha}$.
Let $\Gamma=B_1 \cap \{x_n = \psi(x')\}$. Assume that $0\in \Gamma$, $\psi \in C^{2,\alpha}(0)$, $\psi\not\equiv0$ and
$g\in C^{1,\alpha}(0)$ are such that $|g(0)|\|D_{x'}^2 \psi(0)\|=0$,  $f^\pm$ are continuous at $0$, $f^+(0)=f^-(0)$ and 
$$
\Big( \fint_{B_r \cap \Omega^\pm} |f^\pm(x)-f^\pm(0)|^n\, dx\Big)^{1/n} \leq K_{f^\pm} r^{\alpha},
 \qquad \hbox{for all}~ r>0~\hbox{small}.
$$
Let $u$ be a bounded viscosity solution of \eqref{eq:TP} with $F\equiv F^+ \equiv F^-$.
Then $u^\pm \in C^{2,\alpha}(0)$, that is, there exist quadratic polynomials  $P^\pm(x)= \tfrac{1}{2}x^t A^\pm x  +  B^\pm \cdot x + c $ such that
$$
|u^\pm(x) - P^\pm(x)| \leq C |x|^{2+\alpha}, \qquad \hbox{for all}~x\in \Omega_{r_0}^\pm,
$$
 with $r_0=C_0 / \|\psi\|_{C^{2,\alpha}(0)}$, and $C_0>0$ depending only on $n, \lambda, \Lambda$ and $\alpha$. Moreover,
 $$
 \|A^\pm\|+|B^\pm| + |c| \leq  C_0 \|\psi\|_{C^{2,\alpha}(0)} \big( \|g\|_{C^{1,\alpha}(0)} +|f^-(0)| + |f^+(0)| + K_{f^-}+K_{f^+}\big).
 $$
\end{thm}

In \cite{DSFS}, De Silva, Ferrari and Salsa proved $C^{1,\alpha}$ regularity of viscosity solutions
to a problem similar to \eqref{eq:TP}, but where $\psi\equiv 0$
(flat interface), the transmission condition is replaced by $u_{x_n}^+- \mu u_{x_n}^-=0$, for some $\mu>0$,
and $f^\pm$ are Lipschitz continuous.
Flat transmission problems play a key role in the regularity theory of solutions to two-phase free boundary problems with distributed sources,
see \cite{DeSilva-Ferrari-Salsa} for a complete account of this theory.
The condition given in \eqref{eq:closeness2}, that also appeared in \cite{Pimentel-Swiech} in the context
of fully nonlinear free transmission problems,
may be understood as a closeness condition between the operators $F^+$ and $F^-$.
For example, the linear operators given by $F^\pm(M)=\tr(A^\pm M)$ for some $A^\pm \in \mathcal{S}^n$,
with $\lambda I \leq A^\pm \leq \Lambda I$ and $\|A^+-A^-\|\leq \theta$, satisfy \eqref{eq:closeness2}. 

To develop our theory, we first prove a new Alexandroff--Bakelman--Pucci (ABP) estimate for viscosity supersolutions of \eqref{eq:TP}.
Then we show the comparison principle and the existence, uniqueness and regularity of viscosity solutions
for flat interface transmission problems. Observe that existence is not included in \cite{DSFS} and it is still
an open problem for two-phase free boundary problems with sources \cite{DeSilva-Ferrari-Salsa}.
Furthermore, the transmission condition makes the proof of the comparison principle quite challenging. Using these new tools, we derive
regularity estimates up to curved interfaces $\Gamma$ through new stability results.
Here again, the transmission condition creates new difficulties for the construction of barriers.
It is worth pointing out that our proof for $C^{1,\alpha}$ interfaces is purely constructive. 

The paper is organized as follows.
In Section \ref{section:ABP}, we prove the ABP estimate (Theorem \ref{thm:ABP2}) and the maximum principle
(Corollary \ref{cor:MP}). For this, we construct a barrier using a Hopf lemma for fully nonlinear equations,
see Lemma \ref{lem:barrier}.
We show an oscillation decay result in Section~\ref{sec:holder}, which implies the $C^{0,\alpha}$ regularity of viscosity solutions across the interface $\Gamma$ (Theorem~\ref{thm:holder}).
In Section \ref{sect:FTP}, we study flat interface problems that will play a fundamental role in the regularity theory of our original problem. 
First, we consider a family of regularizations of $u$ in the $x'$-direction and present some of their properties.
Then, in Theorem \ref{thm:existflat}, we prove existence and uniqueness of viscosity solutions of \eqref{eq:TP}
with $\psi\equiv 0$ and prescribed boundary values on $\partial B_1$.
To this end, we prove the comparison principle (Theorem \ref{thm:comparison}) and construct the solution
through a delicate adaptation of Perron's method.
H\"{o}lder estimates up to the flat interface will follow by a perturbation argument that relies on the regularity
of the homogeneous problem that we also prove here.
In Section \ref{sec:approxlem}, we obtain several closedness, approximation and stability lemmas that will be useful in the last part of the paper.
Sections \ref{section:C1alpha} and \ref{section:C2alpha} are devoted to the proofs of Theorem~\ref{thm:main1} and Theorem~\ref{thm:reg2}, respectively.

\smallskip

\noindent\textbf{Notation.~}We will use the notation
and results from Caffarelli--Cabr\'e \cite{CC} and the user's guide \cite{CIL}.
Pointwise boundary regularity estimates were proved in 
\cite{LZ2,Silvestre-Sirakov} for the Dirichlet problem, in \cite{MS} for
Neumann problems on flat boundaries, and in \cite{LiZ} for oblique boundary conditions
on curved boundaries.
\begin{itemize}
\item[$-$] For $x\in\R^n$, we write $x=(x',x_n)$, where $x'\in\R^{n-1}$ and $x_n\in\R$;
$\nabla'$ denotes the gradient in the variables $x'$; $D_{x'}^2$ denotes the Hessian in the variables $x'$;
$B_r'$ denotes the ball in $\R^{n-1}$ of radius $r>0$ centered at the origin; $T=B_1\cap\{x_n=0\}$.
\item[$-$] Given $r>0$, we write $\Omega_r^\pm=\Omega^\pm \cap B_r$ and $\Gamma_r = \Gamma \cap B_r$.
\item [$-$] Given a function $v$, we denote $v_-=-\min\{0,v\}$ and $v_+=\max\{0,v\}$; and
$v^- = v|_{\overline{\Omega^-}}$ and $v^+ = v|_{\overline{\Omega^+}}$.
\item[$-$] $\bar\alpha=\bar\alpha(n,\lambda,\Lambda)\in(0,1)$ denotes the exponent from the interior $C^{1,\bar{\alpha}}$ regularity of viscosity solutions to $F(D^2 u)=0$, where $F\in \mathcal{E}(\lambda,\Lambda)$.
\item[$-$] $\bar{\bar\alpha}=\bar{\bar\alpha}(n,\lambda,\Lambda)\in(0,1)$ denotes the exponent from the interior $C^{2,\bar{\bar{\alpha}}}$ regularity of viscosity solutions to $F(D^2 u)=0$, where $F\in \mathcal{E}(\lambda,\Lambda)$ is a concave operator.
\item[$-$] $\USC({B_1})$ and $\LSC( {B_1})$ are the sets of upper semicontinuous and lower semicontinuous functions on ${B_1}$, respectively.
\item[$-$] Solutions to \eqref{eq:TP} are understood in the viscosity sense.
In particular, $u\in\USC({B_1})$ is a subsolution to the transmission condition if whenever
$\varphi \in C^1 ( B_\delta(x_0) \cap \overline{\Omega^-}) \cap  C^1 ( B_\delta(x_0) \cap \overline{\Omega^+}) $
touches $u$ from above at $x_0\in\Gamma$ then $\varphi_\nu^+(x_0) -  \varphi_\nu^-(x_0) \geq  \  g(x_0)$.
\item[$-$] The Pucci extremal operators are denoted by $\M_{\lambda,\Lambda}^-$ and $\M_{\lambda,\Lambda}^+$;
and $\underline{S}_{\lambda,\Lambda}(f^\pm)$ is the class of functions $u\in\USC(B_1)$ such that 
$\M_{\lambda,\Lambda}^+(D^2 u) \geq f^\pm$ in $\Omega^\pm$.
The classes $\overline{S}_{\lambda,\Lambda}(f^\pm)$, 
$S_{\lambda,\Lambda}(f^\pm)$ and ${S}^*_{\lambda,\Lambda}(f^\pm)$ are defined as usual.
\end{itemize}

%%%%%%%%%%%%%%%%%%%%%%%%%%%%%%%%%%%%%%%%%%%%%
\section{ABP estimate}\label{section:ABP}
%%%%%%%%%%%%%%%%%%%%%%%%%%%%%%%%%%%%%%%%%%%%%

The fundamental tool to build up the regularity theory of viscosity solutions to the transmission problem
\eqref{eq:TP} will be a new ABP estimate.

\begin{thm}[ABP estimate] \label{thm:ABP2}
Let $\Gamma=B_1\cap \{x_n=\psi(x')\}$. Let $u$ satisfy
$$\begin{cases}
u \in \overline{S}_{\lambda, \Lambda}(f^\pm) & \hbox{in}~\Omega^\pm \\
 u_{\nu}^+ -  u_{\nu}^- \leq g & \hbox{on}~\Gamma,
\end{cases}$$
with $f^\pm\in C(\Omega^\pm)\cap L^\infty(B_1)$, $g\in L^\infty(\Gamma)$, and $\psi\in C^{1,\alpha}(\overline{B_1'})$,
for some $0<\alpha<1$. Then
$$
 \sup_{B_1} u_- \leq \sup_{\partial B_1} u_- +  C \big( \max_{\Gamma} g_+ +  \|f_+^-\|_{L^n(\Omega^-)}+ \|f_+^+\|_{L^n(\Omega^+)}\big)
$$ 
where $C>0$ depends only on $n$, $\lambda$, $\Lambda$, $\alpha$ and $[\psi]_{C^{1,\alpha}}$.
\end{thm}

To overcome the singularities given by the transmission condition, we construct a suitable barrier that allows us to avoid the interface.
To this end, we use the following Hopf's lemma, see \cite[Theorem 1.10]{LXZ}.

\begin{lem}[Hopf's lemma] \label{lem:hopf}
Suppose that $\Omega$ is a domain such that $\partial\Omega\in C^{1,\alpha}(0)$, for some $0<\alpha<1$.
Let $w\in  \overline{S}_{\lambda,\Lambda}(0)$ in $\Omega \cap B_1$, with $w(0)=0$ and $w\geq 0$ in $\Omega\cap B_1$.
Then, for any $l \in \R^n$ with $|l|=1$ and $l_n=l\cdot e_n>0$, we have that
$$
w(rl)\geq c l_n w(e_n/2) r,
$$
for all $0<r<r_1$, where $c>0$ and $r_1$ depend only on $n$, $\lambda$, $\Lambda$, $\alpha$ and $[\partial \Omega]_{C^{1,\alpha}(0)}$.
\end{lem}

\begin{lem}[Barrier] \label{lem:barrier}
Let $\Omega$ be a $C^{1,\alpha}$ domain, for some $0<\alpha<1$. Assume that $0\in \partial \Omega$.
Let $w$ be the viscosity solution to
\begin{align}
\begin{cases} \label{eq:barrier}
\M^-_{\lambda,\Lambda}(D^2 w)=0 & \hbox{in}~\Omega\cap B_2\\
 w = \phi & \hbox{on}~\Gamma= B_2  \cap \partial\Omega\\
 w  = 1 & \hbox{on}~ \partial B_2 \cap \Omega,
\end{cases}
\end{align}
where $\phi\in C(\Gamma)$ satisfies $0\leq\phi\leq 1$, $\phi\equiv 0$ on $\Gamma \cap B_1$, and $\phi\equiv 1$ on $\Gamma \cap B_{3/2}^c$.
Then $w$ is a classical solution in $\Omega\cap B_1$ with $0\leq w\leq 1$ and 
\begin{equation}\label{eq:normalder}
w_\nu \geq c_0 > 0 \quad \hbox{on}~\Gamma\cap B_1
\end{equation}
where $\nu$ is the unit normal vector interior to $\Omega$
and $c_0>0$ depends on $n$, $\lambda$, $\Lambda$, $\alpha$ and $[\Gamma]_{C^{1,\alpha}}$.
\end{lem}

\begin{proof}
The existence and uniqueness of a viscosity solution $w$ to \eqref{eq:barrier} follows from the classical theory of fully nonlinear elliptic equations. Moreover, since $\M^-_{\lambda,\Lambda}$ is a concave operator, $w\in C^{2,\bar{\bar{\alpha}}}(\overline{\Omega}_0)$, for any $\Omega_0 \subset\subset \Omega\cap B_2$. Furthermore, by boundary $C^{1,\alpha}$ estimates, $w \in C^{1,\alpha}(\Gamma\cap B_1)$. Hence,
$w$ is a classical solution in $\Omega\cap B_1$.
Applying the classical ABP to $w$ and $1-w$, it is easy to see that $0\leq w \leq 1$.
We are left to show \eqref{eq:normalder}. 

Fix $x_0 \in \Gamma\cap B_1$. After translation, rotation and rescaling, we can assume that $x_0=0$ and
$$ \Omega\cap B_1 = B_1 \cap \{x_n>\psi(x')\},$$ 
for some $\psi \in C^{1,\alpha}(\overline{B_1'})$, with $\nabla'\psi(0)=0$ and $[\psi]_{C^{1,\alpha}(0)}\leq 1/4$.
Then $w$ satisfies the assumptions from Lemma~\ref{lem:hopf}. Hence, by setting $l=\nu(0)=e_n$, we get
$$
w(r\nu(0)) \geq c w(e_n/2) r \quad  \hbox{for all}~0<r<r_1,
$$
where $c$ and $r_1$ depend only on $n,\lambda,\Lambda$, and $[\Gamma]_{C^{1,\alpha}}$.
Since $w$ is differentiable at $0$,  we see that
$w_\nu(0) \geq c  w(e_n/2)$. Consider $\tilde{w}(x)= w(x) - \tfrac{1}{2} x_n  + \tfrac{1}{8}$ in $D=\{x_n \geq 1/4 \}\cap B_2$. 
Then $\tilde w \in \overline{S}(0)$ in $D$ and $\tilde w \geq 0$ on $\partial D$. By the maximum principle, we have $\tilde{w}\geq 0 $ in $D$. In particular, $w(e_n/2)\geq 1/8$, and thus, we conclude that
$w_\nu(0) \geq c_0>0$, with $c_0= c/8$.
\end{proof}

\begin{proof}[Proof of Theorem~\ref{thm:ABP2}]
Let $w\in C(B_2)$ be the viscosity solution to
\begin{align*}  
\begin{cases} 
\M^-_{\lambda,\Lambda}(D^2 w)=0 & \hbox{in}~B_2 \setminus \{x_n = \psi(x')\}\\
 w  = \phi & \hbox{on}~\tilde \Gamma= B_2 \cap \{x_n = \psi(x')\}\\
 w = 1 & \hbox{on}~\partial B_2,
\end{cases}
\end{align*}
where $\phi\in C(\tilde \Gamma)$ satisfies $0\leq\phi\leq 1$, $\phi\equiv 0$ on $\Gamma = \tilde \Gamma \cap B_1$, and $\phi\equiv 1$ on $\tilde \Gamma \cap B_{3/2}^c$.
Let $w^+=w$ on ${B_2}\cap \{x_n \geq \psi(x')\}$ and $w^-=w$ on ${B_2}\cap \{x_n \leq \psi(x')\}$. By Lemma~\ref{lem:barrier},
 $0\leq w\leq 1$ and
$$
w_\nu^+ \geq c^+>0, \quad w_\nu^-\leq - c^{-}<0 \qquad \hbox{on}~\Gamma,
$$
where $\nu$ is the interior normal to $\Omega^+$, and $c^+,c^-$ depend only on $n$, $\lambda$, $\Lambda$, $\alpha$ and $[\psi]_{C^{1,\alpha}}$.
Fix $\vep>0$ small and consider in $B_1$ the function
$$
v = u- \frac{1}{c_0} \Big(\max_{\Gamma} g_++\vep\Big) w,
$$
where $c_0=c^++c^-$.
Then $v\in \overline{S}_{\lambda,\Lambda}(f^\pm)$ in $\Omega^\pm$. 
Moreover,
$$
v^+_{\nu} -  v^-_{\nu} \leq g -  \frac{1}{c_0} \Big(\max_{\Gamma} g_+ +\vep\Big)\big(w_\nu^+-w_\nu^-\big)\leq g_+- \max_{\Gamma} g_+-\vep  \leq - \vep
$$
on $\Gamma$ in the viscosity sense.
Without loss of generality, we may assume that $v\geq 0$ on $\partial B_1$. Otherwise, we  consider $v - \inf_{\partial B_1} v$.
Assume that $v_- \not\equiv 0$, and 
let $\mathcal{C}_v$ be the convex envelope of $-v_-$ in $B_2$, where we have extended $v$ by zero outside of $B_1$.  
Clearly, by definition of $\C_v$, we have that $\partial B_1 \cap \{v=\C_v\} = \varnothing$. We claim that $\Gamma \cap \{v=\C_v\} = \varnothing$. Indeed, if $A\cdot x + b$ touches $v$ from below at $x_0 \in \Gamma$, for some $A\in \Rn$ and $b\in \R$, then
$$
-\vep  \geq A \cdot \nu(x_0) - A \cdot \nu(x_0)  = 0,
$$
which is a contradiction. Moreover, there exists $\delta>0$ such that for any $x_0\in {\Gamma}$, we have $B_\delta(x_0)\cap{ \{v=\C_v\}}=\varnothing$.
If not, for any $k\geq 1$, there exist  $x_k \in \Gamma$ such that there is some $y_k\in B_{1/k}(x_k)\cap {\{v=\C_v\}}$. 
Then, up to a subsequence, it follows that $x_k,y_k\to y_0$ for some $y_0 \in \overline{\Gamma\cap \{v=\C_v\}}$, which is a contradiction.

Next, we show that $\C_v\in C^{1,1}(\overline{B_1})$. It is enough to see that there are $K>0$ and $0<r\leq 1$ such that for any $x_0 \in \overline{B_1} \cap \{v=\C_v\}$ there exists a convex paraboloid of opening $K$ that touches $\C_v$ by above at $x_0$ in $B_r(x_0)$.  Indeed, fix $x_0 \in \overline{B_1} \cap \{v=\C_v\}$. Since $x_0\notin \partial B_1 \cup \Gamma$, we may assume that $x_0\in \Omega^+\cap \{v=\C_v\}$. Furthermore, $B_{\delta}(x_0) \subset \Omega^+$, for $\delta$ small enough. 
Let $l$ be a supporting plane of $\C_v$ at $x_0$. Then $0\leq \C_v - l \leq -v_- - l$ in $B_\delta(x_0)$ and $\C_v(x_0) - l(x_0) = -v_-(x_0)-l(x_0)=0$. We know that $-v_--l\in \overline{S}_{\lambda,\Lambda}(f^+)$.
Applying \cite[Lemma 3.3]{CC} to $-v_--l$ in $B_\delta(x_0)$ and $\varphi=\C_v-l$, we get
\begin{equation} \label{eq:estimateenvelope}
\C_v(x) \leq l(x) + C^+ \Big(\sup_{B_\delta(x_0)}f^+_+\Big)|x-x_0|^2 \quad \hbox{for all}~x\in B_{\delta \gamma^+}(x_0),
\end{equation}
where $\gamma^+<1$ and $C^+$ are universal constants. If $x_0\in \Omega^-\cap \{v=\C_v\}$, the proof is analogous.
Hence, we take $K=2\max\{C^+\|f^+_+\|_{L^\infty(\Omega^+)}), C^-\|f^-_+\|_{L^\infty(\Omega^-)}\}$ and $r =\delta \min\{\gamma^+,\gamma^-\}$. 

Next, there exists a set $E\subset B_1$ such that $|B_1\setminus E|=0$ and $\C_v$ is second order differentiable at any $x\in E$. Moreover, we have that
\begin{equation}\label{eq:conditionv}
\sup_{B_1} v_- \leq C \Big( \int_{ E\cap \{v=\C_v\}} \det D^2 \C_v(x)\, dx\Big)^{1/n},
\end{equation}
where $C>0$ is a constant depending only on $n$.
Since $f^+ \in C(\Omega^+)$, by letting $\delta \to 0$ in \eqref{eq:estimateenvelope}, we see that
$\det D^2 \C_v(x_0) \leq C f_+^+(x_0)^n$ for a.e. $x_0\in B_1^+ \cap \{v=\C_v\}$,
and thus,
\begin{align*}
\int_{E\cap \{v=\C_v\}} \det D^2 \C_v(x)\,dx 
&\leq   \int_{ \Omega^- \cap \{v=\C_v\}} f_+^-(x)^n\,dx
+\int_{ \Omega^+ \cap \{v=\C_v\}} f_+^+(x)^n\,dx.
\end{align*}
Therefore,
$$
\sup_{B_1} v_- \leq  \sup_{\partial B_1} v_- + C \Big(\|f_+^-\|_{L^n( \Omega^- \cap \{v=\C_v\})}+\|f_+^+\|_{L^n( \Omega^+ \cap \{v=\C_v\})}\Big).
$$
From the definition of $v$ and by letting $\vep\to0$, we get the conclusion.
\end{proof}

\begin{rem}[Refined ABP estimate for flat interfaces] \label{refinedABP}
In the previous proof, it is not clear how to relate the contact sets $\{v=\C_v\}$ and $\{u=\C_u\}$
given that the barrier $w$ is not explicit.
In the flat case $\Gamma=B_1\cap \{x_n=0\}$, the ABP estimate from Theorem \ref{thm:ABP2} can be refined.
Indeed, take $w(x)= \max_\Gamma g_+ |x_n|$. If $v=u-w$, then $\{v=\C_v\}=\{u=\C_{v}+w\}\subseteq \{u=\C_u\}$. Therefore, in the flat case,
$$
 \sup_{B_1} u_- \leq \sup_{\partial B_1} u_- +  C \big( \max_{\Gamma} g_+ +  \|f_+^-\|_{L^n(\Omega^-\cap \{u=\C_u\})}+ \|f_+^+\|_{L^n(\Omega^+\cap\{u=\C_u\})} \big). 
 $$
\end{rem}

\begin{cor}[Maximum principle]\label{cor:MP}
Let $u$ satisfy
$$
\begin{cases}
u \in \overline{S}_{\lambda,\Lambda}(0) & \hbox{in}~\Omega^\pm\\ 
u_\nu^+-u_\nu^- \leq 0 & \hbox{on}~\Gamma.
\end{cases}
$$
If $u\geq 0$ on $\partial B_1$, then $u\geq 0$ in $B_1$.
\end{cor}

%%%%%%%%%%%%%%%%%%%%%%%%%%%%%%%%
\section{H\"{o}lder continuity} \label{sec:holder}
%%%%%%%%%%%%%%%%%%%%%%%%%%%%%%%%

Theorem~\ref{thm:holder} is a consequence of the following oscillation decay.

\begin{thm}[Oscillation decay]\label{lem:oscillation2}
Let $\Gamma=B_1\cap \{x_n=\psi(x')\}$. Let $u$ satisfy
$$
\begin{cases}
u \in {S}^*_{\lambda, \Lambda}(f^\pm) & \hbox{in}~\Omega^\pm\\
 u_{\nu}^+ -  u_{\nu}^- = g & \hbox{on}~\Gamma,
\end{cases}
$$
with $f^\pm \in C(\Omega^\pm)\cap L^\infty (\Omega^\pm)$, $g \in L^\infty(\Gamma)$ and $\psi\in C^{1,\alpha}(\overline{B_1'})$,
for some $0<\alpha<1$. Then
$$
\underset{B_{1/3}}{\osc} u \leq \mu\, \underset{B_{1}}{\osc} u + C\big( \|g\|_{L^\infty(\Gamma)} + \|f^-\|_{L^n(\Omega^-)}+\|f^+\|_{L^n(\Omega^+)}\big)
$$
where $0<\mu<1$ and $C>0$ depend only on $n$, $\lambda$, $\Lambda$, $\alpha$ and $\|\psi\|_{C^{1,\alpha}(\overline{B_1'})}$.
\end{thm}

The oscillation decay will be a consequence of the following Harnack inequality.

\begin{lem}[Harnack inequality]\label{lem:holder2}
Let $u$, $f^\pm$, $g$ and $\psi$ be as in Theorem~\ref{lem:oscillation2}. 
Assume further that $\|u\|_{L^\infty(B_1)}\leq 1$,
$u(\bar{x})\geq0$ with $\bar{x}=\tfrac{1}{5}e_n$,
and $B_{1/20}(\bar{x}) \subset \Omega^+$.
There exist $0<\vep_0,c<1$ depending on $n$, $\lambda$, $\Lambda$ and $[\psi]_{C^{1,\alpha}}$ such that if 
$\|g\|_{L^\infty(\Gamma)}  + \|f^-\|_{L^n(\Omega^-)}+\|f^+\|_{L^n(\Omega^+)}\leq \vep_0$, then 
$$\inf_{B_{1/3}} u\geq -1+c.$$
\end{lem}
	
\begin{proof}
Since $u+1\geq0$, by the interior Krylov--Safonov Harnack inequality in $B_{1/20}(\bar{x})$,
$$
\sup_{B_{1/20}(\bar{x})} (u+1) \leq C\Big (\inf_{B_{1/20}(\bar{x})} (u+1) + \|f^+\|_{L^n(\Omega^+)}\Big),
$$
where $C>0$ depends only on $n$, $\lambda$ and $\Lambda$. Then
$$
1 \leq u(\bar{x})+1 \leq \sup_{B_{1/20}(\bar{x})} (u+1) \leq C\big(u(x)+1+\vep_0\big),
$$
for any $x\in B_{1/20}(\bar{x})$, and thus, 
\begin{equation}\label{eq:bound}
u\geq -1 + \tilde{c} \quad \hbox{in}~B_{1/20}(\bar{x}),
\end{equation}
with $\tilde{c} = 1/C-\vep_0$ and $\vep_0 <  1/C $.
For $x\in D= B_{3/4}(\bar{x})\setminus B_{1/20}(\bar{x})$, we define  
$$
v(x) = \eta \phi (r) + \frac{\vep_0}{c_0} w(x), \quad \phi(r)=r^{-\gamma}-(2/3)^{-\gamma}, \quad r=|x-\bar{x}|,
$$
where $w$ and $c_0$ are as in the proof of Theorem~\ref{thm:ABP2}, and
 $\gamma > \max \big\{0, \frac{\Lambda}{\lambda}(n-1)-1\big\}$ and $\eta>0$ to be chosen later. 
For any $x \in D^\pm= \Omega^\pm\cap D$, we get
\begin{align*}
\M^-_{\lambda,\Lambda}(D^2 v^\pm(x)) & \geq \eta\M^-_{\lambda,\Lambda}(D^2 \phi(x)) + \frac{\vep_0}{c_0}\M^-_{\lambda,\Lambda}(D^2 w^\pm(x))\\
&= \eta \gamma r^{-\gamma-2} \big( \lambda(\gamma+1) -\Lambda(n-1)\big)>0
\end{align*}
by the choice of $\gamma$. For $x\in \Gamma \cap D$, it follows that
$$
v_{\nu}^+(x)-v_{\nu}^-(x) = \frac{\vep_0}{c_0}\big(w_\nu^+(x)-w_\nu^-(x)\big)\geq 2 \vep_0  > \|g\|_{L^\infty(\Gamma)} \geq g(x).
$$ 
We will choose $\eta$ and $\vep_0$  so that $v \leq \tilde c$ on $\partial B_{1/20}(\bar{x})$ and 
$v\leq 0$ on $\partial B_{3/4}(\bar{x}).$
Note that $\phi(r)\geq 0$ if $0< r \leq 2/3$ and $\phi(r) \leq 0$ if $ r\geq 2/3$. First, choose $\eta$ such that
$
\eta \leq \frac{\tilde c}{2\phi(1/20)}.
$
Then choose $\vep_0$ such that
$
\vep_0 \leq \frac{1}{c_0}\min \big\{ \tilde c/2, -\eta\phi(3/4) \big\}.
$
By \eqref{eq:bound}, we obtain that
$
v\leq u+1 \ \hbox{on}~\partial D.
$

Since $u+1\in \overline{S}_{\lambda,\Lambda}(|f^\pm|)$ in $D^\pm$, $v^\pm\in C^{2}(D^\pm)$, and  $\M^-_{\lambda,\Lambda}(D^2v^\pm) \geq 0$ in $D^\pm$, we have that $u+1-v \in \overline{S}_{\lambda,\Lambda}(|f^\pm|)$ in $D^\pm$. Also,
$$
(u+1-v)_{\nu}^+-(u+1-v)_{\nu}^- \leq g -(v_{\nu}^+-v_{\nu}^-)  \leq g -g =0 \quad \hbox{on}~\Gamma\cap D
$$ 
in the viscosity sense.
Hence, applying Theorem~\ref{thm:ABP2} to $u+1-v$ in $D$, with $g \equiv0$, we get
\begin{align*}
 \sup_{D} (u+1-v)_- &\leq \sup_{\partial D} (u+1-v)_- + C\big(  \|f^-\|_{L^n(\Omega^-)}+ \|f^+\|_{L^n(\Omega^+)}\big) \leq C \vep_0,
\end{align*}
where $C$ depends only on $n$, $\lambda$, $\Lambda$, $\alpha$ and $[\psi]_{C^{1,\alpha}}$.
Therefore, $u \geq -1 + v  - C\vep_0$ in $D$.
Moreover, for any $x\in B_{1/3}(0)\setminus B_{1/20}(\bar{x})$, we have that
$
v(x)\geq \eta \phi(23/60)=c_1,  
$
and $c_1>0$ depends only on $n$, $\lambda$, and $\Lambda$.
Choosing $\vep_0$ such that $\vep_0 \leq \frac{c_1}{2C}$, we get
$u\geq -1+ \tfrac{c_1}{2}$ in $B_{1/3}(0)\setminus B_{1/20}(\bar{x})$. Therefore,
by choosing $c=\min\{\tilde c,\tfrac{c_1}{2}\}$, we get
$$
\inf_{B_{1/3}} u \geq -1 + c.
$$
\end{proof}

\begin{proof}[Proof of Theorem~\ref{lem:oscillation2}]
By choosing an appropriate system of coordinates, we assume that
$$
\psi(0)=0, \quad \nabla'\psi(0)=0, \quad \text{and}\quad |\psi(x')|\leq  |x'|.
$$
Then $B_{1/20}(\bar{x}) \subset \Omega^+$, with $\bar{x}=\tfrac{1}{5}e_n$. Let $M=\|g\|_{L^\infty(\Gamma)}  + \|f^-\|_{L^n(\Omega^-)}+\|f^+\|_{L^n(\Omega^+)}$ and let $\vep_0$ be as in Lemma~\ref{lem:holder2}. Consider
$$
\tilde u = \, \frac{2u-(\inf_{B_1} u + \sup_{B_1} u\big)}{\osc_{B_1} u + 2M/\vep_0} \in {S}^*_{\lambda,\Lambda}(\tilde f^\pm),
$$
with $\tilde f^\pm = 2f^\pm (\osc_{B_1} u + 2M/\vep_0)^{-1}$. Also,
$(\tilde u^+)_\nu - (\tilde u^-)_\nu \leq \tilde g$  on  $\Gamma$
in the viscosity sense, with $\tilde g=2g(\osc_{B_1} u + 2M/\vep_0)^{-1}$. Note that $\|\tilde u\|_{L^\infty(B_1)}\leq 1$, and
$$
\max_\Gamma \tilde g + \|\tilde f^-\|_{L^n(\Omega^-)}+\|\tilde f^+\|_{L^n(\Omega^+)}\leq \vep_0.
$$
If $\tilde{u}(\bar{x})\geq 0$ then, by Lemma~\ref{lem:holder2}, it follows that $\inf_{B_{1/3}} \tilde{u}\geq -1+c$. Otherwise, $\tilde u(\bar{x})<0$, and applying the lemma to $-\tilde u$, we see that $\sup_{B_{1/3}} \tilde{u}\leq 1-c$.
In both cases, we get 
$$
\underset{B_{1/3}}{\osc} \tilde u = \sup_{B_{1/3}} \tilde u - \inf_{B_{1/3}} \tilde u = \frac{2\osc_{B_{1/3}} u}{\osc_{B_1} u + 2M/\vep_0}\leq 2-c.
$$
Therefore,
$$\underset{B_{1/3}}{\osc} u \leq \mu\, \underset{B_{1}}{\osc} u + C\big( \|g\|_{L^\infty(\Gamma)}  + \|f^-\|_{L^n(\Omega^-)}+\|f^+\|_{L^n(\Omega^+)}\big),$$
with $\mu=\tfrac{2-c}{2}<1$ and $C=\tfrac{2-c}{\vep_0}$.
\end{proof}

When $f^\pm \equiv 0$ and $g$ has compact support on $\Gamma$, we obtain the following global H\"{o}lder regularity result.
The proof is omitted because it follows similar lines as the proof of \cite[Proposition~4.13]{CC}.
Indeed, the key ingredients are the maximum principle (Corollary \ref{cor:MP}), the interior H\"{o}lder continuity (Theorem~\ref{thm:holder}),
and a simple boundary pointwise H\"{o}lder regularity result similar to \cite[Proposition~4.12]{CC}.
Details are left to the interested reader.

\begin{prop}[Global H\"older regularity] \label{prop:globalholder}
Let $\Gamma=B_1\cap \{x_n=\psi(x')\}$. Assume that $u\in C(\overline{B_1})$ satisfies
$$
\begin{cases}
u \in S_{\lambda,\Lambda}(0) & \hbox{in}~\Omega^\pm\\
u_{\nu}^+-u_{\nu}^-=g & \hbox{on}~\Gamma\\
u=\varphi & \hbox{on}~\partial B_1,
\end{cases} 
$$
where $g\in L^\infty(\Gamma)$, with $\supp(g)\subset \Gamma\cap B_{1-2\rho}$, for some $0<\rho<1/4$,
$\varphi \in C^{0,\alpha}(\partial B_1)$, and $\psi\in C^{1,\alpha}(\overline{B_1'})$, for some $0<\alpha<1$. 
Then $u\in C^{0,\beta}(\overline{B_1})$, with $0<\beta\leq \min\{\alpha_1, \alpha/2\}$, and
$$
\|u\|_{C^{0,\beta}(\overline{B_1})} \leq \frac{C}{\rho^{\gamma}}\big( \|\varphi\|_{C^{0,\alpha}(\partial B_1)} + \|g\|_{L^\infty(\Gamma)}\big),
$$
where $0<\alpha_1<1$ is given in Theorem~\ref{thm:holder}, $\gamma=\max\{\alpha_1,\alpha\}$, and $C>0$ depends
only on $n$, $\lambda$, $\Lambda$, $\alpha$ and $[\Gamma]_{C^{1,\alpha}}$.
\end{prop}

%%%%%%%%%%%%%%%%%%%%%%%%%%%%%%%%%%%%%%%%%%%%%
\section{Flat interface transmission problems}\label{sect:FTP}
%%%%%%%%%%%%%%%%%%%%%%%%%%%%%%%%%%%%%%%%%%%%%

This section is devoted to the complete analysis of flat interface transmission problems of the form
\begin{align} \label{pb:existenceflat}
\begin{cases}
\F^\pm(D^2 u) = f^\pm & \hbox{in}~B_1^\pm\\
u_{x_n}^+ - u_{x_n}^- = g & \hbox{on}~T= B_1 \cap \{x_n=0\}.
\end{cases}
\end{align}
We prove the comparison principle, existence and uniqueness of viscosity solutions with prescribed boundary values,
and we derive $C^{1,\alpha}$ and $C^{2,\alpha}$ estimates for $u^+$ and $u^-$ up to the interface. 

%%%%%%%%%%%%%%%%%%%%%%%%%%%%%%%%%%%%%%%%%%%%%
\subsection{Viscosity solutions to \eqref{pb:existenceflat}}
%%%%%%%%%%%%%%%%%%%%%%%%%%%%%%%%%%%%%%%%%%%%%

\begin{defn}[Viscosity solution for flat problem]\label{def:viscosityflat}
We say that $u\in \USC({B_1})$ is a viscosity subsolution to \eqref{pb:existenceflat} in $B_1$ if for any $\varphi $ touching $u$ by above at $x_0$ in $B_1$, the following holds:
\begin{enumerate}[$(i)$]
\item if $x_0 \in B_1^\pm$ and $\varphi \in C^2(B_\delta(x_0))$,  then
$F^\pm(D^2 \varphi (x_0)) \geq  f^\pm(x_0)$;
\item if $x_0 \in T$ and $\varphi \in C^2 (\overline{B_\delta^+(x_0)}) \cap C^2 (\overline{B_\delta^-(x_0)})$, then 
$$
\varphi_{x_n}^+(x_0)-\varphi_{x_n}^-(x_0) \geq g(x_0)
$$
where $\varphi^\pm = \varphi|_{\overline{B_\delta^\pm(x_0)}}$.
\end{enumerate}
The notions of supersolution and solution are defined as usual.
\end{defn}

It can be seen that the following condition (see also \cite{DSFS}) is equivalent to $(ii)$:
\begin{enumerate}[$(ii')$]
\item Let $x_0 \in T$ and let
$$
\varphi (x) = P(x') + p^+ x_n^+ - p^- x_n^-
$$
where $P$ is a quadratic polynomial and $p^\pm\in \R$. If $\varphi$ touches $u$ by above at $x_0$ then
$$
p^+-p^- \geq g(x_0).
$$
\end{enumerate}

\begin{lem} \label{lem:equivdef}
Definition~\ref{def:viscosityflat} for subsolutions to \eqref{pb:existenceflat}
is equivalent to replacing $(ii)$ by the following statement: if $x_0 \in T$ and $\varphi \in C^2 (\overline{B_\delta^+(x_0)}) \cap C^2 (\overline{B_\delta^-(x_0)})$ touches $u$ by above at $x_0$ then
$$
\F^\pm(D^2 \varphi  (x_0)) \geq  f^\pm(x_0) \quad \hbox{or} \quad 
\varphi_{x_n}^+(x_0) -  \varphi_{x_n}^-(x_0) \geq  g(x_0).
$$
An analogous result holds for supersolutions to \eqref{pb:existenceflat}.
\end{lem}

\begin{proof}
If $u$ is a viscosity subsolution of \eqref{pb:existenceflat}, then it is clear that the statement is true. To prove the converse, let $x_0 \in T$ and assume that   $\varphi \in C^2 (\overline{B_\delta^+(x_0)}) \cap C^2 (\overline{B_\delta^-(x_0)})$ touches $u$ by above at $x_0$.
Suppose by way of contradiction that 
\begin{equation}\label{eq:contra1}
\varphi_{x_n}^+(x_0) -  \varphi_{x_n}^-(x_0) <  g(x_0).
\end{equation}
Define the function 
 $
 \psi(x) = \varphi(x)+\eta |x_n| - C|x_n|^2 \ \hbox{for}~x \in B_{\tau}(x_0),
 $
 where $\eta,\tau,C>0$ are constants to be determined. 
 For $\eta$ small and $C$ large fixed, choose $\tau<r$  such that
 $
 \eta |x_n| - C|x_n|^2\geq0 \ \hbox{in}~ B_{\tau}(x_0).
 $
 In particular, $\psi \in  C^2 ( \overline{B_\tau^+(x_0)}) \cap  C^2 (\overline{B_\tau^-(x_0)})$, and
 \begin{align*}
 \begin{cases}
 \psi (x_0) =  \varphi(x_0) = u(x_0)\\
 \psi(x) \geq \varphi(x) \geq u(x) &\hbox{for}~x\in B_{\tau}(x_0).
 \end{cases}
 \end{align*}
Then $\psi$ is a  test function touching $u$ by above at $x_0$ and, thus,
\begin{equation}\label{eq:contra2}
\F^\pm(D^2 \psi(x_0)) \geq  f^\pm(x_0) \quad \hbox{or} \quad
\psi_{x_n}^+(x_0) -  \psi_{x_n}^-(x_0) \geq  g(x_0).
\end{equation}
By \eqref{eq:contra1}, and choosing $\eta$ sufficiently small,
\begin{align*}
\psi_{x_n}^+(x_0) - \psi_{x_n}^-(x_0) = \varphi_{x_n}^+(x_0) - \varphi_{x_n}^-(x_0) +2 \eta < g(x_0).
\end{align*}
Therefore, the first inequality in \eqref{eq:contra2} must hold. 
Let $E_n = e_n e_n^T \in \mathcal{S}^n$. Then
\begin{align*}
\M_{\lambda/n,\Lambda}^+(D^2\psi(x_0))  = \M_{\lambda/n,\Lambda}^+(D^2 \varphi (x_0) - 2C E_n )
 \leq \M_{\lambda/n,\Lambda}^+(D^2 \varphi (x_0))  - 2C \lambda/n 
 < f^+(x_0)
\end{align*}
choosing $C$ sufficiently large. This is a contradiction to the first inequality in \eqref{eq:contra2} since, by uniform ellipticity,
\begin{align*}
f^+(x_0) &\leq \F^+(D^2\psi(x_0)) \leq \Lambda \| [D^2 \psi(x_0)]^+\| - \lambda \| [D^2\psi(x_0)]^-\| \leq \M_{\lambda/n,\Lambda}^+(D^2\psi(x_0)).
\end{align*} 
\end{proof}

%%%%%%%%%%%%%%%%%%%%%%%%%%%%%%%%%%%%%%%%%%%%%
\subsection{Lower and upper $\vep$-envelopes}
%%%%%%%%%%%%%%%%%%%%%%%%%%%%%%%%%%%%%%%%%%%%%

We will use a family of regularizations in the $x'$-direction
that was introduced in \cite[Section~3]{DSFS}. 

\begin{defn}\label{def:envelopes}
Given $u\in \USC(B_1)$ and any $\vep>0$, we define the upper $\vep$-envelope of $u$ in the $x'$-direction as
$$
u^\vep(y',y_n) = \sup_{x\in \overline{B_\rho} \cap \{x_n=y_n\}} \Big \{ u(x',y_n) - \tfrac{1}{\vep}|x'-y'|^2\Big\}
$$
for $y=(y',y_n) \in  \overline{B_\rho}\subset  B_1$. Similarly, given $u \in \LSC(B_1)$, we define the lower $\vep$-envelope of $u$ in the $x'$-direction as
$$
u_\vep(y',y_n) = \inf_{x\in \overline{B_\rho} \cap \{x_n=y_n\}} \Big \{ u(x',y_n) + \tfrac{1}{\vep}|x'-y'|^2\Big\}
$$
for $y=(y',y_n) \in  \overline{B_\rho}\subset  B_1$.
\end{defn}

\begin{lem}[{See~\cite[Lemma~3.1]{DSFS}}] \label{lem:envelopes}
The following properties hold:
\begin{enumerate}[$(i)$]
\item $u^\vep \geq u$ in $B_\rho$ and $\limsup_{\vep\to0} u^\vep = u$;
\item $u^\vep \in C_{y'}^{0,1}(\overline{B_\rho})$, with $[u^\vep]_{ C_{y'}^{0,1}(\overline{B_\rho})}\leq {6\rho}/{\vep}$;
\item $u^\vep \in C_{y'}^{1,1}$ by below in $B_\rho$. Thus, $u^\vep$ is punctually second order differentiable
in the $x'$-direction almost everywhere in $B_\rho$.
\end{enumerate}
\end{lem}

The modulus of continuity $\omega_h$ of a uniformly continuous function $h$ is denoted by
$$
w_h(r) = \sup_{|x-y|\leq r} |h(x)-h(y)| \qquad \hbox{for}~r>0.
$$

\begin{prop} \label{prop:envelopes}
 Let $f^\pm\in C(B_1^\pm)$ and $g\in C(T)$. If $u$ is a bounded viscosity subsolution (supersolution) to \eqref{pb:existenceflat}
then, for any $\vep>0$ small, $u^\vep$ is a viscosity subsolution (supersolution) to
$$ 
\begin{cases}
F^\pm (D^2 u^\vep) = f_\vep^\pm & \hbox{in}~B_r^\pm\\
(u^\vep)_{x_n}^+ - (u^\vep)_{x_n}^- = g_\vep & \hbox{on}~T_r=B_r\cap \{x_n=0\},
\end{cases}
$$
with $r\leq \rho-r_\vep$, $r_\vep=(2\vep\|u\|_{L^\infty(B_1)})^{1/2}$, $f_\vep^\pm = f - \omega_{f^\pm}(r_\vep)$, and $g_\vep=g- \omega_{g}(r_\vep)$.
\end{prop}

\begin{proof}
The fact that $F^\pm (D^2 u^\vep) \geq f_\vep^\pm$ in $B_r^\pm$ in the viscosity sense
is established in \cite[Lemma~3.1]{DSFS}. For the transmission condition,
let $y_0=(y_0',0) \in T_r$ and assume that
$$
\varphi (y) = P(y') + p^+ y_n^+ - p^- y_n^-,
$$
with $P$ a quadratic polynomial, touches $u^\vep$ by above at $y_0$. It is easy to check that
$$
\phi(y)= P(y'+y_0'-y_\vep') + \tfrac{1}{\vep} |y_\vep'-y_0'|^2 +  p^+ y_n^+ - p^- y_n^-
$$
touches $u$ at $y_\vep$ from above. Therefore, 
$$
p^+-p^- = (\phi_{x_n}^+ - \phi_{x_n}^-)(y_\vep)  \geq g(y_\vep) \geq g(y_0) - \omega_{g}(r_\vep) = g_\vep(y_0).
$$
\end{proof}

%%%%%%%%%%%%%%%%%%%%%%%%%%%%%%%%%%%%%%%%%%%%%
\subsection{Comparison principle and uniqueness}
%%%%%%%%%%%%%%%%%%%%%%%%%%%%%%%%%%%%%%%%%%%%%

We recall the classical notion of half-relaxed limits and some of its properties.
Let $\{u_k\}_{k=1}^\infty$ be a sequence of functions. For $x\in \overline{B_1}$, we define
$$
 \textstyle\limsup^* u_k(x) = \displaystyle\lim_{j\to \infty} \sup \Big\{ u_k(y) : k \geq j, \ y\in \overline{B_1}, \ \hbox{and} \ |y-x|\leq \tfrac{1}{j} \Big\}.
$$
Similarly, for $x\in \overline{B_1}$, we define 
$$
 \textstyle\liminf_* u_k(x) = \displaystyle\lim_{j\to \infty} \inf \Big\{ u_k(y) : k \geq j, \ y\in \overline{B_1}, \ \hbox{and} \ |y-x|\leq \tfrac{1}{j} \Big\}.
$$
Then $\textstyle\limsup^* u_k \in \USC(\overline{B_1})$ and $\textstyle\liminf_* u_k  \in \LSC(\overline{B_1})$.
The next lemma follows from \cite{CIL}.

\begin{lem} \label{lem:HRL}
Let $\{u_k\}_{k=1}^\infty \subseteq \USC(\overline{B_1})$ and $u=\limsupstar u_k$. Fix $x_0\in \overline{B_1}$. If
a continuous function $\varphi$ touches $u$ from above at $x_0$, then there exist indexes $k_j\to \infty$, points $x_j \in \overline{B_1}$, and functions $\varphi_j \in C^0$ such that $\varphi_j$ touches $u_{k_j}$ from above at $x_j$,  
$$
  x_j \to x_0 \quad \text{and} \quad u_{k_j}(x_j) \to u(x_0),  \qquad \hbox{as}~j\to \infty.
$$
Moreover, 
$\varphi_j(x) = \varphi(x) - \varphi(x_j) + u_{k_j}(x_j) + \delta(|x-x_0|^2-|x_j-x_0|^2)$,
for an arbitrary $\delta>0$. 
\end{lem}

\begin{thm}\label{thm:uniqueness}
Let $f_1^\pm, f_2^\pm \in C(B_1^\pm)\cap L^\infty(B_1^{\pm})$ and $g_1,g_2\in C(T)$. Assume that $u\in \USC(\overline{B_1})$ and  $v\in \LSC(\overline{B_1})$ are bounded and satisfy
$$
\begin{cases}
F^\pm(D^2 u) \geq f_1^\pm & \hbox{in}~B_1^\pm\\
u_{x_n}^+ - u_{x_n}^- \geq g_1 & \hbox{on}~T\\
\end{cases}
\quad \hbox{and}\quad
\begin{cases}
F^\pm(D^2 v) \leq f_2^\pm & \hbox{in}~B_1^\pm\\
v_{x_n}^+ - v_{x_n}^- \leq g_2 & \hbox{on}~T\\
\end{cases}
$$ 
in the viscosity sense. Then $w=u-v$ satisfies
$$
\begin{cases}
w \in \underline{S}_{\lambda/n,\Lambda}(f_1^\pm - f_2^\pm) & \hbox{in}~B_1^\pm\\
w_{x_n}^+ - w_{x_n}^- \geq g_1-g_2 & \hbox{on}~T
\end{cases}
$$
in the viscosity sense.
\end{thm}

\begin{proof}
We know that
$w \in \underline{S}_{\lambda/n,\Lambda}(f_1^\pm - f_2^\pm)$ in $B_1^\pm$. 
Hence, we only need to show the transmission condition.
Let $x_0=(x_0',0) \in T$ and assume that
$ P(x') + p^+ x_n^+ - p^- x_n^-$
 touches $w$ by above at $x_0$, where $P$ is a quadratic polynomial and $p^\pm \in \R$.
 We need to show that
 \begin{equation} \label{eq:NTS}
 p^+-p^- \geq g_1(x_0)-g_2(x_0).
 \end{equation}
 Fix $\tau>0$ and $C>0$ large to be chosen. Then 
$$
\varphi(x)= P(x') + (p^++\tau) x_n^+ - (p^--\tau) x_n^- - Cx_n^2
$$
 touches $w$ strictly by above at $x_0$, possibly in a smaller neighborhood where $\tau |x_n| - Cx_n^2\geq 0$.

For $\vep>0$, consider  $w_\vep = u^\vep - v_\vep$, where $u^\vep$ and $v_\vep$ are the upper and lower $\vep$-envelopes of $u$ and $v$, respectively, given in Definition~\ref{def:envelopes}. By Lemma~\ref{lem:envelopes}$(i)$, $\limsup_{\vep\to0} w_\vep =w$. 
By Lemma~\ref{lem:HRL}, up to a subsequence, there exist points $x_\vep \in B_1$, with $x_\vep\to x_0$, and functions
$$\varphi_\vep(x)= \varphi(x) - \varphi(x_\vep) + w_\vep(x_\vep)  + |x-x_0|^2 - |x_\vep - x_0|^2$$
such that  $\varphi_\vep$ touches $w_\vep$ strictly by above at  $x_\vep$. In particular, given $\delta>0$ small, there exists $\eta>0$ such that
$ \varphi_\vep - w_\vep \geq  \eta >0 \ \hbox{on}~\partial B_\delta(x_\vep).$
By Proposition~\ref{prop:envelopes},
\begin{equation}\label{eq:pucci1}
\M_{\lambda/n,\Lambda}^+(D^2 w_\vep) \geq (f_1^\pm)_\vep-(f_2^\pm)_\vep \quad \hbox{in}~B_\rho^\pm 
\end{equation}
in the viscosity sense, for some $0<\rho<1$ such that $\overline{B_\delta(x_\vep)}\subset B_\rho$.  

Choose $C$ large enough so that
\begin{align} \label{eq:pucci2}
 \M_{\lambda/n,\Lambda}^+(D^2 \varphi_\vep)
 &\leq  \Lambda \|D_{x'}^2P\| + 2\Lambda  -  2\lambda (C-1)  
 < \inf_{B_\rho^\pm} \big \{   (f_1^\pm)_\vep-(f_2^\pm)_\vep \big\} \quad \hbox{in}~B_\rho^\pm.
\end{align}
Since $\varphi_\vep$ touches $w_\vep$ by above at $x_\vep$, it immediately follows that $x_\vep \in T$. Otherwise,
$$
 \M_{\lambda/n,\Lambda}^+(D^2 \varphi_\vep(x_\vep) ) < (f_1^\pm)_\vep(x_\vep)-(f_2^\pm)_\vep(x_\vep),
$$
which contradicts \eqref{eq:pucci1}.
Define
\begin{equation}\label{eq:psi}
 \psi = \varphi_\vep - w_\vep- \eta/2.
 \end{equation}
Then $\psi \geq\eta/2> 0$ on $\partial B_\delta(x_\vep)$ and $\psi(x_\vep)=-\eta/2<0$.  Let $\C_\psi$ be the convex envelope of $-\psi_-$ in $B_{2\delta}'(x_\vep)=B_{2\delta}(x_\vep) \cap \{x_n=0\}$, where we have extended $-\psi_-\equiv0$ outside of $\overline{B_\delta'(x_\vep)}$. By Lemma~\ref{lem:envelopes}$(iii)$, we know that $\psi \in C_{x'}^{1,1}$ by above in $B_\rho$. Hence, for any $x_0'\in \overline{B_\delta'(x_\vep)}$, there exists a convex paraboloid $P(x')$ with uniform opening that touches $\psi(x',0)$ by above at $x_0'$. 
We have $\C_\psi \in C^{1,1}_{x'}(\overline{B'_\delta(x_\vep)})$, and for any $\gamma>0$, we claim that
$$
|D_\gamma|\equiv \big|\big\{x'\in \overline{B_\delta'(x_\vep)} : \C_\psi(x') = \psi(x',0) \ \hbox{and} \ |\nabla'\C_\psi(x')|\leq \gamma\big\}\big|>0.
$$
Indeed, the claim follows from the proof of the ABP estimate for the function $\psi$ defined on $\overline{B'_\delta(x_\vep)}$. First, if we remove the gradient condition in the definition of $D_\gamma$, then using \eqref{eq:conditionv} and Remark~\ref{refinedABP}, we get that the contact set has positive measure. 
Furthermore, since $\C_\psi(x_\vep') = \psi(x_\vep',0)$ and $\nabla'\C_\psi(x_\vep')=0'$, then $x_\vep' \in D_\gamma\neq \varnothing$, for any $\gamma>0$. Therefore, $|D_\gamma|>0$,  for any $\gamma>0$, since the gradient of $\C_\psi$ is continuous. 
Hence, choosing  $\gamma\leq \tfrac{\eta}{4\delta}$, there exists $y_\vep' \in D_\gamma$ such that both $u^\vep$ and $v_\vep$ are punctually second order differentiable at $y_\vep=(y_\vep',0)$ in the $x'$-direction, and such that 
$$l(x')=\nabla' \C_\psi(y_\vep')\cdot (x'-y_\vep') + \psi(y_\vep)$$ 
touches $\psi$ from below at $y_\vep$ on $\overline{B_\delta(x_\vep)}$. Indeed, 
$l(x')\leq \psi(x',0)$, for all $x'\in B_\delta'(x_\vep)$, and
$$
l(x') \leq |\nabla' \C_\psi(y_\vep')| |x'-y_\vep'| \leq \tfrac{\eta}{4\delta} (2\delta) = \eta/2\leq \psi(x) \quad \hbox{for all}~
x\in \partial B_\delta(x_\vep).
$$
Therefore, $l \leq \psi$ on  $\partial B_\delta^\pm(x_\vep)$. In particular, by \eqref{eq:psi}, we see that
$w_\vep \leq \varphi_\vep- l - \eta/2$  on $\partial B_\delta^\pm(x_\vep).$
Moreover, in view of \eqref{eq:pucci1} and \eqref{eq:pucci2}, we get
$$
\M_{\lambda/n,\Lambda}^+(D^2 w_\vep)  > \M_{\lambda/n,\Lambda}^+(D^2 (\varphi_\vep- l - \eta/2)) \quad \hbox{in}~B_\delta^\pm(x_\vep).
$$
Hence, by the comparison principle, it follows that $w_\vep \leq \varphi_\vep - l- \eta/2$ on $\overline{B_\delta(x_\vep)}$.
Define
$$
\bar \varphi = \varphi_\vep  - l - \eta/2.
$$

Consider the viscosity solutions $\bar u^\vep$ and $\bar v_\vep$ to the Dirichlet problems,
\begin{align*}
    \begin{cases}
F^\pm(D^2 \bar u^\vep) = (f_1^\pm)_\vep & \hbox{in}~B_\delta^\pm(x_\vep)\\
\bar u^\vep = u^\vep & \hbox{on}~\partial B_\delta^\pm(x_\vep)
\end{cases}
\quad
\hbox{and}
\quad
\begin{cases}
F^\pm(D^2 \bar v_\vep) = (f_2^\pm)_\vep & \hbox{in}~B_\delta^\pm(x_\vep)\\
\bar v_\vep = v_\vep & \hbox{on}~\partial B_\delta^\pm(x_\vep).
\end{cases}
\end{align*}
By the comparison principle, $\bar u_\vep \geq u_\vep$ and $\bar v_\vep \leq v_\vep$ in $B_\delta(x_\vep)$, and thus,
\begin{equation} \label{eq:TCreplacements}
(\bar u^\vep)^+_{x_n}-(\bar u^\vep)^-_{x_n}\geq (g_1)_\vep \quad \hbox{and} \quad (\bar v_\vep)^+_{x_n}-(\bar v_\vep)^-_{x_n}\leq (g_2)_\vep
\end{equation}
on $B_\delta(x_\vep)\cap \{x_n=0\}$, in the viscosity sense, where 
$(g_1)_\vep = g_1 - \omega_{g_1}\big((2\vep\|u\|_{L^\infty(B_1)})^{1/2}\big)$
and
$(g_2)_\vep = g_2 - \omega_{g_2}\big((2\vep\|v\|_{L^\infty(B_1)})^{1/2}\big).$
Recall that $u^\vep, v_\vep \in C_{x'}^{1,\alpha}(y_\vep)$, and thus, by pointwise $C^{1,\alpha}$-estimates, there exists $r_0>0$ and linear polynomials $l_u^\pm$ and $l_v^\pm$ such that
$$
\|\bar u_\vep^\pm-l_u^\pm\|_{L^\infty(B_r^\pm(y_\vep))}\leq C r^{1+\alpha},
\qquad\hbox{for all}~0<r<r_0. 
$$
For simplicity, call $p_u^\pm = \nabla l_u^\pm \cdot e_n$ and $p_v^\pm = \nabla l_v^\pm \cdot e_n$. Then by similar arguments as in \cite[Lemma~4.3]{DSFS}, we see that  \eqref{eq:TCreplacements} holds pointwise at $y_\vep$. Namely,
\begin{equation}\label{eq:TCreplacements2}
p_u^+-p_u^- \geq (g_1)_\vep(y_\vep) \quad \hbox{and} \quad p_v^+-p_v^- \leq (g_2)_\vep(y_\vep).
\end{equation}
Let $\bar w_\vep=\bar u^\vep - \bar v_\vep$. Then by previous computations, $\bar w_\vep$ satisfies
$$\begin{cases}
\M_{\lambda/n,\Lambda}^+(D^2 \bar w_\vep) \geq \M_{\lambda/n,\Lambda}^+(D^2 \bar \varphi)&\hbox{in}~B_\delta^\pm(x_\vep)\\
\bar \varphi \geq \bar w_\vep& \hbox{on}~\partial B_\delta^\pm(x_\vep).
\end{cases}$$
It follows that 
$\bar \varphi \geq \bar w_\vep$ in $B_\delta(x_\vep)$ and  $\bar \varphi(y_\vep)=\bar w_\vep(y_\vep).$ Since $\bar w _\vep \in C^{1,\alpha}(y_\vep)$, we have that
\begin{align*}
p^++\tau &= \bar \varphi_{x_n}^+(y_\vep) \geq (\bar w_\vep)_{x_n}^+(y_\vep) = p_u^+ - p_v^+\\ 
p^--\tau &= \bar \varphi_{x_n}^-(y_\vep) \leq (\bar w_\vep)_{x_n}^-(y_\vep) = p_u^- - p_v^-.
\end{align*}
Therefore, combining the previous estimates with \eqref{eq:TCreplacements2}, we get
\begin{align*}
p^+-p^- + 2\tau & \geq (g_1)_\vep(y_\vep) - (g_2)_\vep(y_\vep) \\
&= (g_1-g_2)(y_\vep) +  \omega_{g_1}\big((2\vep\|u\|_\infty)^{1/2}\big) -  \omega_{g_2}\big((2\vep\|v\|_\infty)^{1/2}\big).
\end{align*}
Recall that $y_\vep \in B_\delta(x_\vep)$ and $x_\vep\to x_0$, as $\vep\to 0$. Hence, letting $\tau \to0$, then $\delta\to 0$, so that $y_\vep \to x_\vep$, and finally, $\vep \to 0$, we obtain \eqref{eq:NTS}.
\end{proof}

\begin{cor}[Uniqueness] \label{cor:uniqueness}
The transmission problem 
$$\begin{cases}
F^\pm(D^2 u^\pm) = f^\pm& \hbox{in}~B_1^\pm\\
 u_{x_n}^+ -   u_{x_n}^- = g & \hbox{on}~T\\
u=\phi & \hbox{on}~\partial B_1
\end{cases}$$
has at most one viscosity solution.
\end{cor}

\begin{thm}[Comparison principle] \label{thm:comparison}
Let $u,v : \overline{B_1}\to \R$ be a bounded viscosity subsolution and a bounded viscosity supersolution of \eqref{pb:existenceflat}, respectively. If $u\leq v$ on $\partial B_1$, then 
$$u\leq v \quad\hbox{in}~B_1.$$
\end{thm}

\begin{proof}
Let $w=u-v$. By Theorem~\ref{thm:uniqueness},
$$
\begin{cases}
w \in \underline{S}_{\lambda/n,\Lambda}(0) & \hbox{in}~B_1^\pm\\
w_{x_n}^+ - w_{x_n}^- \geq 0 & \hbox{on}~T.
\end{cases}
$$
From the ABP estimate (Theorem~\ref{thm:ABP2}), $w\geq0$ in $B_1$.
\end{proof}

We finish this subsection by constructing barriers for \eqref{pb:existenceflat}.

\begin{lem}[Barriers]\label{lem:Perronbarriers}
Given any $\phi\in C(\partial B_1)$, there exist functions $\underline{u},\overline{u} \in C^2(B_1\setminus T)\cap C(\overline{B_1})$ such that $\underline u$ is a viscosity subsolution and $\overline{u}$ is a viscosity supersolution of \eqref{pb:existenceflat}, respectively, with
$$\begin{cases}
\underline u \leq \overline u & \hbox{in}~B_1\\
\underline u = \overline u = \phi & \hbox{on} \ \partial B_1.
\end{cases}$$
\end{lem}

\begin{proof}
Let $f= f^+\chi_{B_1^+} + f^- \chi_{B_1^-}$. Let $\underline \psi , \overline \psi \in C^2(B_1)\cap C(\overline{B_1})$
be the unique solutions to the Dirichlet problems
\begin{align*}
\begin{cases}
\M_{\lambda,\Lambda}^- (D^2 \underline\psi)=  \|f\|_{L^\infty(B_1)} & \hbox{in}~B_1\\
\underline \psi = \phi - \tfrac{1}{2} \|g\|_{L^\infty(T)} |x_n| & \hbox{on}~\partial B_1
\end{cases}
\quad
\hbox{and} 
\quad
\begin{cases}
\M_{\lambda,\Lambda}^+(D^2 \overline\psi)= - \|f\|_{L^\infty(B_1)} & \hbox{in}~B_1\\
\overline \psi = \phi + \tfrac{1}{2} \|g\|_{L^\infty(T)} |x_n| & \hbox{on}~\partial B_1.
\end{cases}
\end{align*}
Define the functions $\underline u, \overline u \in C^2(B_1\setminus T)\cap C(\overline{B_1})$ as
$$
\underline u = \underline \psi + \tfrac{1}{2} \|g\|_{L^\infty(T)} |x_n|
\quad\hbox{and}\quad 
\overline u = \overline \psi - \tfrac{1}{2} \|g\|_{L^\infty(T)} |x_n|.
$$
Then  $\underline u = \overline u = \phi $ on $\partial B_1$. By construction,
$\underline u$ is a subsolution of \eqref{pb:existenceflat}. Arguing similarly, we see that $\overline u$ is a supersolution of \eqref{pb:existenceflat}. By the comparison principle (Theorem~\ref{thm:comparison}), we conclude that $\underline u \leq \overline u$ in $B_1$.
\end{proof}

%%%%%%%%%%%%%%%%%%%%%%%%%%%%%%%%%%%%%%%%%%%%%
\subsection{Existence}
%%%%%%%%%%%%%%%%%%%%%%%%%%%%%%%%%%%%%%%%%%%%%

Our main theorem of this subsection is the following.

\begin{thm}[Existence]\label{thm:existflat}
Let $f^\pm\in C(B_1^\pm\cup T)\cap L^\infty(B_1)$, $g\in C(T)$, and $\phi\in C(\partial B_1)$. Then there exists a unique viscosity solution $u\in C(\overline{B_1})$ of \eqref{pb:existenceflat} such that $u=\phi$ on $\partial B_1$.
\end{thm}

We begin by defining the set of admissible subsolutions as
$$
\A =\Big \{ v \in \USC(\overline{B_1}) :  \underline{u} \leq v \leq \overline{u} \ \hbox{and} \ v~\text{viscosity subsolution of}~\eqref{pb:existenceflat}\Big\},
$$
where $\underline{u}$ and $\overline{u}$ are as in Lemma \ref{lem:Perronbarriers}.
Note that $\underline u \in \A$, so $\A \neq \varnothing$. Set 
$$
u(x) = \sup_{v\in \A} v(x).
$$
Let $u_*\in \LSC( \overline{B_1})$ and $u^*\USC(\overline{B_1})$
denote the lower and upper semicontinuous envelopes of $u$ in $B_1$, respectively. 
Clearly, $u_* \leq u \leq u^*$.

\begin{lem} \label{lem:limitenvsubsol}
If $\{v_k\}_{k=1}^\infty \subset \A$ then $v= \limsupstar v_k \in \A$.
\end{lem}

\begin{proof}
We only need to show that $v$ is a subsolution to \eqref{pb:existenceflat}. Fix $x_0 \in {B_1}$ and assume that $\varphi\in C^2$ touches $v$ by above at $x_0$. 
Then, by Lemma~\ref{lem:HRL}, there exist indexes $k_j\to \infty$, points $x_j \in \overline{B_1}$, and functions $\varphi_j \in C^2$ such that $D^2 \varphi_j = D^2\varphi + 2\delta I$, $\varphi_j$ touches $v_{k_j}$ from above at $x_j$, and 
$$
  x_j \to x_0 \quad \text{and} \quad v_{k_j}(x_j) \to v(x_0), \quad \hbox{as}~j\to \infty.
$$

If $x_0 \in B_1^\pm$ then, for $j$ sufficiently large, we may assume that $x_j \in B_1^\pm$.
Since $v_{k_j} \in \A$, and $\varphi$ touches $v_{k_j}$ by above at $x_j$, it follows that
$$
F^\pm (D^2 \varphi(x_j) +2\delta I ) = F^\pm (D^2 \varphi_j(x_j) ) \geq f^\pm(x_j).
$$ 
Letting $j\to \infty$ and $\delta\to 0$, by continuity of $F, D^2 \varphi$, and $f^\pm$, we obtain the result.

If $x_0 \in T$, then either there exists $j_0\geq 1$ such that for all $j\geq j_0$ we have $x_j \in B_1^\pm$, or for all $j_0 \geq 1$ there exists $j \geq j_0$ such that $x_j \in T$. In the first case, by the previous argument,
$F^\pm (D^2 \varphi(x_0) ) \geq f^\pm(x_0)$.
In the second case, we have that $(\varphi_j^+)_{x_n}(x_j) - (\varphi_j^-)_{x_n}(x_j) \geq g(x_j)$.
Passing to the limit, we obtain the desired estimate. 

By Lemma~\ref{lem:equivdef}, we conclude that $v$ is a  viscosity subsolution of \eqref{pb:existenceflat}.
\end{proof}

We divide the proof of Theorem~\ref{thm:existflat} into two steps.

\begin{lem}[Step 1]
$u^*$ is a subsolution of \eqref{pb:existenceflat}. In particular, $u^*\in \A$.
\end{lem}

\begin{proof}
Let $x_0 \in B_1$. By the construction of $u^*$, there exist points $\{x_k\}_{k=1}^\infty$ and functions $\{v_k\}_{k=1}^\infty \subset \A$ such that $x_k \to x_0$, and
$$
 u^*(x_0) = \lim_{k\to\infty} v_k(x_k). 
$$
Hence, $\limsupstar v_k (x_0) \geq u^*(x_0)$. On the other hand, for all $x\in{B_1}$, we have
$$
u^*(x) \geq \limsupstar v_j(x),
$$
for any $\{v_j\}_{j=1}^\infty \subset \A$.
Therefore, $\limsupstar v_k(x_0) = u^*(x_0)$. In particular, if $\varphi\in C^2$  touches $u^*$ by above at $x_0$, then the same holds for $\limsupstar v_k$.
By Lemma~\ref{lem:limitenvsubsol}, $u^*\in \A$.
\end{proof}

By the previous lemma, we can conclude that $u^*=u$ on $\overline{B_1}$ and, thus, $u\in \A$.

\begin{lem}[Step 2] 
$u_*$ is a supersolution of \eqref{pb:existenceflat}.
\end{lem}

\begin{proof}
Assume by means of contradiction that there exist $x_0\in {B_1}$
and some test function $\varphi$ that touches $u_*$ from below at $x_0$ such that the following holds:
\begin{enumerate}[$(i)$]
\item If $x_0\in B_1^\pm$, then
$
F^\pm (D^2\varphi(x_0)) > f^\pm(x_0).
$
\item If $x_0 \in T$, then
$
  \varphi_{x_n}^+(x_0) -  \varphi_{x_n}^-(x_0) >  g(x_0).
$
\end{enumerate}
Without loss of generality, we may assume that $\varphi$ touches $u_*$ strictly from below. Otherwise, take $\varphi-\vep|x-x_0|^2$, for some $\vep>0$ small.
By continuity, the strict inequalities in $(i)$ and $(ii)$ hold in a ball $B_r(x_0)$, for some $r>0$ small and thus
$\varphi$ is a classical strict subsolution  in $B_r(x_0)$.
Consider $\varphi_\delta = \varphi - \delta |x-x_0|^2 + \delta r^2 /2$, for some $\delta>0$ sufficiently small. Then
$$
\varphi_\delta (x_0) > u_*(x_0) \quad \hbox{and} \quad 
\varphi_\delta < u_* \leq u ~ \hbox{on}~\partial B_r(x_0).
$$
Hence, there exists some $x_1 \in B_r(x_0)$ such that 
$
u(x_1) < \varphi_\delta (x_1).
$
Define 
$$
\bar u=
\begin{cases}
\max\{ u, \varphi_\delta\} & \hbox{in}~B_r(x_0)\\
u & \hbox{on}~ \overline{B_1} \setminus B_r(x_0).
\end{cases}
$$
Since $u \in \A$ and $\varphi_\delta$ is a subsolution, it is easy to see that $\bar{u} \in \A$. This is a contradiction with
$$
u(x_1) = \sup_{v\in \A} v(x_1) \geq \bar u(x_1) = \varphi(x_1) > u(x_1).
$$
Therefore, $u_*$ is a supersolution of \eqref{pb:existenceflat}.
\end{proof}

By definition, $u_*\leq u$. Since $u_*$ is a supersolution and $u$ is a subsolution of \eqref{pb:existenceflat}, and $u_*=u$ on $\partial B_1$, then by the comparison principle (Theorem~\ref{thm:comparison}), we get $u_*=u$ on $\overline{B_1}$. Hence, we conclude that
$$
u_*=u=u^* \qquad \hbox{on}~\overline{B_1}.
$$
In particular, by Corollary~\ref{cor:uniqueness}, $u\in C(\overline{B_1})$ is the unique viscosity solution of \eqref{pb:existenceflat} with $u=\phi$ on $\partial B_1$. This concludes the proof of Theorem~\ref{thm:existflat}.

%%%%%%%%%%%%%%%%%%%%%%%%%%%%%%%%%%%%%%%%%%%%%
\subsection{Regularity estimates up to the flat interface}
%%%%%%%%%%%%%%%%%%%%%%%%%%%%%%%%%%%%%%%%%%%%%

In this subsection we derive $C^{1,\alpha}$ and $C^{2,\alpha}$ estimates for viscosity solutions to the flat problem \eqref{pb:existenceflat}. 
Recall the definitions of the H\"older exponents $0<\bar{\alpha},\bar{\bar{\alpha}}<1$ given at the end of the Introduction.

We complete the flat interface theory by proving the following theorems.

\begin{thm}[{$C^{1,\alpha}$ regularity for \eqref{pb:existenceflat}}] \label{thm:regflat}
Fix $0<\alpha<\bar \alpha$. Assume that $g\in C^{0,\alpha}(T)$ and $f^\pm\in C(B_1^\pm \cup T)$ satisfies
$$
\Big( \fint_{B_r(x_0) \cap B_1^\pm} |f^\pm|^n\, dx\Big)^{1/n} \leq C_{f^\pm} r^{\alpha-1},
$$ 
for all $x_0 \in B_1^\pm\cup T $ and $r>0$ small. Let $u$ be a bounded viscosity solution to \eqref{pb:existenceflat}.  
Then $u^\pm\in C^{1,\alpha}(\overline{B_{1/2}^\pm})$ and 
\begin{equation} \label{eq:regflat}
\|u^\pm \|_{C^{1,\alpha}(\overline{B_{1/2}^\pm})} \leq C\big(\|u\|_{L^\infty(B_1)} +  \|g\|_{C^{0,\alpha}(T)} + C_{f^-} + C_{f^+} \big),
\end{equation}
where $C>0$ depends only on $n$, $\lambda$, $\Lambda$ and $\alpha$. 
\end{thm}

\begin{thm}[{$C^{2,\alpha}$ regularity for \eqref{pb:existenceflat}}] \label{thm:regflat2}
Let $\F$ be a concave operator. Fix $0<\alpha<\bar{\bar\alpha}$.
Assume that $g\in C^{1,\alpha}(T)$ and $f^\pm \in C^{0,\alpha}(\overline{B_1^\pm})$ with $f^+=f^-$ on $T$.
Let $u$ be a bounded viscosity solution to \eqref{pb:existenceflat} with  $F\equiv F^+ \equiv F^-$.
Then $u^\pm\in C^{2,\alpha}(\overline{B_{1/2}^\pm})$ and
$$\|u^\pm \|_{C^{2,\alpha}(\overline{B_{1/2}^\pm})} \leq C\big(\|u\|_{L^\infty(B_1)} +  \|g\|_{C^{1,\alpha}(T)} + \|f^-\|_{C^{0,\alpha}(\overline{B_1^-})} +  \|f^+\|_{C^{0,\alpha}(\overline{B_1^+})} \big)$$
where $C>0$ depends only on $n$, $\lambda$, $\Lambda$ and $\alpha$. 
\end{thm}

We begin with some preliminary results.

\begin{lem} 
Let $v$ be a bounded viscosity solution to \eqref{pb:existenceflat} with $f^\pm\equiv g\equiv0$.
Then $v \in C^{1,\bar\alpha}(\overline{B_{1/2}})$ and
$$
\| v \|_{C^{1,\bar\alpha}(\overline{B_{1/2}})} \leq C \| v \|_{L^\infty(B_1)},
$$
where $C>0$ depends only on $n$, $\lambda$ and $\Lambda$.
\end{lem}

\begin{proof}
We apply \cite[Theorem~1.2]{DSFS} with $a=b=1$. Then $v^\pm \in C^{1,\bar\alpha}(\overline{B_{1/2}^\pm})$ with the estimate
$\| v^\pm \|_{C^{1,\bar\alpha}(\overline{B_{1/2}^\pm})} \leq C \| v \|_{L^\infty(B_1)}$.
In particular, $v$ satisfies the transmission condition in the classical sense, and thus, it is differentiable in $B_{1/2}$.
\end{proof}

By a standard scaling argument we obtain the following estimate.

\begin{cor} \label{cor:homog}
Let $0<r\leq 1$. Suppose that $v$ is a bounded viscosity solution to 
 \begin{align*}
\begin{cases}
\F^\pm(D^2 v) = 0 & \hbox{in}~B_r^\pm\\
 v_{x_n}^+ -  v_{x_n}^- = 0 & \hbox{on}~B_r\cap\{x_n=0\}.
\end{cases}
\end{align*}
Then for any $0<\rho\leq r/2$, we have that $v \in C^{1,\bar\alpha}(\overline{B_\rho})$ with
\begin{align*}
\underset{B_{\rho}}{\osc} \big(v-\nabla v(0)\cdot x\big) &\leq C \Big (\frac{\rho}{r} \Big)^{1+\bar\alpha} \underset{\overline{B_r}}{\osc} v\\
|\nabla v(0)| & \leq C \frac{1}{r} \underset{\overline{B_r}}{\osc} v, 
\end{align*}
where $C>0$ depends only on $n$, $\lambda$ and $\Lambda$.
\end{cor}

\begin{lem} 
Let $\F$ be a concave operator. Let $v$ be a bounded viscosity solution to
\begin{align} \label{pb:flateqhomog}
\begin{cases}
\F(D^2 v) = 0 & \hbox{in}~B_1^\pm\\
v_{x_n}^+ - v_{x_n}^- = 0 & \hbox{on}~T.
\end{cases}
\end{align}
Then $v \in C^{2,\bar{\bar\alpha}}(\overline{B_{1/2}})$ and
$$
\| v \|_{C^{2,\bar{\bar\alpha}}(\overline{B_{1/2}})} \leq C \| v \|_{L^\infty(B_1)},
$$
where $C>0$ depends only on $n$, $\lambda$ and $\Lambda$.
\end{lem}

\begin{proof}
Let $w$ be the unique viscosity solution of the Dirichlet problem
 \begin{align*}
\begin{cases}
\F(D^2 w) = 0 & \hbox{in}~B_1\\
w = v & \hbox{on}~\partial B_1.
\end{cases}
\end{align*}
Since $\F$ is a concave operator, we know that $w\in C_{\mathrm{loc}}^{2,\bar{\bar\alpha}}(B_1)$ and
$\| w \|_{C^{2,\bar{\bar\alpha}}(\overline{B_{1/2}})} \leq C \| w \|_{L^\infty(B_1)}$.
By uniqueness of viscosity solutions to flat interface problems (see Corollary~\ref{cor:uniqueness})
it follows that $w=v$ in $B_1$.
\end{proof}

Again a scaling argument gives the following estimate.

\begin{cor} \label{cor:homog2}
Let $0<r\leq 1$. Suppose that $v$ is a bounded viscosity solution of \eqref{pb:flateqhomog}.
Then, for any $0<\rho\leq r/2$, we have that $v^\pm \in C^{2,\bar{\bar\alpha}}(\overline{B_\rho^\pm})$ with
\begin{align*}
\underset{B_{\rho}^\pm}{\osc} \big(v^\pm- \tfrac{1}{2}x^t D^2 v^\pm(0) x-\nabla v(0)\cdot x\big) &\leq C \Big (\frac{\rho}{r} \Big)^{2+\bar{\bar\alpha}} \underset{\overline{B_r}}{\osc} v\\
r^2 \|D^2v^\pm (0)\|+ r |\nabla v(0)| & \leq C \underset{\overline{B_r}}{\osc} v, 
\end{align*}
where $C>0$ depends only on $n$, $\lambda$ and $\Lambda$.
\end{cor}

\begin{proof}[Proof of Theorem~\ref{thm:regflat}]
By interior estimates, it is enough to prove \eqref{eq:regflat} for points on $T\cap \overline{B_{1/2}}$.
Without loss of generality we assume that $u(0)=0$ and $g(0)=0$. Otherwise, we may consider $u-u(0)-\tfrac{g(0)}{2}|x_n|$. 
Let $M=\|u\|_{L^\infty(B_1)} +  \|g\|_{C^{0,\alpha}(T)} + C_{f^-}+C_{f^+}$.  

We will show that there exist $0< \gamma<1$ and $C_0, C_1>0$, depending only on $n$, $\lambda$, $\Lambda$ and $\alpha$, and a sequence of vectors $\{A_k\}_{k=0}^\infty$ such that
\begin{align} \label{eq:induction1}
\underset{B_{\gamma^k}}{\osc} ( u-A_k\cdot x) &\leq C_0 M \gamma ^{k(1+\alpha)}\\
|A_{k}-A_{k-1}| &\leq C_1 M \gamma^{(k-1)\alpha}\label{eq:induction2},
\end{align}
for any $k\geq 0$, where $A_{-1}=0$.
If this holds, then one can check that $A_k\to A_\infty$, as $k\to \infty$, and that
$u\in C^{1,\alpha}(0)$ with 
$$
|u(x)-A_\infty \cdot x| \leq CM |x|^{1+\alpha}.
$$
We prove \eqref{eq:induction1}--\eqref{eq:induction2} by induction. For $k=0$, we set $A_0=0$, and choose $C_0\geq 2$ universal such that 
$$
\underset{B_1}{\osc} u \leq 2\|u\|_{L^\infty(B_1)} \leq C_0 M.
$$
Assume that the estimates hold for some $k\geq0$. 
Let $r=\gamma^k$ and $B=A_k$. Let $v\in C(\overline{B_r})$ be the unique viscosity solution to $\F^\pm(D^2 v) = 0$ in $B_r^\pm$,   with $v_{x_n}^+ -  v_{x_n}^- = 0$ on $T\cap B_r$, and $ v=u-B\cdot x $ on $\partial B_r$,
see Theorem~\ref{thm:existflat}. From the ABP estimate (Theorem~\ref{thm:ABP2}) we have
\begin{equation} \label{eq:MP}
\underset{\overline{B_r}}{\osc} v \leq \underset{\overline{B_r}}{\osc} (u-B\cdot x). 
\end{equation}
Fix $\rho \leq r/2$ to be determined. By Corollary~\ref{cor:homog}, we know that $v\in C^{1,\bar\alpha}(\overline{B_\rho})$, with
\begin{align} \label{eq:decayosc}
\underset{B_{\rho}}{\osc} (v-A\cdot x) &\leq C \Big (\frac{\rho}{r} \Big)^{1+\bar\alpha} \underset{\overline{B_r}}{\osc} v\\
|A| & \leq C \frac{1}{r} \underset{\overline{B_r}}{\osc} v, \label{eq:boundA}
\end{align}
where $A=\nabla v (0)$.
Let $\rho = \gamma r$ and $\vep=\bar\alpha - \alpha>0$. Choose $\gamma\leq 1/2$ small enough so that $C\gamma^\vep\leq1/2$. 
Combining \eqref{eq:MP}, \eqref{eq:decayosc}, and the induction hypothesis,  we see that
\begin{equation}\label{eq:boundv}
\underset{B_\rho}{\osc} (v-A\cdot x) \leq C \Big (\frac{\rho}{r} \Big)^{1+\bar\alpha} \underset{\overline{B_r}}{\osc}  (u-B\cdot x)
\leq C\gamma^{1+\alpha+\vep} C_0 M \gamma^{k(1+\alpha)}\leq\tfrac{1}{2} C_0 M \gamma^{(k+1)(1+\alpha)}.
\end{equation}

Let $w= u - B\cdot x - v$. 
By Theorem~\ref{thm:uniqueness} and the fact that $B\cdot x+v$ is differentiable, we have $w \in {S}_{\lambda/n,\Lambda}(f^\pm)$ in $B_r^\pm$, with $w_{x_n}^+ -  w_{x_n}^- = g$ on $T\cap B_r$, and $w=0$ on $\partial B_r.$
Using the rescaled ABP estimate, and the assumptions on $g$ and $f^\pm$, we get
\begin{align*}\nonumber
\| w \|_{L^\infty(B_\rho)} &\leq C \rho \big(\|g\|_{L^\infty(T \cap B_\rho)} +  \|f^-\|_{L^n(B_\rho^-)}+ \|f^+\|_{L^n(B_\rho^+)}\big) \\ \nonumber
& \leq C \rho^{1+\alpha} \big(\|g\|_{C^{0,\alpha}(T\cap B_\rho)} + C_{f^-} + C_{f^+}\big)\leq CM \rho^{1+\alpha}.
\end{align*}
Choose $C_0\geq 4C$. In view of \eqref{eq:boundv} and the previous estimate, we have
\begin{align*} 
\underset{B_{\gamma^{k+1}}}{\osc} (u-(A+B)\cdot x) &= \underset{B_\rho}{\osc} (u-(A+B)\cdot x) 
\leq \underset{B_\rho}{\osc} w + \underset{B_\rho}{\osc} (v-A\cdot x) \\
&\leq 2C M  \rho^{1+\alpha} + \tfrac{1}{2} C_0 M \gamma^{(k+1)(1+\alpha)}
\leq C_0 M \gamma^{(k+1)(1+\alpha)}.
\end{align*}
Hence, the estimate in \eqref{eq:induction1} holds for $k+1$ with $A_{k+1}=A+B$. To prove \eqref{eq:induction2}, we use \eqref{eq:boundA}, \eqref{eq:MP}, and the induction hypothesis to get
$|A_{k+1}-A_k|=|A| \leq C_1 M \gamma^{k\alpha}$,
where $C_1=CC_0$.
\end{proof}

\begin{proof}[Proof of Theorem~\ref{thm:regflat2}]
It is enough to prove the estimate at $x=0$.
Without loss of generality, we assume that $u(0)=0$ and $g(0)=0$. 
Suppose further that $f^+(0)=f^-(0)=0$. 
Otherwise,  by uniform ellipticity, there exists $s \in \R$ such that $\F(sI)=f^\pm(0)$ and $|s|\leq |f^\pm(0)|/\lambda$. Hence, consider $\tilde \F(M)= \F(M+s I) - f^\pm (0)$ and $v= u- \tfrac{s}{2}|x|^2$. We have that $\tilde \F \in \mathcal{E}(\lambda, \Lambda)$, $\tilde \F(0)=0$, and $\tilde \F(D^2 v)= f^\pm-f^\pm(0)$.

Let $M=\|u\|_{L^\infty(B_1)} +  \|g\|_{C^{1,\alpha}(T)} + \|f^-\|_{C^{0,\alpha}(\overline{B_1^-})} + \|f^+\|_{C^{0,\alpha}(\overline{B_1^+})}$.  
We will show that there exist $0< \gamma<1$ and $C_0, C_1>0$, depending only on $n$, $\lambda$, $\Lambda$ and $\alpha$, and sequences of quadratic polynomials $P_k^\pm= \tfrac{1}{2} x^t A_k^\pm x + B_k \cdot x$, $k\geq0$, such that
\begin{align} \label{eq:induction3}
\underset{B_{\gamma^k}}{\osc} ( u^\pm- P_k^\pm ) &\leq C_0 M \gamma ^{k(2+\alpha)}\\
\gamma^{k-1} \|A_{k}^\pm-A_{k-1}^\pm\| + |B_{k}-B_{k-1}|   &\leq C_1 M \gamma^{(k-1)(1+\alpha)}, \label{eq:induction4}
\end{align}
for any $k\geq 0$, where $A_{-1}^\pm=0$ and $B_{-1}=0$. Furthermore, 
\begin{align} \label{eq:FAF}
\F (A_k^\pm) &=0\\
(A_k^+)_{ij} - (A_k^-)_{ij} = 0 \ \hbox{if}~i,j \neq n 
 & \quad \hbox{and} \quad
(A_k^+)_{jn} - (A_k^-)_{jn}=  g_{x_j} (0) \  \hbox{if}~ j \neq n, \label{eq:matrix}
\end{align}
where $A_{ij}$ denotes the element in the $(i,j)$-entry of the matrix $A$.

The proof is by induction. For $k=0$, choose $B_0=0$ and $A_0^\pm$ symmetric such that 
\eqref{eq:matrix} holds and $(A_0^\pm)_{nn}$ is given by $\F(A_0^\pm)=0$. Then $\|A_0^\pm\| \leq C_1 |\nabla ' g(0)|  \leq C_1 M$, for some $C_1>0$. Moreover, choose $C_0>1$ large so that
$$
\underset{B_1}{\osc} ( u^\pm- P_0^\pm ) \leq 2 (\|u\|_{L^\infty(B_1)} + \|A_0^\pm\|) \leq 2(1+C_1)M \leq C_0 M. 
$$

Assume that \eqref{eq:induction3}--\eqref{eq:matrix} hold for some $k\geq0$.
Let $r=\gamma^k$ and $P_k= P_k^+ \chi_{\overline{B_r^+}} + P_k^-\chi_{\overline{B_r^-}}$. Note that, by \eqref{eq:matrix}, we have $P_k^+= P_k^-$ on $T \cap B_r$, and thus, $P_k \in C(\overline{B_r})$.
Let $v\in C(\overline{B_r})$ be the viscosity solution to
\begin{align*}
\begin{cases}
\F(D^2 v+A_k^\pm) = 0 & \hbox{in}~B_r^\pm\\
 v_{x_n}^+ -  v_{x_n}^- = 0 & \hbox{on}~T\cap B_r\\
 v=u- P_k & \hbox{on}~\partial B_r,
\end{cases}
\end{align*}
see Theorem~\ref{thm:existflat}. From the ABP estimate (Theorem~\ref{thm:ABP2}) and the fact that $\F(A_k^\pm)=0$,
\begin{equation} \label{eq:MP2}
\underset{\overline{B_r}}{\osc} v \leq \underset{\overline{B_r}}{\osc} (u-P_k). 
\end{equation}
Fix $\rho \leq r/2$ to be determined. By Corollary~\ref{cor:homog2}, we know that $v^\pm\in C^{2,\bar{\bar\alpha}}(\overline{B_\rho^\pm})$, with
\begin{align} \label{eq:decayosc2}
\underset{B_{\rho}^\pm}{\osc} (v^\pm- P^\pm) &\leq C \Big (\frac{\rho}{r} \Big)^{2+\bar{\bar\alpha}} \underset{\overline{B_r}}{\osc} v\\\
r^2 \|Q^\pm\|+ r|R| & \leq C \underset{\overline{B_r}}{\osc} v, \label{eq:boundQR}
\end{align}
where $P^\pm(x)=x^t Q^\pm x+R\cdot x$, $Q^\pm =D^2 v^\pm (0)$, and  $R=\nabla v (0)$. Furthermore,
since $D^2 v^\pm$ is continuous up to the interface, then by the equation, $\F(D^2 v + A_k^\pm)=0$, it follows that
\begin{equation} \label{eq:FAFk}
\F(Q + A_k^\pm)=0.
\end{equation}

Let $\rho = \gamma r$ and $\vep=\bar{\bar\alpha} - \alpha>0$. Choose $\gamma\leq 1/2$ small enough so that $C\gamma^\vep\leq1/2$. 
Combining \eqref{eq:MP2}, \eqref{eq:decayosc2}, and the induction hypothesis,  we see that
\begin{equation}\label{eq:boundv2}
\begin{aligned}
\underset{B_\rho^\pm}{\osc} (v^\pm-P^\pm) &\leq C \Big (\frac{\rho}{r} \Big)^{2+\bar{\bar\alpha}} \underset{\overline{B_r}}{\osc}  (u-P_k) 
&\leq\tfrac{1}{2} C_0 M \gamma^{(k+1)(2+\alpha)}.
\end{aligned}
\end{equation}

Let $w= u - P_k - v$. Then $w \in {S}_{\lambda/n,\Lambda}(f^\pm)$ in $B_r^\pm$, with $w_{x_n}^+ -  w_{x_n}^- = g-\nabla'g(0)\cdot x'$ on $T\cap B_r$, and $w = 0$ on $\partial B_r$.
Note that the term $\nabla'g(0)\cdot x'$ comes from \eqref{eq:matrix} since for any $x=(x',0)\in T\cap B_r$, we have
$$
(P_k^+)_{x_n}(x) - (P_k^-)_{x_n}(x) = \sum_{j=1}^{n-1} \big( (A_k^+)_{jn} - (A_k^-)_{jn}\big) x_j = \sum_{j=1}^{n-1}  g_{x_j} (0) x_j
= \nabla'g(0)\cdot x'.
$$
Using the rescaled ABP estimate, and the assumptions on $g$ and $f^\pm$, we get
\begin{align*}
\| w \|_{L^\infty(B_\rho)} &\leq C \rho \big(\|g-g(0)-\nabla'g(0)\cdot x'\|_{L^\infty(T \cap B_\rho)} \\
&\quad\qquad+  \|f^--f^-(0)\|_{L^n(B_\rho^-)}+ \|f^+-f^+(0)\|_{L^n(B_\rho^+)}\big) \\ \nonumber
& \leq C \rho^{2+\alpha} \big([g]_{C^{1,\alpha}(0)} + [ f^-]_{C^{0,\alpha}(0)}  +  [f^+]_{C^{0,\alpha}(0)} \big)\leq CM \rho^{2+\alpha}.
\end{align*}
Choose $C_0\geq 4C$. In view of \eqref{eq:boundv2} and the previous estimate, we have
\begin{align*} 
\underset{B_{\gamma^{k+1}}^\pm}{\osc} (u^\pm-P_k^\pm-P^\pm) &
\leq \underset{B_\rho}{\osc} w + \underset{B_\rho^\pm}{\osc} (v^\pm-P^\pm) 
\leq C_0 M \gamma^{(k+1)(2+\alpha)}.
\end{align*}
Hence, the estimate in \eqref{eq:induction3} holds for $k+1$ with $P_{k+1}=P_k+P$. To prove \eqref{eq:induction4}, we use \eqref{eq:boundQR}, \eqref{eq:MP2}, and the induction hypothesis to get
$$
\gamma^{k} \|A_{k+1}^\pm-A_k^\pm\| + |B_{k+1}-B_k|=\gamma^k \|Q^\pm \| + |R| \leq C_1 M \gamma^{k(1+\alpha)},
$$ 
where $C_1=CC_0$. Since $A_{k+1}^\pm = A_k^\pm+ Q$, then \eqref{eq:FAF} follows from \eqref{eq:FAFk}. Moreover, \eqref{eq:matrix} follows from the induction hypothesis.
\end{proof}

%%%%%%%%%%%%%%%%%%%%%%%%%%%%%%%%%%%%%%%%%%%%%
\section{Approximation and stability results} \label{sec:approxlem}
%%%%%%%%%%%%%%%%%%%%%%%%%%%%%%%%%%%%%%%%%%%%%
 
In this section we prove several approximation and stability results for viscosity solutions to
\eqref{eq:TP} that will be useful for the proof of the $C^{1,\alpha}$ and $C^{2,\alpha}$ estimates in the following sections. 
We point out that in the case of $C^{1,\alpha}$ interfaces our proof is constructive, given that in the closedness result
Lemma~\ref{lem:closedness}, the condition $\Gamma_k\in C^2$ cannot be relaxed.

%%%%%%%%%%%%%%%%%%%%%%%%%%%%%%%%%
\subsection{Closedness}
%%%%%%%%%%%%%%%%%%%%%%%%%%%%%%%%%

First, we prove that viscosity solutions to transmission problems \eqref{eq:TP} with $C^2$ interfaces are closed under uniform limits.

\begin{lem} \label{lem:closedness}
Assume that $\Gamma_k \in C^2$ and $u_k \in C(B_1)$ satisfy
$$
\begin{cases}
\F_k^\pm(D^2u_k) = \ f_k^\pm & \hbox{in}~\Omega_k^\pm\\
(u_k^+)_\nu- (u_k^-)_\nu = g_k & \hbox{on}~\Gamma_k,
\end{cases}
$$
where $\Gamma_k=B_1\cap \{x_n=\psi_k(x')\}$ for $\psi_k \in C^2(B_1')$, $f_k^\pm \in C(\Omega_k^\pm \cup \Gamma_k)$, and $g_k \in C(\Gamma_k)$, for $k\geq1$.
Suppose that there are continuous functions $u$, $f^{\pm}$ and $g$, and elliptic operators $F^{\pm}\in\mathcal{E}(\lambda,\Lambda)$
such that, as $k\to\infty$,
\begin{enumerate}[$(i)$]
\item $\F_k^{\pm}\to \F^{\pm}$ uniformly on compact subsets of $\mathcal{S}^n$;
\item $u_k\to u$ uniformly on compact subsets of $B_1$;
\item $\|f_k^\pm-f^\pm\|_{L^\infty(\Omega_k^\pm)}\to 0$;
\item $\|g_k-g\|_{L^\infty(\Gamma_k)}=\sup_{x'\in B_1'} |g_k(x',\psi_k(x'))-g(x',0)|\to 0$;
\item $\Gamma_k\to T$ in $C^2$ in the sense that $\|\psi_k\|_{C^2(B_1')}\to 0$. 
\end{enumerate}
Then $u \in C(B_1)$ is a viscosity solution to  
$$
\begin{cases}
\F^\pm(D^2u) = f^{\pm} & \hbox{in}~B_1^\pm\\
u^+_{x_n}- u^-_{x_n} = g & \hbox{on}~T.
\end{cases}
$$
\end{lem}

\begin{proof}
We only prove that $u$ is a viscosity subsolution. 
First, we show that
$$
\F^\pm(D^2u) \geq f^\pm   \qquad\hbox{in}~B_1^\pm.
$$
Suppose by contradiction that this fails. Then there is a point $x_0\in B_1^\pm$ and a test function $\varphi \in C^2(B_1^\pm)$ such that $\varphi$ touches $u$ from above at  $x_0$, and
$$
\F^\pm(D^2\varphi(x_0)) < f^\pm(x_0). 
$$
Without loss of generality, we can assume that
$x_0\in B_1^+$, and that $\varphi$ touches $u$ strictly from above. Otherwise, we can replace $\varphi$ by 
$
\varphi + \vep|x-x_0|^2,
$
with $\vep$ small.
Then, since $u_k\to u$ uniformly on compact sets, there exists $\vep_k>0$
such that $\varphi+\vep_k \geq u_k$ in $\overline{B_r(x_0)}\subset B_1^+$, for $k$ large and  some $r$ small.
Define
$$
d_k=\inf_{B_{r_k}(x_0)} (\varphi+\vep_k -u_k)\geq0,
$$
with $0<r_k< r$ and  $r_k \searrow 0$. Since $\Gamma_k \to T$, we can choose $r_k$ such that $\overline{B_{r_k}(x_0)} \subset \Omega_k^+$, for $k$ large. Let $x_k \in \Omega_k^+$ be a point where the infimum is attained, that is, 
$$
d_k = \varphi(x_k)+\vep_k -u_k(x_k)
$$
and define $c_k = \vep_k - d_k$. Then $x_k\to x_0$, $c_k \to 0$, 
and $\varphi+c_k$ touches $u_k$ from above at $x_k \in \Omega^+_k$, for $k$ large.
Hence, since $\F_k^+(D^2u_k(x_k)) \geq f^+_k$ in $\Omega_k^+$, we must have
$$
\F_k^+ (D^2\varphi(x_k))\geq f^+_k(x_k) .
$$
Passing to the limit as $k\to\infty$, we get
$$
\F^+(D^2\varphi(x_0))\geq f^+(x_0),
$$
which is a contradiction. 
It remains to show that the transmission condition holds. If not, there exists $x_0\in T$, $r>0$ small, and $\varphi \in C^2(\overline{B_{r}^\pm(x_0)})$ such that $\varphi$ touches $u$ from above at $x_0$, and
$$\varphi^+_{x_n}(x_0)- \varphi^-_{x_n}(x_0) < g(x_0).$$
We can assume that $\varphi$ touches $u$ strictly from above at $x_0$, and that 
$$ \F^\pm(D^2\varphi(x_0)) < f^\pm(x_0).$$
Otherwise, we can replace $\varphi$ by
 $
 \varphi(x)+\eta |x_n| - C|x_n|^2,
 $
 with $\eta$ small and $C$ large such that $\eta |x_n| - C|x_n|^2\geq0$ in a small neighborhood of $x_0$.
Arguing as before, there exist $c_k, r_k, x_k$ such that $\phi(x)=\varphi(x',x_n-\psi_k(x'))+ c_k$ touches $u_k$ from above at $x_k\in B_{r_k}(x_0)$, with $c_k\to 0$, $x_k\to x_0$ and $r_k \to 0$. Then either there exists $k_0\geq 1$ such that for every $k\geq k_0$ we have $x_k \in \Omega_k^\pm$, and thus,
$$
\F_k^\pm(D^2 \phi(x_k))\geq f^\pm_k (x_k), 
$$
or, for every $k_0\geq1$, there exists $k \geq k_0$ such that $x_k\in \Gamma_k$ and
$$
\phi_{\nu_k}^+(x_k) -  \phi_{\nu_k}^-(x_k) \geq g_k(x_k).
$$
Passing to the limit, we get a contradiction in both cases.
\end{proof}

%%%%%%%%%%%%%%%%%%%%%%%%%%%%%%%%%%%%%%%%%%%%%
\subsection{Stability for $C^{1,\alpha}$ interfaces} 
%%%%%%%%%%%%%%%%%%%%%%%%%%%%%%%%%%%%%%%%%%%%%

We distinguish two cases in the transmission condition in \eqref{eq:TP}. If $g$ is close to $0$,
then we will approximate $u$ with a differentiable function across $\Gamma$ (Lemma~\ref{lem:approx1}).
If $g$ is away from $0$, then $u$ is singular at the interface and we will approximate $u$ with solutions to flat interface problems
(Lemma~\ref{lem:stability}). Along the way, we prove the stability of flat solutions (Lemma~\ref{lem:stabilityflat}).
Observe that our methods here are constructive in nature.

\begin{lem}[Stability for $g$ close to $0$] \label{lem:approx1}
Fix $0<\alpha<\gamma<\bar\alpha$, $0<\tau <3/4$, and $0<\delta<1$. 
Assume that $\psi \in C^{1,\alpha}(\overline{B_1'})$ and $u\in C(B_1)$ is a viscosity solution to \eqref{eq:TP}, with $\|u\|_{L^\infty(B_1)}\leq 1$
and
$$\|g\|_{L^\infty(\Gamma_{3/4})}+\|f^-\|_{L^n(\Omega^-)}+\|f^+\|_{L^n(\Omega^+)} \leq \delta.$$
Then there exists 
$\theta>0$, depending only on $n$, $\lambda$ and $\Lambda$, such that, if \eqref{eq:closeness2} holds,
then there exists a viscosity solution $v\in C_{\mathrm{loc}}^{1,\gamma}(B_{3/4})\cap C^{0,\beta}(\overline{B_{3/4}})$ to $F^\pm(D^2 v)=0$ in $\Omega^\pm_{3/4}$ with $v=u$ on $ \partial B_{3/4}$ such that
\begin{align*}
\| u - v \|_{L^\infty(B_{3/4-\tau})} &\leq C(\tau^\beta+\delta),
\end{align*} 
for some $C>0$ depending only on $n$, $\lambda$, $\Lambda$, $\alpha$ and $\|\psi\|_{C^{1,\alpha}(\overline{B_1'})}$, where $\beta \equiv \alpha_1/2<1$ and $\alpha_1>0$ is given in Theorem~\ref{thm:holder}.
\end{lem}

\begin{proof}
Fix $0<\rho< 1/2$, $0<\delta<1$, and $\theta>0$ to be determined. 
Given $\vep>0$ small, for $x\in B_1$ we define
$$
\F_\vep(M,x) = h_\vep(x) \F^+(M) + ((1-h_\vep(x)) \F^-(M),
$$
where $h_\vep \in C^\infty (B_1)$, $0 \leq h_\vep \leq 1$, and
$$
h_\vep(x) =
\begin{cases}
1 & \hbox{if}~x\in B_1 \cap  \{x_n> \psi(x')+\vep\}\\
0 & \hbox{if}~x\in B_1\cap \{x_n<\psi(x')-\vep\}.
\end{cases}
$$
Note that $\F_\vep \in \mathcal{E}({\lambda,\Lambda})$ and  $\F_\vep(0,x)\equiv 0$. Moreover,
it is easy to see that $\F_\vep \to \F^\pm$ uniformly in compact subsets of $\mathcal{S}^n \times \Omega^\pm$.

Let $v_\vep$ be the viscosity solution to
$$
\begin{cases}
\F_\vep(D^2 v_\vep,x) = 0 & \hbox{in}~B_{3/4}\\
v_\vep = u & \hbox{on}~\partial B_{3/4}.
 \end{cases}
$$
For $x\in B_1$, define
$$
\beta_\vep(x) = \sup_{M\in \mathcal{S}^n \setminus\{0\}} \frac{|\F_\vep(M,x)-\F_\vep(M,0)|}{\|M\|}.
$$
By the previous estimate and the fact that $0\leq h_\vep \leq 1$, we have
 $$
\beta_\vep(x) \leq (1-h_\vep(x)) \theta 
+ (1-h_\vep(0))  \theta \leq 2\theta
 $$
for all $x\in B_1$. Hence, for any $0<r\leq 1$, it follows that
$$
\Big(\fint_{B_r} \beta_\vep(x)^n \, dx\Big)^{1/n} \leq 2\theta.
$$
Choose $0<\theta\leq \theta_0/2$, where $\theta_0>0$ depends only $n$, $\lambda$ and $\Lambda$
and is sufficiently small so that 
$v_\vep\in C_{\mathrm{loc}}^{1,\bar\gamma}(B_{3/4})$, for some $0<\gamma<\bar\gamma<\bar\alpha$ and, for any $0<\rho<3/4$,
$$
\| v_\vep\|_{C^{1,\bar\gamma}(\overline{B_{\rho}})}\leq C_0 \|u\|_{L^\infty(B_1)}\leq C_0.
$$
By compactness, $v_\vep \to v$ in $C_{\mathrm{loc}}^{1,\gamma}(B_{3/4})$ as $\vep \to 0$.
Moreover, by the closedness of viscosity solutions under uniform limits, $v$ satisfies
\begin{equation} \label{eq:equationforv}
\F^\pm(D^2 v) = 0 \quad \hbox{in}~\Omega^\pm_{3/4}
\end{equation}
in the viscosity sense.
Let $w=u-v$. By Theorem~\ref{thm:holder}, we have that $u\in C^{0,\alpha_1}(\overline{B_{3/4}})$ and
$$
\|u\|_{C^{0,\alpha_1}(\overline{B_{3/4}})} \leq C ( 1 + \delta) \leq 2C,
$$
for some $C>0$ depending only on $n$, $\lambda$, $\Lambda$, $\alpha$, and $\|\psi\|_{C^{1,\alpha}(\overline{B_1'})}$.
It follows that $v_\vep \in C^{0,\beta}(\overline{B_{3/4}})$, with $\beta\equiv\alpha_1/2$, and 
$$
\| v_\vep\|_{C^{0,\beta}(\overline{B_{3/4}})} \leq C \|u\|_{C^{0,\alpha_1}(\partial B_{3/4})}  \leq C_1.
$$
Then $w\in C^{0,\beta}(\overline{B_{3/4}})$, with  $w=0$ on $\partial B_{3/4}$, and for any $0<\tau <1/4$,
$$
\| w \|_{L^\infty(\partial B_{3/4-\tau})} \leq [w]_{C^{0,\beta}(\overline{B_{3/4}})} \tau^\beta \leq C_2 \tau^\beta, 
$$
where $C_2=2C+C_1$. Also, since $u$ and $v$ satisfy \eqref{eq:TP}  and \eqref{eq:equationforv}, respectively,  then
$$
\begin{cases}
w \in S_{\lambda/n,\Lambda}(f^\pm) & \hbox{in}~\Omega^\pm_{3/4-\tau}\\
 w_{\nu}^+ -  w_{\nu}^- = g & \hbox{on}~\Gamma_{3/4-\tau}.
 \end{cases}
 $$
 From the ABP estimate (Theorem~\ref{thm:ABP2}) and the assumptions on $g$ and $f^\pm$, we conclude that
 $\| u - v \|_{L^\infty(B_{3/4-\tau})} \leq C(\tau^\beta+\delta)$.
\end{proof}

For $|a|<1/2$, define $T_a=B_1\cap \{x_n=a\}$. Consider the flat interface transmission problem
\begin{equation}
\begin{cases} \label{eq:flata}
\F^\pm(D^2 v) =0 & \hbox{in}~B_1\setminus T_{a} \\
v_{x_n}^+ - v_{x_n}^-=g_a & \hbox{on}~T_{a},
 \end{cases}
\end{equation}
where $g_a$ is a smooth mollification of $\chi_{T_a}$ such that $\supp(g_a)\subset B_{3/4} \cap T_{a}$. For convenience, when $a=0$ we call the solution $v_0$. Since $g_a$ has compact support, Proposition~\ref{prop:globalholder}, with $\rho=1/8$, gives global H\"{o}lder continuity of solutions to \eqref{eq:flata}.

\begin{cor} \label{cor:globalholderflat}
Fix $0<\gamma<1$.
Let $v$ be a viscosity solution to \eqref{eq:flata} in $B_1$, with $\varphi=v|_{\partial B_1} \in C^{0,\gamma}(\partial B_1)$. Then $v\in C^{0,\beta}(\overline{B_1})$ with $0<\beta\leq \min\{\alpha_1, \gamma/2\}$ and
$$
\|v\|_{C^{0,\beta}(\overline{B_1})} \leq C \big ( 1+ \|\varphi\|_{C^{0,\gamma}({\partial B_1})}\big),
$$
where $0<\alpha_1<1$ is given in Theorem~\ref{thm:holder} and  $C>0$ depends only on $n$, $\lambda$, $\Lambda$ and $\gamma$.
\end{cor}

\begin{lem}[Stability of flat solutions]\label{lem:stabilityflat}
Fix $0<\alpha <\gamma<\bar \alpha$ and $C_0>0$.
For any $\vep>0$, there exists $0<\delta<\min\{\vep,1/2\}$ such that if $v$ satisfies \eqref{eq:flata}, with $a=\delta$, $v=v_0$ on $\partial B_1$, $\|v_0\|_{C^{0,\gamma}(\partial B_1)}\leq C_0$,
and $\|g_\delta(\cdot,\delta)- g_0(\cdot,0)\|_{L^\infty(B_1')}\leq \delta$, then
\begin{equation*}
\| v - v_0 \|_{L^\infty(B_{1})} \leq \vep.
\end{equation*}
Fix $0<r<3/4$. If, in addition, $[g_\delta(\cdot,\delta)- g_0(\cdot,0)]_{C^{0,\gamma}(\overline{B_1'})}\leq \delta$, then
\begin{align*}
\|\nabla'v-\nabla'v_0\|_{C^{0,{\alpha}}(\overline{B_{r}})}+\|{(v_0^+)}_{x_n}- v^-_{x_n}-1\|_{L^\infty(D_{\delta,r})}  &\leq (1-r)^{-(1+\gamma)} \vep,
\end{align*}
where  $D_{\delta,r} = B_{r} \cap \{0<x_n< \delta\}$.
An analogous statement holds when $a=-\delta$.
\end{lem}

\begin{proof}
We prove the lemma by contradiction. Assume there exist $\vep_0$, $v_k$, $g_k$ such that
\begin{equation*}
\begin{cases}
\F^\pm (D^2 v_k) =0 & \hbox{in}~B_1\setminus T_{1/k} \\
(v_k^+)_{x_n} - (v_k^-)_{x_n}=g_k & \hbox{on}~T_{1/k}\\
v_k = v_0 & \hbox{on}~\partial B_1,
 \end{cases}
\end{equation*}
with $\|g_k(\cdot,1/k)-g_0(\cdot,0)\|_{C^{0,\gamma}(\overline{B_1'})} \leq 1/k$, such that
\begin{align}\label{eq:contradiction1}
\| v_k - v_0 \|_{L^\infty(B_{1})} &> \vep_0 \\ \label{eq:contradiction2}
 \|\nabla'v_k-\nabla'v_0\|_{C^{0,\alpha}(\overline{B_r})} &> (1-r)^{-(1+\gamma)} \vep_0\\\label{eq:contradiction3}
\|{(v_0^+)}_{x_n}- (v_k^-)_{x_n}-1\|_{L^\infty(D_{1/k,r})} &> (1-r)^{-(1+\gamma)} \vep_0,
\end{align}
for all $k\geq 1$. 
From the ABP estimate (Theorem~\ref{thm:ABP2}), we get
$$
\|v_k\|_{L^\infty(B_1)} \leq \sup_{\partial B_1} |v_0| + C \|g_k\|_{L^\infty(T_{1/k})}\leq C_0 + 2C.
$$
Also, from the global H\"{o}lder estimate in Corollary~\ref{cor:globalholderflat}, we have that
$$
\|v_k\|_{C^{0,\beta}(\overline{B_{1}})}\leq C\big(\|g_k\|_{L^\infty(T_{1/k})} +\|v_0\|_{C^{0,\gamma}(\partial B_1)}\big)\leq C(2+C_0)\equiv C_1,
$$
where $0<\beta\leq \min\{\alpha_1,\gamma/2\}$.
By compactness, up to a subsequence,  $v_k\to v$ uniformly in $B_1$. Moreover, by Lemma~\ref{lem:closedness}, $v$ satisfies 
\begin{equation*}
\begin{cases}
\F^\pm(D^2 v) =0 & \hbox{in}~B_1^\pm \\
v_{x_n}^+ - v_{x_n}^-=g_0 & \hbox{on}~T_{0}\\
v = v_0 & \hbox{on}~\partial B_1.
 \end{cases}
\end{equation*}
By uniqueness of viscosity solutions (Corollary~\ref{cor:uniqueness}), we see that $v=v_0$ on $B_1$. This contradicts \eqref{eq:contradiction1} for $k$ sufficiently large. Moreover, since $g_k$ is smooth, by Theorem~\ref{thm:regflat} (rescaled)
we have that $v_k \in C^{1,\gamma}$ in the $x'$-direction in $B_{r}$, with
$$
\|\nabla' v_k \|_{C^{0,\gamma}(\overline{B_{r}})} \leq \frac{C}{(1-r)^{1+\gamma}}\big(\|v_k\|_{L^\infty(B_1)} + \|g_k\|_{C^{0,\gamma}(T_{1/k})}\big)\leq \frac{C_2}{(1-r)^{1+\gamma}}. 
$$
By compactness, up to a subsequence,  $\nabla' v_k\to w$ in $C^{0,\alpha}(\overline{B_{r}})$, where $0<\alpha<\gamma$. By uniqueness of distributional limits, we have that $w=\nabla' v_0$. This contradicts \eqref{eq:contradiction2} for $k$ sufficiently large. Furthermore, by \eqref{thm:regflat} in Theorem~\ref{thm:regflat}, we have
$$
\|(v_k^-)_{x_n}\|_{C^{0,\gamma}\big(\overline{B_{r}}\cap \{x_n\leq 1/k\}\big)}\leq \frac{C}{(1-r)^{1+\gamma}}\big(\|v_k\|_{L^\infty(B_1)} + \|g_k\|_{C^{0,\gamma}(T_{1/k})}\big)\leq \frac{C_3}{(1-r)^{1+\gamma}}.
$$
By the previous argument, up to a subsequence, $(v_k^-)_{x_n} \to (v_0^-)_{x_n}$ uniformly in $B_{r}^-$. Let $x\in D_{1/k,r}$ and denote $\bar{x}=(x',0)\in T_0$. Note that $|x-\bar{x}|< 1/k$. Then
\begin{align*}
|{(v_0^+)}_{x_n}(x)- (v_k^-)_{x_n}(x)-1| & 
\leq |{(v_0^+)}_{x_n}(x)-{(v_0^+)}_{x_n}(\bar{x})|
+ |{(v_0^+)}_{x_n}(\bar{x})-(v_0^-)_{x_n}(\bar{x})-1|\\
&\quad+ |(v_0^-)_{x_n}(\bar{x})-(v_k^-)_{x_n}(\bar{x})|
+|(v_k^-)_{x_n}(\bar{x})-(v_k^-)_{x_n}(x)|\\
&\equiv{\rm I+II+III+IV}.
\end{align*}
By construction of $v_0$,  ${\rm II}=0$.  Moreover, ${\rm III}\to0$ as $k\to +\infty$, by uniform convergence. Also,
\begin{align*}
{\rm I} &\leq [(v_0^+)_{x_n}]_{C^{0,\gamma}(\overline{B_{r}^+})}|x-\bar{x}|^\gamma\leq \frac{C}{(1-r)^{1+\gamma}}\frac{1}{k^\gamma}\to 0 \quad \hbox{($r$ is fixed)}\\
{\rm IV} &\leq [(v_k^-)_{x_n}]_{C^{0,\gamma}\big(\overline{B_{r}}\cap \{x_n\leq 1/k\}\big)}|x-\bar{x}|^\gamma\leq \frac{C_3}{(1-r)^{1+\gamma}}\frac{1}{k^\gamma}\to 0.
\end{align*}
This contradicts \eqref{eq:contradiction3} for $k$ sufficiently large.
\end{proof}

\begin{cor} \label{cor:stabilityflat}
Fix $0<\alpha<\gamma<\bar\alpha$ and $\vep >0$.  Let $\overline{v}, \underline{v} \in C(\overline{B_1})$ be as in Lemma~\ref{lem:stabilityflat}, with $a=\delta$, and $a=-\delta$, respectively, where $\delta$ is the minimum between the two that exist for $\overline{v}$ and $\underline{v}$. Fix $0<r<1$. Then
\begin{align*}
\| \overline{v} - \underline{v} \|_{L^\infty(B_{1})} &\leq \vep\\
\|\nabla'\overline{v}-\nabla'\underline{v}\|_{C^{0,\alpha}(\overline{B_r})} &\leq (1-r)^{-(1+\gamma)}  \vep\\
\|(\underline{v}^+)_{x_n}- (\overline{v}^-)_{x_n}-1\|_{L^\infty(D_{\delta,r})}& \leq (1-r)^{-(1+\gamma)} \vep,
\end{align*}
where $D_{\delta,r} = B_{r} \cap \{|x_n|< \delta\}$.
\end{cor}

\begin{lem}[Stability for $g$ away from $0$] \label{lem:stability}
Fix $0<\alpha< \gamma <\bar\alpha$ and $\vep>0$.
Let $\delta>0$ and $v=\overline{v}\chi_{\Omega^-} + \underline{v} \chi_{\Omega^+}$, where $\delta$, $\overline{v}$, and $\underline{v}$ are given in Corollary~\ref{cor:stabilityflat} with $B_{3/4}$
in place of $B_1$. 
Assume that $\psi \in C^{1,\alpha}(\overline{B_1'})$ and $u\in C(B_1)$ is a viscosity solution to \eqref{eq:TP} with $\|u\|_{L^\infty(B_1)}\leq 1$
and
$$\|\psi\|_{C^{1,\alpha}(\overline{B_1'})} + \|g-1\|_{L^\infty(\Gamma)}+\|f^-\|_{L^n(\Omega^-)}+\|f^+\|_{L^n(\Omega^+)} \leq \delta.$$
If $v=u$ on $\partial B_{3/4}$ then
$$
\|u-v\|_{L^\infty(B_{1/2})}\leq  C\vep^{1/2},
$$
where $C>0$ depends only on $n$, $\lambda$, $\Lambda$ and $\alpha$.
\end{lem}

\begin{proof}
By Theorem~\ref{thm:holder} and the assumptions of $u$, $g$, $f^\pm$, we have $u\in C^{0,\alpha_1}(\overline{B_{3/4}})$, with
$$
\|u\|_{C^{0,\alpha_1}(\overline{B_{3/4}})} \leq C\big ( \|u\|_{L^\infty(B_1)}+ \|g\|_{L^\infty(\Gamma)} + \|f^-\|_{L^n(\Omega^-)}+\|f^+\|_{L^n(\Omega^+)}\big)\leq C_1.
$$
Since $v=u$ on $\partial B_{3/4}$, by Corollary~\ref{cor:globalholderflat} it follows that $v\in C^{0,\beta}(\overline{\Omega_{3/4}^\pm})$ with $\beta\leq\alpha_1/2$, and
$$
\|v\|_{C^{0,\beta}(\overline{\Omega_{3/4}^\pm})} \leq C\big(1+ \|u\|_{C^{0,\alpha_1}(\partial {B_{3/4}})}\big),
$$
for some $C>0$ depending only on $n$, $\lambda$, $\Lambda$, and $\alpha$.  
Hence, 
$$
\|v\|_{C^{0,\beta}(\overline{\Omega_{3/4}^\pm})} \leq C(1+C_1)\equiv C_2.
$$
Note that $v$ is not continuous across $\Gamma_{3/4}$ since $\overline{v}-\underline{v} \neq0$ on $\Gamma_{3/4}$. Let $w$ satisfy
\begin{equation*}
\begin{cases}
\F^\pm(D^2 w^\pm) = 0 & \hbox{in}~\Omega_{3/4}^\pm\\
w = \frac{1}{2}(v^++v^-) & \hbox{on}~\Gamma_{3/4}\\
 w=v & \hbox{on}~\partial B_{3/4}.
\end{cases}
\end{equation*}
By Theorem~\ref{thm:regflat} (rescaled), we know that $v^\pm \in C^{1,\alpha}(\overline{\Omega_{3/4-\eta}^\pm})$, for any $0<\eta<3/4$,  and
\begin{equation} \label{eq:estv}
\eta \|\nabla v^\pm\|_{L^\infty(\Omega_{3/4-\eta}^\pm)}+\eta^{1+\alpha} [\nabla v^\pm]_{C^{0,\alpha}(\overline{\Omega_{3/4-\eta}^\pm})} \leq C_3,
\end{equation}
for some $0<\alpha < \bar \alpha$. Then, by pointwise boundary $C^{1,\alpha}$ estimates,
we have $w^\pm\in C^{1,\alpha}(\overline{\Omega_{3/4-\eta}^\pm})$, and we will see that
$w_\nu^+-w_\nu^- \approx 1$ on  $\Gamma_{3/4-\eta}.$
Indeed, for any $x\in \Gamma_{3/4-\eta}$,  we have
\begin{align*}
w_\nu^+(x) -w_\nu^-(x) - 1 &=\big( (w-v)_\nu^+(x)\big) - \big((w-v)_\nu^-(x) \big) + \big( v_\nu^+(x)-v_\nu^-(x)-1\big) \\
&\equiv g_1+g_2+g_3.
\end{align*}
We will show that $g_1$, $g_2$, and $g_3$ are small in terms of $\vep$ and $\delta$. For $g_3$, it holds that
\begin{align*}
|g_3(x)| & \leq |v_\nu^+(x)-v_{x_n}^+(x)| + |v_{x_n}^+(x)-v_{x_n}^-(x)-1|+ |v_{x_n}^-(x)-v_\nu^-(x)| \equiv {\rm I+II+III}.
\end{align*}
By Corollary~\ref{cor:stabilityflat}, it follows that ${\rm II}\leq \eta^{-(1+\gamma)}\vep$. Moreover, since $\|\nabla'\psi\|_{L^\infty(B_1')}\leq \delta$, we get
\begin{align*}
{\rm I} &\leq |\nabla v^+(x)||\nu(x) - e_n| \leq C \|\nabla v^+\|_{L^\infty(\Omega_{3/4-\eta}^+)}|\nabla'\psi(x')| \leq \frac{C_3}{\eta} \delta.
\end{align*}
Similarly for III. For $g_1$ and $g_2$ we consider $w-v$. Since $w^\pm-v^\pm \in {S}_{\lambda/n, \Lambda}(0)$ in $\Omega_{3/4}^\pm$,
by the classical ABP estimate, and Corollary~\ref{cor:stabilityflat}, we get
$$
\|w^\pm-v^\pm\|_{L^\infty(\Omega_{3/4}^\pm)} \leq \|w^\pm-v^\pm\|_{L^\infty(\Gamma_{3/4})} = \| \overline{v} - \underline{v}\|_{L^\infty(\Gamma_{3/4})} \leq \vep.
$$
Hence,
$
\|w-v\|_{L^\infty(B_{3/4})} \leq \vep.
$
Moreover, pointwise boundary $C^{1,\alpha}$ estimates give
$$
|g_1(x)| \leq \frac{C}{\eta} \big( \|w^+-v^+\|_{L^\infty(\Omega_{3/4}^+)}
+\| \overline{v} - \underline{v}\|_{C^{1,\alpha}(\overline{\Gamma_{3/4-\eta}})} \big),
$$
where $\|v\|_{C^{1,\alpha}(\overline{\Gamma_{3/4-\eta}})}= \|v\|_{L^\infty(\Gamma_{3/4-\eta})}+\|\nabla'v+v_{x_n}\nabla'\psi\|_{C^{0,\alpha}(\overline{\Gamma_{3/4-\eta}})}$. By Corollary~\ref{cor:stabilityflat}, and estimate \eqref{eq:estv}, it follows that
\begin{align*}
\| \overline{v} - \underline{v}\|_{C^{1,\alpha}(\overline{\Gamma_{3/4-\eta}})} 
& \leq \| \overline{v} - \underline{v}\|_{L^\infty(B_{3/4})}
+\|\nabla'\overline{v}-\nabla'\underline{v}\|_{C^{0,\alpha}(\overline{B_{3/4-\eta}})}\\
&\quad+\|( \overline{v}_{x_n}-\underline{v}_{x_n})\nabla'\psi\|_{C^{0,\alpha}(\overline{\Gamma_{3/4-\eta}})}  \\
&\leq \vep + \eta^{-(1+\gamma)}\vep +2\| \overline{v}_{x_n}-\underline{v}_{x_n}\|_{C^{0,\alpha}(\overline{\Gamma_{3/4-\eta}})} \|\nabla'\psi\|_{C^{0,\alpha}(\overline{B_1'})} \\
&\leq \vep + \eta^{-(1+\gamma)}\vep + 4C_3\eta^{-(1+\gamma)} \delta.
\end{align*}
Therefore, $|g_1(x)|\leq C \eta^{-1}(\vep +\eta^{-(1+\gamma)}(\vep+\delta))\leq C\eta^{-(2+\gamma)}\vep$. Similarly for $g_2$.

Next, $u-w$ satisfies
\begin{equation*}
\begin{cases}
u-w \in {S}_{\lambda/n, \Lambda}(f^\pm) & \hbox{in}~\Omega_{3/4}^\pm\\
(u-w)_\nu^+- (u-w)_\nu^-= (g-1)-(g_1+g_2+g_3) & \hbox{on}~\Gamma_{3/4-\eta}\\
 u-w =0 & \hbox{on}~\partial B_{3/4}.
\end{cases}
\end{equation*}
Therefore, by Theorem~\ref{thm:ABP2} applied to $u-w$ in $B_{3/4-\eta}$, we get
\begin{align*}
\|u-w\|_{L^\infty(B_{3/4-\eta})} &\leq  \|u-v\|_{L^\infty(\partial B_{3/4-\eta})} + \|v-w\|_{L^\infty(\partial B_{3/4-\eta})} \\
&\quad+ C \Big(\|g-1\|_{L^\infty(\Gamma_{3/4})}+ \|f^-\|_{L^n(\Omega_{3/4}^-)}+ \|f^+\|_{L^n(\Omega^+_{3/4})}  \\
&\qquad\quad  + \|g_1\|_{L^\infty(\Gamma_{3/4-\eta})} + \|g_2\|_{L^\infty(\Gamma_{3/4-\eta})} +\|g_3\|_{L^\infty(\Gamma_{3/4-\eta})}\Big)\\
&\leq [u-v]_{C^{0,\beta}(\overline{B_{3/4}})} \eta^\beta+ \|v-w\|_{L^\infty(B_{3/4})} + C\delta +2C \eta^{-(2+\gamma)}\vep
 +C\eta^{-1} \delta \\
&\leq (C_1+C_2)\eta^\beta + \vep+ \tilde{C}\eta^{-(2+\gamma)}\vep.
\end{align*}
Choose $0<\eta<1/4$ such that $\eta < \min\{\vep^{\frac{1}{2(2+\gamma)}}, \vep^{\frac{1}{2\beta}}\}$.
We conclude that
\begin{align*}
\|u-v\|_{L^\infty(B_{1/2})} \leq \|u-w\|_{L^\infty(B_{3/4-\eta})} + \|w-v\|_{L^\infty(B_{3/4})} \leq C\vep^{1/2},
\end{align*}
where $C>0$ depends only on $n$, $\lambda$, $\Lambda$, and $\alpha$.
\end{proof}

%%%%%%%%%%%%%%%%%%%%%%%%%%%%%%%%%
\subsection{Stability for $C^{2,\alpha}$ interfaces} 
%%%%%%%%%%%%%%%%%%%%%%%%%%%%%%%%%

Our proofs of stability of solutions in this case does not need to distinguish
between whether $g$ is close to or away from $0$. 

\begin{lem} \label{lem:approx2}
Assume that $\Gamma=B_1 \cap\{x_n=\psi(x')\}$, for some $\psi \in C^2(B_1')$. Fix $D'\in \R^{n-1}$ and $b\in \R$.
Given $\vep>0$, there exists $\delta>0$ such that if $u \in C(B_1)$ is a viscosity solution to \eqref{eq:TP},
with $\|u\|_{L^\infty(B_1)}\leq 1$ and
$$\|\psi\|_{C^2(B_1')} + \|g-D'\cdot x'-b\|_{L^\infty(\Gamma)}+ \|f^-\|_{L^\infty(\Omega^-)} + \|f^+\|_{L^\infty(\Omega^+)} \leq \delta$$
then there exists a bounded viscosity solution $v\in C(B_{3/4})$ to
\begin{equation} \label{eq:approxv}
 \begin{cases}
 \F(D^2 v) = 0 & \hbox{in}~B_{3/4}^\pm\\
 v_{x_n}^+ - v_{x_n}^- = D'\cdot x' + b & \hbox{on}~B_{3/4}\cap \{x_n=0\}
 \end{cases}
 \end{equation}
such that
 $$
 \|u-v\|_{L^\infty(B_{3/4})} \leq \vep.
 $$
\end{lem}

\begin{proof}
Suppose the statement is false. Then there exists $\vep_0>0$ and a sequence of functions $u_k$,  $\psi_k$, $f_k^\pm$
and $g_k$ such that $u_k$ is a viscosity solution to
$$
\begin{cases}
\F (D^2 u_k) = f_k^\pm& \hbox{in}~\Omega_k^\pm\\
 (u_k^+)_{\nu} -  (u_k^-)_{\nu} = g_k & \hbox{on}~ \Gamma_k=B_1 \cap \{x_n=\psi_k(x')\},
 \end{cases}
 $$
 with $\|u_k\|_{L^\infty(B_1)}\leq 1$ and 
 \begin{equation}\label{eq:compactnessk}
 \|\psi_k\|_{C^2(B_1')} + \|g_k-D' \cdot x' - b\|_{L^\infty(\Gamma_k)}+ \|f_k^-\|_{L^\infty(\Omega_k^-)} + \|f_k^+\|_{L^\infty(\Omega_k^+)} \leq \tfrac{1}{k},
 \end{equation}
and such that, for any viscosity solution $v$ to \eqref{eq:approxv},
 \begin{equation}\label{eq:contradiction}
 \|u_k-v\|_{L^\infty(B_{3/4})} > \vep_0 \quad \hbox{for all}~k\geq1.
\end{equation}
By Theorem~\ref{thm:holder} and \eqref{eq:compactnessk},
$$
\|u_k\|_{C^{0,\alpha_1}(\overline{B_{3/4}})} \leq C ( \|u_k\|_{L^\infty(B_1)} + \|g_k\|_{L^\infty(\Gamma)}+ \|f_k^-\|_{L^\infty(\Omega_k^-)} + \|f_k^+\|_{L^\infty(\Omega_k^+)} ) \leq C_1, 
$$
for some $C_1>0$ independent of $k$, and for all $k\geq 1$. Hence, the sequence $\{u_k\}_{k=1}^\infty$ is bounded and equicontinuous on $\overline{B_{3/4}}$. By compactness, up to a subsequence,
$u_k \to u_\infty$ uniformly on compact subsets of ${B_{3/4}}$ as $k\to\infty$.
By Lemma~\ref{lem:closedness} with $B_{3/4}$ in place of $B_1$, we see that $u_\infty\in C(B_{3/4})$ is a viscosity solution to
\eqref{eq:approxv}. This is a contradiction with \eqref{eq:contradiction}.
\end{proof}

In the previous lemma we assumed that $u$ is continuous in $B_1$. 
In later proofs, we will also need a similar approximation result when $u$ has a jump discontinuity across the interface.
We prove this version in the next lemma. 

\begin{lem}  \label{lem:approx3}
Assume that $\Gamma=B_1 \cap\{x_n=\psi(x')\}$, for some $\psi \in C^2(B_1')$, and that $f^\pm$ satisfy
$$
\Big( \fint_{B_r(x_0) \cap \Omega^\pm} |f^\pm|^n\, dx\Big)^{1/n} \leq C_{f^\pm} r^{\alpha-1}
$$ 
for all $r>0$ small, $x_0\in \Omega^\pm \cup \Gamma$, and some $0<\alpha<1$.
Fix $D'\in \R^{n-1}$ and $b\in \R$.
Given $\vep>0$, there exists $\tilde \delta>0$ such that if $u\in C(B_1\setminus \Gamma)\cap L^\infty(B_1)$ satisfies \eqref{eq:TP}
in the viscosity sense, with $\|u\|_{L^\infty(B_1)}\leq 1$, $u^+ - u^-\equiv h \in C^2(\Gamma)$, and  
\begin{equation} \label{eq:assumptionu}
\|\psi\|_{C^2(B_1')} +\|h\|_{C^2(\Gamma)}+ \|g-D'\cdot x'-b\|_{L^\infty(\Gamma)}+ C_{f^-}+C_{f^+} \leq \tilde\delta,
\end{equation}
then there exists a bounded viscosity solution $v\in C(B_{1/2})$ to
\begin{equation} \label{eq:approxv2}
 \begin{cases}
 \F(D^2 v) = 0 & \hbox{in}~B_{1/2}^\pm\\
 v_{x_n}^+ - v_{x_n}^- = D'\cdot x' + b & \hbox{on}~B_{1/2}\cap \{x_n=0\}
 \end{cases}
 \end{equation}
such that
 $$
 \|u-v\|_{L^\infty(B_{1/2})} \leq \vep.
 $$
\end{lem}

\begin{proof}
Given $\vep>0$, let $\delta>0$ be the one given in Lemma~\ref{lem:approx2}. Fix $\tilde \delta>0$ to be chosen.
Let $w\in C(B_{7/8})$ be the viscosity solution to the Dirichlet problems
\begin{equation*}
\begin{cases}
\F (D^2 w) = 0 & \hbox{in}~\Omega_{7/8}^\pm\\
w = \frac{1}{2}(u^++u^-) & \hbox{on}~\Gamma_{7/8}\\
 w=u & \hbox{on}~\partial B_{7/8}.
\end{cases}
\end{equation*}
We will prove that $w$ satisfies the assumptions of Lemma~\ref{lem:approx2} with $B_{3/4}$ in place of $B_1$.
Indeed, by the maximum principle, $\|w\|_{L^\infty(B_{7/8})} \leq \| u\|_{L^\infty(B_1)}\leq 1$.
Moreover,  $u^\pm - w^\pm \in S_{\lambda/n,\Lambda}(f^\pm)$ in $ \Omega_{7/8}^\pm$, $u^\pm - w^\pm = \pm \tfrac{h}{2}$ on $ \Gamma_{7/8}$, and $u^\pm - w^\pm =0$ on $\partial \Omega_{7/8}^\pm \setminus  \Gamma_{7/8}.$
 Then by the classical ABP estimate and \eqref{eq:assumptionu},
\begin{equation} \label{eq:smalldiff2}
\|u^\pm - w^\pm \|_{L^\infty(\Omega_{7/8}^\pm)} \leq \tfrac{1}{2} \|h \|_{L^\infty(\Gamma)} + C \|f^\pm\|_{L^n(\Omega^\pm)} \leq C \tilde \delta,
\end{equation}
for some $C>0$, depending only on $n$, $\lambda$ and $\Lambda$. 
Since $h \in C^2(\Gamma)$, by  boundary pointwise $C^{1,\alpha}$ estimates and by using
\eqref{eq:assumptionu} and \eqref{eq:smalldiff2}, for any $x_0\in \Gamma_{3/4}$, we have
\begin{align} \label{eq:regdiff2}
|\nabla(u^\pm - w^\pm)(x_0)| & \leq {C} \big (\|u^\pm - w^\pm \|_{L^\infty(\Omega_{7/8}^\pm)} + \|h \|_{C^{1,\alpha}(x_0)} + 
C_{f^\pm} \big)  \leq  C \tilde \delta.
\end{align}
Suppose there exists a test function $\varphi$ touching $w$ by above at $x_0$ in a small neighborhood of $x_0$ contained in $B_{3/4}$.
In particular, $\phi = \varphi - (w-u)$ is a test function that touches $u$  by above at $x_0$. 
Therefore, 
$
\phi_{\nu}^+(x_0) - \phi_{\nu}^-(x_0) \geq g(x_0).
$
It follows that:
$$
\varphi_{\nu}^+(x_0) - \varphi_{\nu}^-(x_0) \geq g(x_0) +  (w^+-u^+)_{\nu}(x_0)-(w^--u^-)_{\nu}(x_0) \equiv \tilde g(x_0).
$$
Similarly if $\varphi$ is a test function touching $w$ from below at $x_0$.
Hence, $w_{\nu}^+ - w_{\nu}^- = \tilde g$ on $\Gamma_{3/4}$ in the viscosity sense.
Furthermore,  by \eqref{eq:assumptionu} and \eqref{eq:regdiff2}, we get
\begin{align*}
\|\tilde g-D'\cdot x'-b\|_{L^\infty(\Gamma_{3/4})}
 &\leq  \| g - D' \cdot x - b\|_{L^\infty(\Gamma_{3/4})} 
  + 2 \| \nabla (w^\pm-u^\pm)\|_{L^\infty(\Gamma_{3/4})}
 \leq \tilde \delta + 2C \tilde \delta
\leq \delta
\end{align*}
if we choose $\tilde\delta$ small enough. Applying Lemma~\ref{lem:approx2} with $B_{3/4}$ in place of $B_1$, we see that there exists a bounded viscosity solution $v\in C(B_{1/2})$ of \eqref{eq:approxv2} such that
$
\| w-v\|_{L^\infty(B_{1/2})} \leq \vep/2.
$
By the previous estimate and \eqref{eq:smalldiff2}, we conclude that 
$$\| u-v\|_{L^\infty(B_{1/2})} \leq \| u-w \|_{L^\infty(B_{1/2})} + \| w-v\|_{L^\infty(B_{1/2})} \leq  \vep/2 +\vep/2=\vep. $$
\end{proof}

%%%%%%%%%%%%%%%%%%%%%%%%%%%%%%%%%%%%%%%%%%%%%
\section{$C^{1,\alpha}$ regularity: proof of Theorem \ref{thm:main1}}\label{section:C1alpha}
%%%%%%%%%%%%%%%%%%%%%%%%%%%%%%%%%%%%%%%%%%%%%

Fix $0<\alpha<\bar\alpha$, where $0<\bar\alpha<1$ is given at the end of the Introduction.
In this section we derive pointwise $C^{1,\alpha}$ estimates for viscosity solutions to \eqref{eq:TP} up to the interface $\Gamma$.
The approximation and stability lemmas from Section~\ref{sec:approxlem} will be key ingredients.
The $C^{1,\alpha}$ regularity estimates of $u^+$ and $u^-$ up to the interface (Theorem~\ref{thm:main1})
follow by patching the classical interior estimates and the boundary estimates as usual.
Thus, we only need to prove the following result.
 
\begin{thm}[Boundary pointwise $C^{1,\alpha}$ regularity]\label{thm:pointwisereg}
Assume that $0\in \Gamma$, $\psi \in C^{1,\alpha}(0)$,  $g\in C^{0,\alpha}(0)$, and $f^\pm$ satisfy 
$$
\Big( \fint_{B_r \cap \Omega^\pm} |f^\pm|^n\, dx\Big)^{1/n} \leq C_{f^\pm} r^{\alpha-1}
 \quad \hbox{for all}~ r>0~\hbox{small}.
$$ 
Suppose that $u$ is a bounded viscosity solution to \eqref{eq:TP}, with $\|u\|_{L^\infty(B_1)}\leq 1$.
If $g(0)=0$, assume further that \eqref{eq:closeness2} holds, where $\theta>0$ is given in Lemma~\ref{lem:approx1}.
Then $u^\pm \in C^{1,\alpha}(0)$, that is, there exist affine functions $l^\pm(x)=A^\pm \cdot x + b$ such that
$$
|u^\pm(x) - l^\pm(x)| \leq C |x|^{1+\alpha}, \qquad \hbox{for all}~x\in \Omega_{r_0}^\pm,
$$
 with $r_0=C_0 / \|\psi\|_{C^{1,\alpha}(0)}$, and $C_0>0$ depending only on $n$, $\lambda$, $\Lambda$ and $\alpha$. Moreover,
 $$
 |A^-|+|A^+|+|b| + |C| \leq  C_0 \|\psi\|_{C^{1,\alpha}(0)} \big( |g(0)| + [g]_{C^{0,\alpha}(0)} +C_{f^-}+C_{f^+}\big).
 $$
\end{thm}

The theorem will follow from iterating the next two lemmas.

\begin{lem} \label{lem:case1}
Given $0<\alpha<\bar\alpha$, there exist $C_0>0$, $0<\delta<1$, and $0<\rho< 1/2$ depending only on $n$, $\lambda$, $\Lambda$, and $\alpha$,
such that for any viscosity solution $u\in C(B_1)$ of \eqref{eq:TP} with $\|u\|_{L^\infty(B_1)}\leq 1$
and $\|g\|_{L^\infty(\Gamma_{3/4})} +C_{f^-}+C_{f^+} \leq \delta,$ if \eqref{eq:closeness2} holds, where
$\theta>0$ is as in Lemma~\ref{lem:approx1}, then
there is an affine function $l(x)=A\cdot x + b$, with $|A|+|b|\leq C_0$, such that
\begin{align*}
\| u - l \|_{L^\infty(B_\rho)} &\leq \rho^{1+\alpha}. 
\end{align*} 
\end{lem}

\begin{proof}
Fix $0<\tau <1/4$ and $0<\delta<1$ to be chosen.
Let $v\in C^{1,\gamma}_{\mathrm{loc}}(B_{3/4})$ be the function given in Lemma~\ref{lem:approx1}, with $\gamma=\bar \alpha -\epsilon$, for some $\epsilon>0$ sufficiently small so that $\bar\alpha - \epsilon > \alpha$. Then
$$
\| u - v \|_{L^\infty(B_{3/4-\tau})} \leq C(\tau^\beta+\delta).
$$
Moreover, if
$l(x)=v(0)+\nabla v (0)\cdot x$,
then $|\nabla l|+|l(0)|\leq C_0$, and the following estimate holds:
\begin{align*}
\| v - l \|_{L^\infty(B_\rho)} &\leq C_0 \rho^{1+\bar\alpha-\epsilon},
\end{align*}
 for any $0<\rho < 1/2$. It follows  that
 \begin{align*}
 \|u - l \|_{L^\infty(B_\rho)} &\leq \|u - v\|_{L^\infty(B_\rho)} + \| v - l \|_{L^\infty(B_\rho)}
 \leq   C(\tau^\beta+ \delta)  + C_0 \rho^{1+\bar\alpha-\epsilon}.
 \end{align*}
First, choose $\rho$ small enough such that $C_0 \rho^{1+\bar\alpha-\epsilon} \leq \rho^{1+\alpha}/3$. Then choose $\tau$ and $\delta$ such that $C\tau^\beta \leq \rho^{1+\alpha}/3$ and $C\delta \leq \rho^{1+\alpha}/3$.
\end{proof}

\begin{lem} \label{lem:approx}
Given $0<\alpha<\bar\alpha$, there exist constants $C_0>0$, $0<\rho<1/2$, and $0<\delta <\rho$,
depending only on $n$, $\lambda$, $\Lambda$ and $\alpha$, such that for any viscosity solution $u$ to \eqref{eq:TP}
with $\|u\|_{L^\infty(B_1)}\leq 1$ and $\|\psi\|_{C^{1,\alpha}(\overline{B_1'})}+ \|g-1\|_{L^\infty(\Gamma)} + C_{f^-}+C_{f^+} \leq \delta$,
 there exist affine functions $l^\pm(x)=A^\pm \cdot x + b$, with $|A^-|+|A^+|+|b| \leq C_0$, such that
\begin{align*}
\| u^\pm - l^\pm \|_{L^\infty(\Omega_\rho^\pm)} \leq  \rho^{1+\alpha}.
\end{align*}
Moreover, $\nabla ' l^- = \nabla' l^+$ and $l^+_{x_n} - l^-_{x_n} =1$. 
\end{lem}

\begin{proof}
Fix $0<\vep, \delta, \rho<1/2$ to be chosen. Let $v$ be as in Lemma~\ref{lem:stability}. Define 
$$l_v^\pm(x)=\nabla v^\pm(0)\cdot x+v^\pm(0),$$
 where $v^\pm = v\big|_{\overline{\Omega_{3/4}^\pm}}$. 
 By construction, $v^+= \underline v$ and $v^-= \overline v$, where $\underline v$ and $\overline v$ satisfy \eqref{eq:flata}, replacing $B_1$ by $B_{3/4}$, with $a=-\delta$ and $a=\delta$, respectively. Since $g_{\delta}$ is smooth, by Theorem~\ref{thm:regflat} (rescaled), we have $\overline v\in C^{1,\gamma}\big(\overline{B_{1/2}} \cap \{x_n \leq  \delta\}\big)$, for any $0<\gamma<\bar\alpha$, with
 $$
\| \overline v\|_{C^{1,\gamma}\big(\overline{B_{1/2}} \cap \{x_n \leq  \delta\}\big)} \leq C \big (\|u\|_{L^\infty(B_1)} + \|g_\delta\|_{C^{0,\gamma}(T_\delta)}\big) \leq C_0,
 $$
 where $C_0>0$ depends only on $n$, $\lambda$, $\Lambda$, and $\gamma$. 
 Similarly, $\underline v \in C^{1,\gamma}\big(\overline{B_{1/2}} \cap \{x_n \geq -\delta\}\big)$ and $\| \overline v\|_{C^{1,\gamma}\big(\overline{B_{1/2}}\cap \{x_n \geq - \delta\}\big)} \leq C_0$. 
Since $\|\psi\|_{C^{1,\alpha}(\overline{B_1'}) }\leq \delta$, then $\overline{\Omega_{1/2}} \subseteq \overline{B_{1/2}} \cap \{|x_n|\leq \delta\}.$
Therefore, $v^\pm \in C^{1,\gamma}(\overline{\Omega_{1/2}^\pm})$, with
$\|v^\pm\|_{C^{1,\gamma}(\overline{\Omega_{1/2}^\pm})} \leq C_0$. In particular, 
$|\nabla v^\pm (0)|+|v^\pm(0)|\leq C_0.$

We define the affine functions $l^\pm$ as $l^\pm_v$ plus a small correction, that is, 
$$
l^\pm (x) = l_v^\pm(x) + l^\pm_\vep(x)
$$
with $l^\pm_\vep(x) = A_\vep^\pm \cdot x+ b_\vep^\pm$ such that
\begin{align*}
b_\vep^+&=-b_\vep^-=\tfrac{1}{2}(v^-(0)-v^+(0))\\
(A_\vep^+)' &=-(A_\vep^-)'=\tfrac{1}{2}(\nabla'v^-(0)-\nabla'v^+(0)) \\
(A_\vep^+)_n &=-(A_\vep^-)_n=\tfrac{1}{2}(1-v_{x_n}^+(0)+v_{x_n}^-(0)). 
\end{align*}
By Corollary~\ref{cor:stabilityflat}, we have that $|b_\vep^\pm| + |A_\vep^\pm|\leq \vep.$
Moreover, by definition of $l^\pm$, it holds that
$$
 l^-(0)=l^+(0), \quad \nabla' l^- = \nabla' l^+, \quad \hbox{and} \quad l^+_{x_n}-l^-_{x_n}=1.
 $$

For any $x\in \Omega^\pm_{\rho}$, by Lemma~\ref{lem:stability} and the $C^{1,\gamma}$ estimate for $v^\pm$, we have that
\begin{align*}
|u^\pm (x)-l^\pm(x)| &=  |u^\pm (x)-l_v^\pm(x)-l_\vep^\pm(x)|\\ 
& \leq | u^\pm(x) - v^\pm(x)| + |v^\pm(x)-l_v^\pm(x)| + |l_\vep^\pm(x)|\\
&\leq C\vep^{1/2}+ C_0 \rho^{1+\gamma} + \vep \rho + \vep.
\end{align*}
Fix $\alpha<\gamma = \gamma(\alpha, \bar \alpha)<\bar\alpha$. First, choose $0<\rho<1/2$ such that
$C_0 \rho^{1+\gamma} \leq {\rho^{1+\alpha}}/{2}.$
Then choose $0<\vep<\rho$ such that 
$C\vep^{1/2} + \vep \rho + \vep \leq {\rho^{1+\alpha}}/{2}.$
Finally, recall that $0<\delta<\vep$ is given as in Corollary~\ref{cor:stabilityflat}. Therefore, we conclude that
$\| u^\pm - l^\pm \|_{L^\infty(\Omega_\rho^\pm)} \leq  \rho^{1+\alpha}$.
\end{proof}

\begin{proof}[Proof of Theorem~\ref{thm:pointwisereg}]
Fix $0<\alpha<\bar\alpha$. Let $C_0, \rho, \delta>0$ be the minimum of the constants given in Lemma~\ref{lem:case1} and Lemma~\ref{lem:approx}. 
Let $0<\delta_0<\min \big \{\delta/3,(2+2\tilde C)^{-1} \delta,(2\tilde C)^{-1} \rho^{1+\alpha}\big\}$.
First, we normalize the problem. Recall that we are assuming  $0\in\Gamma$, that is, $\psi(0')=0$.\medskip

$(i)$ After a rotation, we can assume that $\nu(0)=e_n$. In particular, $\nabla'\psi(0')=0'$. Also, we can suppose that $[\psi]_{C^{1,\alpha}(0)} \leq \delta_0$, where
$$
[\psi]_{C^{1,\alpha}(0)}  =  \sup_{x'\in B_1', \, x'\neq 0'} \frac{|\nabla' \psi(x')|}{|x'|^{\alpha}}.
$$
Indeed, let $K^\alpha=[\psi]_{C^{1,\alpha}(0)}/\delta_0$, and consider $v(y)=u(y/K)$, for $y\in B_1$. Then $v$ satisfies 
$$
\begin{cases}
\F_K^\pm(D^2v) =  f_K^\pm &  \hbox{in}~ \tilde \Omega^\pm\\
v_\nu^+- v_\nu^- = g_K & \hbox{on}~\tilde \Gamma,
\end{cases}
$$
where $\F_K^\pm(M) = K^{-2} \F^\pm(K^2 M)$, for  $M\in \mathcal{S}^n$, $\tilde \Omega^\pm = \{ y \in B_1 : y/K \in \Omega^\pm\}$,  $\tilde \Gamma = \{ y \in B_1 : y/K \in \Gamma\}$,  $f_K^\pm(y)=K^{-2} f^\pm(y/K)$, for $y\in \tilde \Omega^\pm $,  and $g_K(y)=K^{-1} g(y/K)$, for $y\in \tilde \Gamma$.
In particular, it holds that $\F_K^\pm \in \mathcal{E}(\lambda,\Lambda)$, that is, $\F_K^\pm$ is a fully nonlinear uniformly elliptic operator with the same ellipticity constants as $\F^\pm$. Also, $f^\pm_K$ satisfy
$$\Big( \fint_{B_r \cap \tilde\Omega^\pm} |f_K^\pm(y)|^n\, dy \Big)^{1/n} 
\leq K^{-(1+\alpha)} C_{f^\pm} r^{\alpha-1}.$$
Hence, $C_{f^\pm_K} = K^{-(1+\alpha)} C_{f^\pm}$. Moreover, $g_K$ satisfies
$$[g_K]_{C^{0,\alpha}(0)} = \sup_{y \in B_1, \, y\neq 0} \frac{|g_K(y)-g_K(0)|}{|y|^\alpha}
\leq K^{-(1+\alpha)} [g]_{C^{1,\alpha}(0)}.$$
If $y\in \tilde\Gamma$, then $y_n=\tilde \psi(y')$, with $\tilde{\psi}(y') = K \psi (y'/K).$
Moreover, 
\begin{align*}
[\tilde \psi]_{C^{1,\alpha}(0)} &=  \sup_{y'\in B_1', \, y'\neq 0'} \frac{| \nabla'\tilde \psi(y')|}{|y'|^{\alpha}}
=  \sup_{y'\in B_1', \, y'\neq 0'} \frac{| \nabla' \psi(y'/K)|}{|y'|^{\alpha}}
\leq K^{-\alpha} [\psi]_{C^{1,\alpha}(0)}= \delta_0.
\end{align*}

$(ii)$ Assume that $\| u \|_{L^\infty(B_1)}\leq 1$, $ C_{f^-} + C_{f^+} \leq \delta_0$, and 
$$ 
[g]_{C^{0,\alpha}(0)} = \sup_{x\in \Gamma\cap B_1,\,  x\neq 0} \frac{|g(x)-g(0)|}{|x|^\alpha} \leq {\delta_0}.
$$
Indeed, let $K=\| u \|_{L^\infty(B_1)}+ \delta_0^{-1}([g]_{C^{0,\alpha}(0)} +C_{f^-} + C_{f^+})$, and consider 
$v =u/K$. Then $v$ satisfies
$$
\begin{cases}
\F_K^\pm(D^2v) =  f_K^\pm &  \hbox{in}~ \Omega^\pm\\
v_\nu^+- v_\nu^- = g_K & \hbox{on}~\Gamma,
\end{cases}
$$
where $\F_K^\pm(M) = K^{-1} \F^\pm(K M)$, for  $M\in \mathcal{S}^n$, $f_K^\pm=K^{-1} f^\pm$, and $g_K=K^{-1} g$.
Moreover, $\|v\|_{L^\infty(B_1)}\leq1 $, and $[g_K]_{C^{0,\alpha}(0)} + C_{f_K^-} + C_{f_K^+} \leq \delta_0$.

For simplicity, we use the same notation as in the statement, that is, $\psi$, $u$, $F^{\pm}$, $f^\pm$ and $g$. 

The proof is different when the value of $g(0)=0$ or $g(0)\neq0$.
We will start with the case $g(0)\neq 0$, and the case $g(0)=0$ will be addressed at the end. For convenience, we set $g(0)=1$.
Under these assumptions, it is enough to prove the following:

\medskip

\noindent{\bf Claim.} {\it For every $k \geq 1$, there exist affine functions $l_k^\pm(x)=A_k^\pm \cdot x + b_k$ such that
\begin{align*}
\rho^{k} |A^\pm_{k+1}-A^\pm_{k}|  + |b_{k+1}-b_{k}|  \leq C_0 \rho^{k(1+\alpha)}
\end{align*}
where $C_0>0$ depends only on $n$, $\lambda$, $\Lambda$ and $\alpha$, and such that
\begin{align*}
\|u^\pm-l_k^\pm \|_{L^\infty( \Omega^\pm_{\rho^k})} &\leq \rho^{k(1+\alpha)}. 
\end{align*}
Moreover, $\nabla'l^-_k=\nabla'l_k^+$, and  $(l_k^+)_{x_n}-(l_k^-)_{x_n}=1$.\medskip}

We prove the claim by induction. For $k=1$, by the normalization, we are under the assumptions of Lemma~\ref{lem:approx}. Indeed, by $(i)$, we have that
\begin{align*}
\|\psi\|_{C^{1,\alpha}(\overline{B_1'})} &= \|\psi-\psi(0')\|_{L^\infty(B_1')} + \|\nabla \psi-\nabla'\psi(0')\|_{L^\infty(B_1')} + [\nabla\psi]_{C^{0,\alpha}(\overline{B_1'})}\\
&\leq 3[\psi]_{C^{1,\alpha}(0)} \leq 3 \delta_0  \leq \delta.
\end{align*}
Moreover, by $(ii)$ and $(iii)$, it follows that 
$$
\|g-1\|_{L^\infty(\Gamma)} = \|g-g(0)\|_{L^\infty(\Gamma)} \leq [g]_{C^{0,\alpha}(0)} \leq \delta_0 \leq \delta.
$$
Hence, by Lemma~\ref{lem:approx}, there exist $l_1^\pm (x) = A_1^\pm \cdot x + b_1$, with $|A_1^-|+|A_1^+|+|b_1| \leq C_0$ such that
\begin{align*}
\|u^\pm -l_1^\pm \|_{L^\infty(\Omega_\rho^\pm)} &\leq \rho^{1+\alpha}.
\end{align*}
Moreover, $\nabla' l_1^-=\nabla'l_1^+$, and $(l_1^+)_{x_n} - (l_1^-)_{x_n} = 1$.

For the induction step, assume that the claim holds for some $k\geq 1$, and let $l^\pm_k$ be such affine functions. Denote by 
$$\tilde {\Omega}^\pm_{k} =\{ x\in B_1 : \rho^k x\in \Omega^\pm\}\quad\hbox{and}\quad
\tilde \Gamma_{k}  =\{ x \in B_1 : \rho^k  x \in \Gamma \}.$$
Note that if  $\psi_{k}$ is a parametrization of $\tilde\Gamma_{k} $ in $B_1'$, then $\psi_{k}(x') = \rho^{-k}\psi(\rho^k x').$
In particular, $\nabla' \psi_{k} (x') =\nabla' \psi(\rho^k x)$, and thus, for $x\in \tilde\Gamma_{k}$, if $\nu_{k}(x)$ is the normal vector on $x$ pointing at $\tilde\Omega_{k}^+$, then  $\nu_{k}(x) = \nu(\rho^k x)$.
Define $l_k= l_k^+ \chi_{\tilde\Omega_{k}^+} + l_k^- \chi_{\tilde\Omega_{k}^-}$.
Consider the rescaled function
$$v(x) = \frac{u (\rho^k x) - {l}_k(\rho^k x)}{\rho^{k(1+\alpha)}}\qquad\hbox{for}~x\in \overline{B_1}.$$
Then $v$ satisfies 
\begin{equation} \label{eq:tpv}
\begin{cases}
\F_k^\pm(D^2 v) = f_k^\pm & \hbox{in}~\tilde\Omega_{k}^\pm\\
v_{\nu_k}^+ - v_{\nu_k}^- = g_k & \hbox{on}~\tilde\Gamma_{k}\\
\end{cases}
\end{equation}
in the viscosity sense, where 
\begin{align*}
\F_k^\pm(M) &= \rho^{k(1-\alpha)} \F^\pm( \rho^{k(\alpha-1)}M), \quad\hbox{for}~M\in \mathcal{S}^n\\
f_k^\pm(x) &= \rho^{k(1-\alpha)} f^\pm(\rho^k x), \quad \hbox{for}~x\in \tilde\Omega_{k}^\pm\\ 
g_k(x) & = \rho^{- k \alpha} (g(\rho^k x)-\nu_n(\rho^kx)), \quad \hbox{for}~x\in \tilde\Gamma_{k}.
\end{align*}
By the induction hypothesis, 
$\| v \|_{L^\infty(B_1)}  \leq 1.$
Notice that 
\begin{equation} \label{eq:cttf}
\Big( \fint_{B_r \cap \tilde\Omega_k^\pm} |f_k^\pm(y)|^n\, dy\Big)^{1/n} 
\leq C_{f^\pm} r^{\alpha-1}.
\end{equation}
Hence, $C_{f_k^\pm} = C_{f^\pm}$, and $C_{f_k^+}+C_{f_k^-}\leq \delta_0$.
Moreover,
\begin{equation} \label{eq:estg}
\|g_k\|_{L^\infty(\tilde\Gamma_k)} \leq [g]_{C^{0,\alpha}(0)}+ [\nu_n]_{C^{0,\alpha}(0)} \leq \delta_0+\delta_0=2\delta_0.
\end{equation}
However, we cannot apply Lemma~\ref{lem:case1} to $v$ since it has a jump discontinuity on $\tilde\Gamma_{k}$. In fact, if
$v^\pm  = v \big|_{\overline{\tilde \Omega^\pm_{k}}},$
then for $x\in \tilde\Gamma_{k}$, by the normalization $(i)$, and the induction hypothesis, we have
\begin{align} \label{eq:jump}
|(v^- - v^+)(x)| &= \frac{|l_k^-(\rho^k x)-l_k^+(\rho^k x)|}{\rho^{k(1+\alpha)}} = \rho^{-k \alpha} |x_n|
\leq   \rho^{-k \alpha} \sup_{x \in \tilde\Gamma_{k}} |x_n|\\\nonumber
 &\leq   \sup_{x'\in B_1'} \frac{|\psi_{k}(x')|}{ \rho^{k \alpha} } \leq [\psi]_{C^{1,\alpha}(0)} \leq  \delta_0. 
\end{align}
Let $w\in C(B_1)$ be the viscosity solution of the Dirichlet problems:
\begin{equation*}
\begin{cases}
\F_k^\pm (D^2 w) = 0 & \hbox{in}~\tilde\Omega_{k}^\pm\\
w = \frac{1}{2}(v^++v^-) & \hbox{on}~\tilde\Gamma_{k}\\
 w=v & \hbox{on}~\partial B_{1}.
\end{cases}
\end{equation*}
We will prove that $w$ satisfies the assumptions of Lemma~\ref{lem:case1}. By the maximum principle, $\|w\|_{L^\infty(B_1)} \leq \| v\|_{L^\infty(\partial B_1)}\leq 1$.
Moreover,  $v^\pm - w^\pm \in S_{\lambda/n,\Lambda}(f^\pm_k)$ in $\tilde \Omega_k^\pm$, $v^\pm - w^\pm = \pm \tfrac{1}{2} \rho^{-k\alpha} x_n$ on $\tilde \Gamma_k$, and $v^\pm - w^\pm =0$ on $\partial \tilde \Omega_k^\pm \setminus \tilde \Gamma_k.$
 Then, by the classical ABP estimate and \eqref{eq:jump},
\begin{equation} \label{eq:smalldiff}
\|v^\pm - w^\pm \|_{L^\infty(\tilde\Omega_k^\pm)} \leq \|v^\pm - w^\pm \|_{L^\infty(\tilde\Gamma_k)} + C \|f_k^\pm\|_{L^n(\tilde\Omega_k^\pm)} \leq C \delta_0.
\end{equation}
since $\|f_k^\pm\|_{L^n(\tilde\Omega_k^\pm)} \leq |B_1|^{1/n} C_{f^\pm}\leq C(n) \delta_0$ by \eqref{eq:cttf} with $r=1$.
By  boundary pointwise $C^{1,\alpha}$ estimates,  for any $x_0\in \tilde\Gamma_k\cap B_{3/4}$, we have
\begin{align} \label{eq:regdiff}
|\nabla(v^\pm - w^\pm)(x_0)| & \leq {C} \big (\|v^\pm - w^\pm \|_{L^\infty(\tilde\Omega_k^\pm)} + \tfrac{1}{2} \rho^{-k\alpha} \|\psi_k \|_{C^{1,\alpha}(x_0)} +C_{f_k^\pm} \big)  \leq \tilde C\delta_0,
\end{align}
where the last inequality follows from \eqref{eq:cttf},  \eqref{eq:smalldiff}, and the normalization $(i)$.

Let $x_0\in \tilde \Gamma_{k} \cap B_{3/4}$. Suppose there exists a test function $\varphi$ touching $w$ by above at $x_0$ in a small neighborhood of $x_0$ contained in $B_{3/4}$.
In particular, $\phi = \varphi - (w-v)$ is a test function that touches $v$  by above at $x_0$. 
Therefore, 
$
\phi_{\nu_k}^+(x_0) - \phi_{\nu_k}^-(x_0) \geq g_k(x_0).
$
It follows that:
$$
\varphi_{\nu_k}^+(x_0) - \varphi_{\nu_k}^-(x_0) \geq g_k(x_0) +  (w^+-v^+)_{\nu_k}(x_0)-(w^--v^-)_{\nu_k}(x_0) \equiv \tilde{g}_k(x_0)
$$
Moreover, by \eqref{eq:estg} and \eqref{eq:regdiff}, we get
\begin{align*}
\|\tilde g_k\|_{L^\infty(\tilde \Gamma_{k} \cap B_{3/4})} 
& \leq 2\delta_0 + 2\tilde C \delta_0 \leq \delta.
\end{align*}
Similarly, if $\varphi$ is a test function touching $w$ from below at $x_0$, in a small neighborhood of $x_0$ contained in $B_{3/4}$, then
$
\varphi_{\nu_k}^+(x_0) - \varphi_{\nu_k}^-(x_0) \leq \tilde g_k(x_0).
$
Hence, $w_{\nu_k}^+ - w_{\nu_k}^- = \tilde g_k$ on $\tilde\Gamma_{k} \cap B_{3/4}$ in the viscosity sense.
Applying Lemma~\ref{lem:case1} to $w$, we get that
there exist $C_0>0$ depending only on $n$, $\lambda$, $\Lambda$, and $\alpha$, and
 an affine function $l(x)=A\cdot x + b$, with $|A|+|b|\leq C_0$, such that
\begin{align} \label{eq:estw}
\| w - l \|_{L^\infty(B_\rho)} &\leq  \rho^{1+\alpha}/2. 
\end{align} 
Note that we can always choose $\rho$ sufficiently small such that the previous estimate holds (see proof of Lemma~\ref{lem:case1}).
Hence, by \eqref{eq:smalldiff} and \eqref{eq:estw}, we get
\begin{align*}
\| v - l \|_{L^\infty(B_\rho)} & \leq  \| v- w \|_{L^\infty(B_\rho)} + \| w - l \|_{L^\infty(B_\rho)} 
\leq \tilde C \delta_0 +\rho^{1+\alpha}/2 \leq \rho^{1+\alpha}.
\end{align*} 
In particular, for any $x\in B_\rho$, we have
\begin{align*}
\left|\frac{u (\rho^k x) - {l}_k(\rho^k x)}{\rho^{k(1+\alpha)}} - l(x) \right| \leq \rho^{1+\alpha},
\end{align*}
or equivalently, if $y = \rho^k x $, then for any $y \in B_{\rho^{k+1}}$, 
\begin{equation} \label{eq:indstep}
| u(y)  - l_k(y)  - \rho^{k(1+\alpha)} l(\rho^{-k}y)| \leq \rho^{(k+1)(1+\alpha)}.
\end{equation}
Define the affine approximations at the step $k+1$ as
$$
l_{k+1}^\pm (y) = l_k^\pm (y)  + \rho^{k(1+\alpha)} l(\rho^{-k}y).
$$ 
If $l_{k+1}^\pm (y) = A_{k+1}^\pm \cdot y+ b_{k+1}$, then
$A_{k+1}^\pm = A_k^\pm + \rho^{k\alpha}A$ and $b_{k+1} =  b_k  +  \rho^{k(1+\alpha)} b$.
Using the estimate $|A|+|b|\leq C_0$, we have 
\begin{align*}
\rho^{k} |A^\pm_{k+1}-A^\pm_{k}|  + |b_{k+1}-b_{k}|  \leq C_0 \rho^{k(1+\alpha)}.
\end{align*}
From \eqref{eq:indstep}, we see that 
$$
\| u^\pm - l_{k+1}^\pm \|_{L^\infty(\Omega^\pm_{\rho^{k+1}})} \leq \rho^{(k+1)(1+\alpha)}.
$$
Moreover, by the induction hypothesis, 
\begin{align*}
\nabla'l^-_{k+1}- \nabla'l_{k+1}^+  &=\nabla'l^-_{k}- \nabla'l_{k}^+=0\\
(l_{k+1}^+)_{x_n}-(l_{k+1}^-)_{x_n} &= (l_{k}^+)_{x_n}-(l_{k}^-)_{x_n}=1. 
\end{align*}
The proof of the claim is completed.

\medskip

Finally, we consider the case $g(0)=0$. As before, it is enough to prove the following: \medskip

\noindent{\bf Claim.} {\it For all $k \geq 1$, there exists an affine function $l_k=A_k \cdot x + b_k$  such that
\begin{align*}
\rho^{k} |A_{k+1}-A_{k}| + |b_{k+1}-b_{k}|  \leq C_0 \rho^{k(1+\alpha)},
\end{align*}
where $C_0>0$, depends only on $n$, $\lambda$, $\Lambda$ and $\alpha$, and such that
\begin{align*}
\|u-l_k\|_{L^\infty(B_{\rho^k})} &\leq \rho^{k(1+\alpha)}.
\end{align*}}

The proof is by induction. For $k=1$, we can apply Lemma~\ref{lem:case1} to $u$.  Indeed, $\| u\|_{L^\infty(B_1)}\leq 1$,  and $\|g\|_{L^\infty(\Gamma)}+C_{f^-}+C_{f^+} \leq \delta$, given that 
$$
\|g\|_{L^\infty(\Gamma)} =\sup_{x\in \Gamma}  |g(x)-g(0)|\leq [g]_{C^{0,\alpha}(0)}\leq \delta_0 \leq \delta.
$$
Then we find an affine function $l_1(x) = A_1\cdot x +b _1$, with $|A_1|+|b_1|\leq {C}_0$, such that
$$
\|u- l_1 \|_{L^\infty(B_\rho)} \leq \rho^{1+\alpha}.
$$
Assume the claim holds for $k\geq 1$. Define
$$
v(x) = \frac{u(\rho^k x)-l_k(\rho^k x)}{\rho^{k(1+\alpha)}} \qquad \hbox{for}~x\in B_1.
$$ 
Then, arguing as before, we have that $v\in C(B_1)$ satisfies \eqref{eq:tpv}, with the same operators $\F_k^\pm$ and the same right-hand sides $f_k^\pm$, but with different $g_k$:
\begin{align*}
g_k(x) = \rho^{- k \alpha} g(\rho^k x) \quad \hbox{for}~x\in \tilde\Gamma_{k}.
\end{align*}
In particular, for any $x\in \tilde\Gamma_{k}$, we have
$
|g_k(x)| = \rho^{- k \alpha} |g(\rho^k x)|  \leq [g]_{C^{0,\alpha}(0)} \leq \delta_0 \leq \delta$.
Then the claim follows for $k+1$ by applying again Lemma~\ref{lem:case1}.
\end{proof}

%%%%%%%%%%%%%%%%%%%%%%%%%%%%%%%%%%%%%%%%%%%%%
\section{$C^{2,\alpha}$ regularity: proof of Theorem \ref{thm:reg2}}\label{section:C2alpha}
%%%%%%%%%%%%%%%%%%%%%%%%%%%%%%%%%%%%%%%%%%%%%

Recall the definition of the H\"older exponent $0<\bar{\bar{\alpha}}<1$ given at the end of the Introduction.
Here we derive boundary pointwise $C^{2,\alpha}$ estimates for viscosity solutions of nonflat interface problems up to $\Gamma$ (Theorem \ref{thm:reg2}).
In this result we require that
\begin{equation} \label{eq:additional}
g(0)=0\qquad\hbox{or}\qquad D_{x'}^2 \psi(0')=0.
\end{equation}
We point out that, in fact, if $g(0)=0$ then we only need $\psi \in C^2$ instead of $\psi\in C^{2,\alpha}$. 
Observe that for boundary value problems, such as the Dirichlet and the oblique derivative problems studied in \cite{LiZ,LZ2},
an assumption like \eqref{eq:additional} is not needed because appropriate normalizations can be performed.
However, such normalizations are not possible for transmission problems when the interface is not flat.

Theorem \ref{thm:reg2} will be a consequence of iterating the next lemma.

\begin{lem} \label{lem:case2}
Given $0<\alpha<\bar{\bar\alpha}$, there exist constants $C_0>0$ and  $0<\vep, \rho<1/2$, depending only on $n$, $\lambda$, $\Lambda$, and $\alpha$ such that, if $u$ satisfies the assumptions of Lemma~\ref{lem:approx3} with $D'=\nabla' g(0)$ and $b=g(0)$, and $\|g\|_{C^1(\Gamma)}\leq 1$,
then there exist quadratic polynomials $P^\pm(x)= \tfrac{1}{2}x^t A^\pm x  +  B^\pm\cdot x + c$, where $A^\pm \in \mathcal{S}^n$, $B^\pm\in \Rn$,  $c\in \R$, and
 $ \|A^\pm\|+|B^\pm|+|c|\leq C_0$
  such that
\begin{align*}
\| u^\pm - P^\pm \|_{L^\infty(\Omega^\pm_\rho)} \leq  \rho^{2+\alpha}.
\end{align*}
Moreover, the coefficients satisfy the following conditions:
\begin{align*}
B_i^+-B_i^-= 0 \quad  \hbox{if}~i < n & \quad \hbox{and} \quad B_n^+-B_n^-= g(0) \\
A^+_{ij} - A^-_{ij} = 0 \quad \hbox{if}~i,j < n,  \quad A^+_{nj}-A^-_{nj} &= g_{x_j}(0)  \quad \hbox{if}~j<n, \quad \hbox{and} \quad  \F (A^\pm) =0.
\end{align*}
\end{lem}

\begin{proof}
Fix $0<\vep, \rho<1/2$ to be determined. By Lemma~\ref{lem:approx3}, there exists a bounded viscosity solution $v\in C(B_{1/2})$ to
\begin{equation}  \label{eq:problemforv}
 \begin{cases}
 \F(D^2 v) = 0 & \hbox{in}~B_{1/2}^\pm\\
 v_{x_n}^+ - v_{x_n}^- = \nabla' g(0) \cdot x' + g(0) & \hbox{on}~ \{x_n=0\} \cap B_{1/2}
 \end{cases}
 \end{equation}
such that 
\begin{equation} \label{eq:proof1}
\|u-v\|_{L^\infty(B_{1/2})} \leq \vep.
\end{equation}
Since $\nabla' g(0) \cdot x' + g(0)$ is smooth and the interface is flat,
by Theorem \ref{thm:regflat2}, we have that $v^\pm \in C^{2,\gamma}(\overline{B_{1/3}^\pm})$, for all $0<\gamma < \bar{\bar\alpha}$, and
\begin{align*}
\|v^\pm \|_{C^{2,\gamma}(\overline{B_{1/3}^\pm})} \leq C\big(\|v\|_{L^\infty(B_{1/2})} +  \|g\|_{C^{1}(\Gamma)} \big) \leq C_0,
\end{align*}
where $C_0>0$ only depends on $n$, $\lambda$, $\Lambda$ and $\gamma$. Let $A^\pm = D^2 v^\pm(0)$, $B^\pm = \nabla v^\pm (0)$, and $c=v(0)$. By the previous estimate, it follows that  $ \|A^\pm\|+|B^\pm|+|c|\leq C_0$ and
\begin{equation} \label{eq:proof2}
\|v^\pm - P^\pm\|_{L^\infty(B_\rho^\pm)} \leq [v^\pm]_{C^{2,\gamma}(\overline{B_\rho^\pm})} \rho^{2+\gamma} \leq C_0 \rho^{2+\gamma}. 
\end{equation}
Fix $\alpha<\gamma<\bar{\bar\alpha}$. First, choose $\rho<1/2$ such that $C_0 \rho^{2+\gamma} < \rho^{2+\alpha}/2$. Then choose $\vep < \rho^{2+\alpha}/2$. Combining \eqref{eq:proof1} and \eqref{eq:proof2}, we get
\begin{align*}
\|u^\pm-P^\pm\|_{L^\infty(\Omega^\pm_\rho)} \leq \|u-v\|_{L^\infty(B_{1/2})} + \|v^\pm - P^\pm\|_{L^\infty(B_\rho^\pm)} \leq \vep + C_0\rho^{2+\gamma}\leq \rho^{2+\alpha}.
\end{align*}
Moreover, since $v^+=v^-$ on $\{x_n=0\} \cap B_{1/2}$ and $v$ is $C^2$ up to the flat interface, we see that $\nabla ' v^+(0)= \nabla' v^-(0)$ and $D_{x'}^2 v^+(0) = D_{x'}^2 v^-(0)$. In particular, $B_i^+-B_i^-= 0$ and $A^+_{ij} - A^-_{ij} = 0$ if $i,j < n$. From the transmission condition in \eqref{eq:problemforv}, we get  
$ B_n^+-B_n^- = v^+_{x_n}(0) - v^-_{x_n}(0)=g(0)$, and $A^+_{nj} - A^-_{nj} =v^+_{x_nx_j}(0) - v^-_{x_nx_j}(0) = g_{x_j}(0)$ for all $j<n$. From the equation in \eqref{eq:problemforv} and by continuity of $\F$ and $D^2 v^\pm$ in a neighborhood of 0, we also have that $\F(A^\pm)=0$.
\end{proof}

\begin{proof}[Proof of Theorem~\ref{thm:reg2}]
Fix $0<\alpha<\bar{\bar\alpha}$. Let $C_0, \rho, \epsilon>0$ be the constants given in Lemma~\ref{lem:case2}. 
Let $\delta>0$ be the one given in Lemma~\ref{lem:approx3}.
Fix $\delta_0>0$ to be chosen sufficiently small.
Arguing similarly as in the proof of Theorem~\ref{thm:pointwisereg} and Theorem~\ref{thm:regflat2}, we may assume that we are under the following normalization conditions:
\begin{enumerate}[$(i)$]
\item $\psi(0')=0$, $\nabla'\psi(0')=0'$, $\|D_{x'}^2 \psi(0')\|\leq \delta_0$,  and $[\psi]_{C^{2,\alpha}(0)} \leq \delta_0$;
\item  $\| u \|_{L^\infty(B_1)}\leq 1$ and $\|g\|_{C^1(\Gamma)}\leq 1$;
\item $[g]_{C^{1,\alpha}(0)} + K_{f^-} + K_{f^+} \leq \delta_0$;
\item $f^+(0)=f^-(0)=0$.
\end{enumerate}
For the sake of clarity, we explain how to obtain the last two conditions in $(i)$. Let $K=\max\big\{ \big( [\psi]_{C^{2,\alpha}(0)}/ \delta_0\big)^{1/(1+\alpha)} , \|D_{x'}^2 \psi(0')\| / \delta_0 \big\}$
, and define $\tilde \Gamma= \{ y\in B_1 : y/K\in \Gamma\}.$
If $y\in \tilde\Gamma$, then $y_n=\tilde \psi(y')$, with $\tilde{\psi}(y') = K \psi (y'/K).$ We get
$\|D_{x'}^2\tilde \psi(0')\| =  \|D_{x'}^2 \psi(0')\|/K \leq \delta_0,$ and
$$
[\tilde \psi]_{C^{2,\alpha}(0)} =  \sup_{y'\in B_1', \, y'\neq 0'} \frac{\|D^2_{y'}\tilde \psi(y')-D^2_{y'}\tilde\psi(0)\|}{|y'|^{\alpha}}
\leq K^{-(1+\alpha)} [\psi]_{C^{2,\alpha}(0)}\leq \delta_0.
$$
The following property derived from $(i)$ will also be useful:
\begin{equation}\label{eq:hessianpsi}
\|D_{x'}^2\psi\|_{L^\infty(B_1')} \leq \sup_{x'\in B_1'} \|D_{x'}^2\psi(x') - D_{x'}^2\psi(0') \| + \|D_{x'}^2\psi(0') \| \leq 2 \delta_0.
\end{equation}

It is enough to prove the following: \medskip

\noindent{\bf Claim.} {\it For all $k \geq 1$, there exist quadratic polynomials $P_k^\pm(x)= \tfrac{1}{2} x^t A_k^\pm x  +  B_k^\pm\cdot x + c_k$, where $A_k^\pm \in \mathcal{S}^n$, $B_k^\pm\in \Rn$,  and $c_k\in \R$
such that
\begin{align} \label{eq:cauchy}
\rho^{2(k-1)} \|A^\pm_{k}-A^\pm_{k-1}\| + \rho^{k-1} |B_{k}-B_{k-1}| + |c_{k}-c_{k-1}|  \leq C_0 \rho^{(k-1)(2+\alpha)},
\end{align}
with $P_{0}\equiv 0$ and  $C_0>0$ depends only on $n$, $\lambda$, $\Lambda$, and $\alpha$, and such that
\begin{align*}
\| u^\pm - P_k^\pm \|_{L^\infty(\Omega^\pm_{\rho^k})} \leq  \rho^{k(2+\alpha)}.
\end{align*}
Moreover, the coefficients satisfy the following conditions:
\begin{align}\label{eq:cond1}
(B_k^+)_i-(B_k^-)_i= 0 \ \hbox{if}~i < n & \quad \hbox{and} \quad (B_k^+)_n-(B_k^-)_n= g(0) \\ \label{eq:cond2}
(A_k^+)_{ij} - (A_k^-)_{ij} = 0 \ \hbox{if}~i,j < n, & \ (A_k^+)_{nj}-(A_k^-)_{nj}= g_{x_j}(0) \ \hbox{if}~j<n, \ \hbox{and} \  \F (A_k^\pm) =0.
\end{align}}

We prove the claim by induction. In view of the assumption \eqref{eq:additional}, we divide the proof into two cases.

\smallskip

\noindent\textbf{Case 1: $D_{x'}^2 \psi(0')=0$.}
For $k=1$, by the normalization, we are under the assumptions of Lemma~\ref{lem:case2}. Indeed, by $(ii)$, we have $\| u \|_{L^\infty(B_1)}\leq 1$ and $\|g\|_{C^1(\Gamma)}\leq 1$. 
Also, $h=0$ on $\Gamma$, since $u$ is continuous in $B_1$.
By $(i)$, 
\begin{align*}
\|\psi\|_{C^2({B_1'})} &= \|\psi\|_{L^\infty(B_1')} + \|\nabla' \psi\|_{L^\infty(B_1')} + \|D_{x'}^2\psi\|_{L^\infty(B_1')} 
\leq 3[\psi]_{C^{2,\alpha}(0)} \leq 3 \delta_0 \leq \delta/2,
\end{align*}
choosing $\delta_0 \leq \delta/6$. By $(iii)$ and $(iv)$, for any $r<1$,
\begin{align*}
&\|g-\nabla' g(0)\cdot x'-g(0)\|_{L^\infty(\Gamma)} + C_{f^-} + C_{f^+} \leq [g]_{C^{1,\alpha}(0)} + rK_{f^-} + rK_{f^+}   \leq \delta_0 \leq \delta/2.
\end{align*}
Hence, there exist quadratic polynomials $P_1^\pm(x)= \tfrac{1}{2} x^t A_1^\pm x  +  B_1^\pm\cdot x + c_1$ satisfying \eqref{eq:cond1}, \eqref{eq:cond2},  and  $ \|A_1^\pm\|+|B_1^\pm|+|c_1|\leq C_0$, 
  such that
\begin{align*}
\| u^\pm - P_1^\pm \|_{L^\infty(\Omega^\pm_\rho)} \leq  \rho^{2+\alpha}.
\end{align*}

Assume the claim holds for some $k\geq 1$ and let $P^\pm_k$ be such quadratic polynomials.
Denote by 
$\tilde {\Omega}^\pm_{k} =\{ x\in B_1 : \rho^k x\in \Omega^\pm\}$ and
$\tilde \Gamma_{k}  =\{ x \in B_1 : \rho^k  x \in \Gamma \}= \{ x\in B_1 : x_n=\psi_k(x')\}$, where $\psi_{k}(x') = \rho^{-k}\psi(\rho^k x').$ In particular, $\nabla' \psi_{k} (x') =\nabla' \psi(\rho^k x)$ and $\nu_{k}(x) = \nu(\rho^k x)$. Furthermore,
\begin{equation}\label{eq:assumptionpsi}
\|\psi_k\|_{C^2(B_1')} \leq \rho^{-k} \|\psi\|_{C^2(B_{\rho^k}')} \leq 3\rho^{k(1+\alpha)} [\psi]_{C^{2,\alpha}(0)} \leq 3\delta_0 \leq \delta/5,
\end{equation}
choosing $\delta_0\leq \delta/15$.
Define $P_k= P_k^+ \chi_{\tilde\Omega_{k}^+} + P_k^- \chi_{\tilde\Omega_{k}^-}$, 
and consider the rescaled function:
\begin{align}\label{eq:rescale2}
v(x) = \frac{u (\rho^k x) - P_k(\rho^k x)}{\rho^{k(2+\alpha)}}\qquad\hbox{for}~x\in \overline{B_1}.
\end{align}
By the induction hypothesis, $\| v \|_{L^\infty(B_1)}  \leq 1$, and $v$ satisfies
$$\begin{cases}
\F_k(D^2 v) = f_k^\pm & \hbox{in}~\tilde\Omega_{k}^\pm\\
v_{\nu_k}^+ - v_{\nu_k}^- = g_k & \hbox{on}~\tilde\Gamma_{k}\\
\end{cases}$$
in the viscosity sense, where 
$\F_k(M) = \rho^{-k\alpha} \F( \rho^{k\alpha}M + A_k^\pm)$ for $M\in \mathcal{S}^n$,
$f_k^\pm(x) = \rho^{-k\alpha} f^\pm(\rho^k x)$ for $x\in \tilde\Omega_{k}^\pm$, and
$g_k(x)  = \rho^{- k (1+\alpha)} \big (g(\rho^k x)-(A_k^+-A_k^-)(\rho^k x)\cdot \nu_k(x) -(B_k^+-B_k^-)\cdot \nu_k (x)\big)$ for $x\in \tilde\Gamma_{k}.$ 

We will show that $v$ satisfies the assumptions of Lemma~\ref{lem:case2}. Indeed, first
note that $\F_k$ is concave, $\F_k \in \mathcal{E}(\lambda, \Lambda)$, and $\F_k(0)=0$.
Moreover,
\begin{equation} \label{eq:assumptionf}
\Big( \fint_{B_r \cap \tilde\Omega_k^\pm} |f_k^\pm(y)|^n\, dy\Big)^{1/n} 
\leq K_{f^\pm} r^{\alpha} \leq K_{f^\pm} r^{\alpha-1}
\end{equation}
for any $r<1$. Hence, $C_{f_k^\pm} = K_{f^\pm}$ and $C_{f_k^+}+C_{f_k^-}\leq \delta_0$.
Let $A_{nn}= (A_k^+)_{nn}- (A_k^-)_{nn}$. By \eqref{eq:cond1} and \eqref{eq:cond2}, for $x \in \tilde \Gamma_k$, we have that
\begin{align*}
g_k(x) & = \rho^{-k(1+\alpha)} \big( g(\rho^k x) - \nabla ' g(0) \cdot \rho^k x' (\nu_k)_n(x)
 - A_{nn} \rho^k \psi_k(x') (\nu_k)_n(x) - g(0)(\nu_k)_n(x) \big).
\end{align*}
Observe that  since $\psi_k(0')=0$, $\nabla' \psi_k(0')=0'$, and  $\nu_k(0)=e_n$, then $g_k(0)=0$ and $\nabla' g_k(0)=0'$. 
For any  $x \in \tilde \Gamma_k$, we get
$$\big|g_k(x) \big|  
\leq [g]_{C^{1,\alpha}(0)} + 2[\nu_n]_{C^{1,\alpha}(0)} + \tfrac{2C_0}{1-2^{-\alpha}} [\psi]_{C^{2,\alpha}(0)} 
\leq \Big(3+\tfrac{2C_0}{1-2^{-\alpha}}\Big)\delta_0$$
where $\|A_{nn}\|\leq \tfrac{2C_0}{1-2^{-\alpha}}$ follows from \eqref{eq:cauchy} since $\rho<1/2$ and
$$
\|A_k^\pm\| \leq  C_0 \sum_{j=0}^{k-1} \rho^{\alpha j} \leq C_0 \frac{1}{1-\rho^\alpha}\leq \frac{C_0}{1-2^{-\alpha}}. 
$$
Furthermore, it can be checked that, for any $x\in \tilde \Gamma_k$,
$$|\nabla ' g_k (x)|  \leq C \big( [g]_{C^{1,\alpha}(0)} + [\psi]_{C^{2,\alpha}(0)} \big) \leq 2 C \delta_0.$$
Choose $\delta_0$ small enough so that $ 2 C \delta_0\leq 1/2$ and $5\big(3+\tfrac{2C_0}{1-2^{-\alpha}}\big)\delta_0 \leq \delta$. 
We conclude that
\begin{equation} \label{eq:assumptiong}
\|g_k\|_{C^1(\tilde \Gamma_k)}\leq  1
\quad\hbox{and}\quad \| g_k - \nabla' g_k(0) \cdot x' - g_k(0) \|_{L^\infty(\tilde\Gamma_k)} \leq \delta/5.
\end{equation}
It remains to show that $h=v^+-v^- \in C^2(\tilde \Gamma_k)$ with $\|h\|_{C^2(\tilde \Gamma_k)}\leq \delta/5$. 
Indeed, since $u\in C(B_1)$, and $P_k^+\neq P_k^-$ on $\tilde \Gamma_k$,  then $v \in C(B_1 \setminus \tilde\Gamma_k)$, and for $x\in \tilde \Gamma_k$, we have
$$
h(x) =\rho^{-k(2+\alpha)}\tfrac{A_{nn}}{2} \psi(\rho^k x')^2 + \rho^{-k(1+\alpha)}(\nabla' g(0)\cdot x') \psi(\rho^k x')
+  \rho^{-k(2+\alpha)} g(0) \psi(\rho^k x')$$
since $x_n=\psi_k(x')=\rho^{-k} \psi(\rho^k x')$. Then $h\in C^2(\tilde \Gamma_k)$, given that $\psi \in C^2(B_1')$. 
Furthermore, 
\begin{align*}
 \|h\|_{L^\infty (\tilde \Gamma_k)} & \leq \big( \tfrac{\|A_{nn}\|}{2} \|\psi\|_{L^\infty(B_1')} +  |\nabla'g(0)| +|g(0)|\big)
 \rho^{-k(2+\alpha)} \|\psi \|_{L^\infty(B_{\rho^k}')}\\
 & \leq  \big( \tfrac{C_0}{1-2^{-\alpha}}+2 \big) [\psi]_{C^{2,\alpha}(0)} \leq C\delta_0.
\end{align*}
To estimate the $C^2$ norm of $h$, we also need to compute $\nabla' h$ and $D_{x'}^2 h$. First, we have
\begin{align*}
\nabla ' h(x) &= \rho^{-k(1+\alpha)} A_{nn} \psi(\rho^k x') \nabla'\psi(\rho^k x') 
+\rho^{-k(1+\alpha)} \nabla' g(0) \psi(\rho^k x') \\
&\quad + \rho^{-k\alpha} (\nabla' g(0) \cdot x') \nabla' \psi(\rho^k x')
+ \rho^{-k(1+\alpha)} g(0) \nabla' \psi(\rho^k x').
\end{align*}
Moreover,
\begin{align*}
 \|\nabla' h\|_{L^\infty (\tilde \Gamma_k)} &\leq  \big( \|A_{nn}\| \|\psi\|_{L^\infty(B_1')} + 2 |\nabla'g(0)| +|g(0)|\big)
 \rho^{-k(1+\alpha)} \|\nabla' \psi \|_{L^\infty(B_{\rho^k}')}\\
 & \leq  \big( \tfrac{2C_0}{1-2^{-\alpha}}+3 \big) [\psi]_{C^{2,\alpha}(0)} \leq C\delta_0.
\end{align*}
Second, we have
\begin{align*}
D_{x'}^2 h(x) &= \rho^{-k\alpha} A_{nn} \big((\nabla'\psi)(\nabla'\psi)^t\big)(\rho^k x')
+ \rho^{-k\alpha} A_{nn} \psi(\rho^k x') D_{x'}^2 \psi(\rho^k x')\\
& \quad + 2 \rho^{-k\alpha} \nabla' g(0) (\nabla' \psi)^t (\rho^k x')
+  \rho^{k(1-\alpha)} (\nabla' g(0)\cdot x') D_{x'}^2 \psi(\rho^k x')\\
& \quad + \rho^{-k\alpha} g(0)  D_{x'}^2 \psi(\rho^k x').
\end{align*}
Therefore,
\begin{align*}
 \|D_{x'}^2 h\|_{L^\infty (\tilde \Gamma_k)} &\leq  
 \big(\|A_{nn}\| \|\nabla'\psi\|_{L^\infty(B_1')}+ 2|\nabla'g(0)|\big) \rho^{-k(1+\alpha)} \|\nabla'\psi\|_{L^\infty(B_{\rho^k}')} \\
 &\qquad +\big( \|A_{nn}\| \|\psi\|_{L^\infty(B_1')} +  |\nabla'g(0)| +|g(0)|\big)
 \rho^{-k\alpha} \|D_{x'}^2 \psi \|_{L^\infty(B_{\rho^k}')}\\
 & \leq  \big( \tfrac{C_0}{1-2^{-\alpha}}+4 \big) [\psi]_{C^{2,\alpha}(0)} \leq C\delta_0.
\end{align*}
Finally, from the previous computations, we get
\begin{align}\label{eq:assumptionh}
\|h\|_{C^2(\tilde \Gamma_k)} = \|h\|_{L^\infty (\tilde \Gamma_k)} + \|\nabla 'h\|_{L^\infty (\tilde \Gamma_k)} + \|D_{x'}^2 h\|_{L^\infty (\tilde \Gamma_k)} \leq C\delta_0 \leq \delta/5,
\end{align}
choosing $\delta_0$ sufficiently small.
Combining \eqref{eq:assumptionpsi}, \eqref{eq:assumptionf}, \eqref{eq:assumptiong}, and \eqref{eq:assumptionh}, $v$ satisfies the assumptions of Lemma~\ref{lem:case2}. Then there exist quadratic polynomials $P^\pm(x)= \tfrac{1}{2}x^t A^\pm x  +  B^\pm\cdot x + c$, where $A^\pm \in \mathcal{S}^n$, $B^\pm\in \Rn$,  $c\in \R$, and
 $ \|A^\pm\|+|B^\pm|+|c|\leq C_0$
  such that
\begin{align} \label{eq:estimatev}
\| v^\pm - P^\pm \|_{L^\infty(\tilde \Omega_k^\pm \cap B_\rho)} \leq  \rho^{2+\alpha}.
\end{align}
Moreover, the coefficients satisfy the following conditions:
\begin{align} \label{eq:cond3}
B_i^+-B_i^-= 0  \quad &  \hbox{for all}~1\leq i \leq n  \\\label{eq:cond4}
A^+_{ij} - A^-_{ij} = 0   \quad &  \hbox{for all}~1\leq i,j \leq n  \quad    \hbox{and} \quad  \F_k (A^\pm) =0,
\end{align}
since $g_k(0)=0$ and $(g_k)_{x_j}(0)=0$. In view of \eqref{eq:rescale2} and \eqref{eq:estimatev}, for any $x\in  \Omega_k^\pm\cap B_\rho$, we have
\begin{align*}
\left|\frac{u^\pm (\rho^k x) - P_k^\pm(\rho^k x)}{\rho^{k(2+\alpha)}} - P^\pm(x) \right| \leq \rho^{2+\alpha},
\end{align*}
or equivalently, if $y = \rho^k x $, then for any $y \in \Omega^\pm_{\rho^{k+1}}$, 
\begin{equation} \label{eq:indstep2}
| u^\pm(y)  - P_k^\pm (y)  - \rho^{k(2+\alpha)} P^\pm (\rho^{-k}y)| \leq \rho^{(k+1)(2+\alpha)}.
\end{equation}
Define the quadratic polynomials at the step $k+1$ as $P_{k+1}^\pm (y) = P_k^\pm (y)  + \rho^{k(2+\alpha)} P^\pm(\rho^{-k}y)$. 
By \eqref{eq:indstep2}, we conclude that
$$
\|u^\pm - P_{k+1}^\pm\|_{L^\infty(\Omega^\pm_{\rho^{k+1}})} \leq \rho^{(k+1)(2+\alpha)}
$$
and, by \eqref{eq:cond1}, \eqref{eq:cond2},  \eqref{eq:cond3}, and  \eqref{eq:cond4}, the coefficients satisfy the conditions in the claim.

\smallskip

\noindent\textbf{Case 2: $g(0)=0$.}
The proof is analogous to the one in Case 1. The main difference is how we control the terms involving the function $\psi$, in order to satisfy the smallness assumption \eqref{eq:assumptionu} given in Lemma~\ref{lem:approx3}.
 From the normalization $(i)$, it follows that
\begin{equation}\label{eq:psi}
 |\psi(x')| = |\psi(x')-\psi(0')-\nabla'\psi(0')\cdot x' |\leq \|D^2_{x'}\psi\|_{L^\infty(B_1')} |x'|^2\leq 2 \delta_0|x'|^2,
 \end{equation}
 and
 \begin{equation}\label{eq:psi2}
 |\nabla'\psi(x')| = |\nabla ' \psi(x')-\nabla'\psi(0') |\leq \|D^2_{x'}\psi\|_{L^\infty(B_1')} |x'|\leq 2 \delta_0|x'|.
 \end{equation}
 Hence, \eqref{eq:assumptionpsi} and \eqref{eq:assumptiong} follow similarly using \eqref{eq:psi}, \eqref{eq:psi2}, and \eqref{eq:hessianpsi}. It remains to estimate $\|h\|_{C^2(\tilde \Gamma_k)}$. The only problematic term in the expression of $h$ given above is the last term: 
 $$\rho^{-k(2+\alpha)} g(0) \psi(\rho^k x').$$ 
  Indeed, by \eqref{eq:psi} we have $|\rho^{-k(2+\alpha)} g(0) \psi(\rho^k x')|\leq 2\delta_0 \rho^{-k\alpha} |g(0)|$, which is large for $\rho$ small, unless $g(0)=0$. The same argument applies for the last term in $\nabla' h$ and $D_{x'}^2 h$. Hence, assuming that $g(0)=0$, we can make the $C^2$ norm of $h$ small enough, proceeding as in Case~1, using \eqref{eq:psi}, \eqref{eq:psi2}, and \eqref{eq:hessianpsi}. This concludes the proof of the theorem.
 \end{proof}

%%%%%%%%%%%%%%%%%%%%%%%%%%%%%%%%%%%%%%%%%%%%%

%%%%%%%%%%%%%%%%%%%%%%%%%%%%%%%%%%%%%%%%%%%%%


\begin{thebibliography}{10}
%%%%%%%%%%%%%%%%%%%%%%%%%%%%%%%%%%%%%%%%%%%%%

\bibitem{Borsuk} M.~V.~Borsuk,
\textit{Transmission Problems for Elliptic Second-Order Equations in Non-Smooth Domains},
Frontiers in Mathematics, Birkh\"auser/Springer Basel AG, Basel (2010).

\bibitem{CC} L.~A.~Caffarelli and X.~Cabr\'{e},
\textit{Fully Nonlinear Elliptic Equations},
American Mathematical Society Colloquium Publications \textbf{43},.
American Mathematical Society, Providence, RI, 1995.

\bibitem{CSCS} L.~A.~Caffarelli, M.~Soria-Carro, and P.~R.~Stinga,
{Regularity for $C^{1,\alpha}$ interface transmission problems},
\textit{Arch. Ration. Mech. Anal.}
\textbf{240} (2021), 265--294.

\bibitem{Campanato} S.~Campanato,
{Sul problema di M. Picone relativo all'equilibrio di un corpo elastico incastrato},
\textit{Ricerche Mat.}
\textbf{6} (1957), 125--149.

\bibitem{Citti-Ferrari} G.~Citti and F.~Ferrari,
{A sharp regularity result of solutions of a transmission problem},
\textit{Proc. Amer. Math. Soc.}
\textbf{140} (2012), 615--620.

\bibitem{CIL} M.~G.~Crandall, H.~Ishii and P.~L.~Lions, 
{User's guide to viscosity solutions of second order partial differential equations},
\textit{Bull. Amer. Math. Soc. (N.S.)}
\textbf{27} (1992), 1--67. 

\bibitem{DSFS} D.~De Silva, F.~Ferrari, and S.~Salsa,
{Regularity of transmission problems for uniformly elliptic fully nonlinear equations}, 
Proceedings of the International Conference ``Two nonlinear days in Urbino 2017'',
\textit{Electron. J. Differ. Equ. Conf.}
\textbf{25} (2018), 55--63.

\bibitem{DeSilva-Ferrari-Salsa} D.~De Silva, F.~Ferrari and S.~Salsa,
{Two-phase free boundary problems: from existence to smoothness},
\textit{Adv. Nonlinear Stud.}
\textbf{17} (2017), 369--385.
 
\bibitem{Kriventsov} D.~Kriventsov,
{Regularity for a local-nonlocal transmission problem},
\textit{Arch. Ration. Mech. Anal.}
\textbf{217} (2015), 1103--1195.
 
 \bibitem{Ladyzhenskaya-Uraltseva} O.~A.~Ladyzhenskaya and N.~N.~Ural'tseva,
\textit{Linear and Quasilinear Elliptic Equations},
Academic Press, New York-London, 1968.

\bibitem{Li-Nirenberg} Y.~Li and L.~Nirenberg,
{Estimates for elliptic systems from composite material. Dedicated to the memory of J\"urgen K. Moser},
\textit{Comm. Pure Appl. Math.}
\textbf{56} (2003), 892--925.

\bibitem{Li-Vogelius} Y.~Li and M.~Vogelius,
{Gradient estimates for solutions to divergence form elliptic equations with discontinuous coefficients},
\textit{Arch. Ration. Mech. Anal.}
\textbf{153} (2000), 91--151.

\bibitem{LiZ} D.~Li and K.~Zhang,
{Regularity for fully nonlinear elliptic equations with oblique boundary conditions},
\textit{Arch. Ration. Mech. Anal.}
\textbf{228} (2018), 923--967.

\bibitem{LZ2} Y.~Lian and K.~Zhang,
{Boundary pointwise $C^{1,\alpha}$ and $C^{2,\alpha}$ regularity for fully nonlinear elliptic equations},
\textit{J. Differential Equations}
\textbf{269} (2020), 1172--1191.

\bibitem{LXZ} Y.~Lian, W.~Xu and K.~Zhang,
{Boundary Lipschitz regularity and the Hopf lemma on Reifenberg domains for fully nonlinear elliptic equations},
\textit{Manuscripta Math.}
\textbf{166} (2021), 343--357.

\bibitem{Lions}  J.~L.~Lions,
{Contributions \`a un probl\`eme de M. M. Picone}, 
\textit{Ann. Mat. Pura Appl. (4)}
\textbf{41} (1956), 201--219.


\bibitem{MS} E.~Milakis and L.~E.~Silvestre,
{Regularity for fully nonlinear elliptic equations with Neumann boundary data},
\textit{Comm. Partial Differential Equations}
\textbf{31} (2006), 1227--1252.

\bibitem{Oleinik} O.~A.~Oleinik,
{Boundary value problems for linear elliptic and parabolic equations with discontinuous coefficients},
\textit{Amer. Math. Soc. Transl.}
\textbf{42} (1964), 175--194.

\bibitem{Picone} M.~Picone, 
{Sur un probl\`eme nouveau pour l'\'equation lin\'{e}aire aux d\'eriv\'{e}es partielles de la th\'{e}orie mathematique classique de l'\'elasticit\'e}, 
\textit{Colloque sur les \'{e}quations aux d\'{e}riv\'{e}es partielles},
Bruxelles, May 1954.

\bibitem{Pimentel-Swiech} E.~Pimentel and A.~\'Swi\c{e}ch,
{Existence of solutions to a fully nonlinear free transmission problem},
\textit{J. Differential Equations}
\textbf{320} (2022), 49--63.

\bibitem{Schechter} M.~Schechter,
{A generalization of the problem of transmission},
\textit{Ann. Scuola Norm. Sup. Pisa Cl. Sci. (3)}
\textbf{14} (1960), 207--236.

\bibitem{Silvestre-Sirakov} L.~Silvestre and  B.~Sirakov,
{Boundary regularity for viscosity solutions of fully nonlinear elliptic equations},
\textit{Comm. Partial Differential Equations}
\textbf{39} (2014), 1694--1717.

\bibitem{Stampacchia} G.~Stampacchia,
{Su un problema relativo alle equazioni di tipo ellittico del secondo ordine}, 
\textit{Ricerche Mat.}
\textbf{5} (1956), 3--24.

%%%%%%%%%%%%%%%%%%%%%%%%%%%%%%%%%%%%%%%%%%%%%
\end{thebibliography}
\end{document}